\documentclass[reqno,11pt,a4paper,final]{amsart}
\usepackage[a4paper,left=35mm,right=35mm,top=30mm,bottom=30mm,marginpar=25mm]{geometry} 
\usepackage{amsmath}
\usepackage{amssymb}
\usepackage{amsthm}
\usepackage{amscd}
\usepackage[ansinew]{inputenc}
\usepackage{bbm}
\usepackage{color}
\usepackage[final]{graphicx}
\usepackage{hyperref}
\usepackage{calc}
\usepackage{mathptmx}
\usepackage{bm}
\usepackage{enumerate}
\usepackage[shortlabels]{enumitem}

\graphicspath{{Pictures/}}

\numberwithin{equation}{section}

\newtheoremstyle{thmlemcorr}{10pt}{10pt}{\itshape}{}{\bfseries}{.}{10pt}{{\thmname{#1}\thmnumber{ #2}\thmnote{ (#3)}}}
\newtheoremstyle{thmlemcorr*}{10pt}{10pt}{\itshape}{}{\bfseries}{.}\newline{{\thmname{#1}\thmnumber{ #2}\thmnote{ (#3)}}}
\newtheoremstyle{remexample}{10pt}{10pt}{}{}{\bfseries}{.}{10pt}{{\thmname{#1}\thmnumber{ #2}\thmnote{ (#3)}}}

\theoremstyle{thmlemcorr}
\newtheorem{theorem}{Theorem}
\numberwithin{theorem}{section}
\newtheorem{lemma}[theorem]{Lemma}
\newtheorem{corollary}[theorem]{Corollary}
\newtheorem{proposition}[theorem]{Proposition}

\theoremstyle{thmlemcorr*}
\newtheorem{theorem*}{Theorem}
\newtheorem{lemma*}[theorem]{Lemma}
\newtheorem{corollary*}[theorem]{Corollary}
\newtheorem{proposition*}[theorem]{Proposition}
\newtheorem{problem*}[theorem]{Problem}
\newtheorem{conjecture*}[theorem]{Conjecture}
\newtheorem{definition*}[theorem]{Definition}

\theoremstyle{remexample}
\newtheorem{remark}[theorem]{Remark}
\newtheorem{example}[theorem]{Example}

\newcommand{\Crm}{\mathrm{C}}

\newcommand{\Lrm}{\mathrm{L}}

\newcommand{\Wrm}{\mathrm{W}}

\newcommand{\Ecal}{\mathcal{E}}

\newcommand{\Mcal}{\mathcal{M}}

\newcommand{\Rcal}{\mathcal{R}}
\newcommand{\Scal}{\mathcal{S}}

\newcommand{\Wcal}{\mathcal{W}}

\newcommand{\ffrak}{\mathfrak{f}}

\newcommand{\Ibb}{\mathbb{I}}
\newcommand{\Jbb}{\mathbb{J}}

\newcommand{\Nbb}{\mathbb{N}}

\newcommand{\Rbb}{\mathbb{R}}

\newcommand{\Zbb}{\mathbb{Z}}

\DeclareMathOperator{\Sgn}{Sgn}

\newcommand{\ee}{\mathrm{e}}

\newcommand{\set}[2]{\left\{\, #1 \ \ \textup{\textbf{:}}\ \ #2 \,\right\}}

\newcommand{\setb}[2]{\bigl\{\, #1 \ \ \textup{\textbf{:}}\ \ #2 \,\bigr\}}
\newcommand{\setB}[2]{\Bigl\{\, #1 \ \ \textup{\textbf{:}}\ \ #2 \,\Bigr\}}
\newcommand{\setBB}[2]{\biggl\{\, #1 \ \ \textup{\textbf{:}}\ \ #2 \,\biggr\}}

\newcommand{\norm}[1]{\|#1\|}

\newcommand{\normb}[1]{\bigl\|#1\bigr\|}

\newcommand{\normBB}[1]{\biggl\|#1\biggr\|}

\newcommand{\abs}[1]{|#1|}

\newcommand{\absb}[1]{\bigl|#1\bigr|}

\newcommand{\absBB}[1]{\biggl|#1\biggr|}

\newcommand{\dpr}[1]{\langle #1 \rangle}

\newcommand{\dprb}[1]{\bigl\langle #1 \bigr\rangle}

\newcommand{\dprBB}[1]{\biggl\langle #1 \biggr\rangle}

\newcommand{\cl}[1]{\overline{#1}}
\newcommand{\di}{\mathrm{d}}
\newcommand{\dd}{\;\mathrm{d}}
\newcommand{\DD}{\mathrm{D}}
\newcommand{\N}{\mathbb{N}}
\newcommand{\R}{\mathbb{R}}

\newcommand{\loc}{\mathrm{loc}}

\newcommand{\ONE}{\mathbbm{1}}

\newcommand{\toweak}{\rightharpoonup}
\newcommand{\toweakstar}{\overset{*}\rightharpoonup}

\newcommand{\toup}{\uparrow}
\newcommand{\todown}{\downarrow}

\newcommand{\BigO}{\mathrm{\textup{O}}}

\newcommand{\sbullet}{\begin{picture}(1,1)(-0.5,-2.5)\circle*{2}\end{picture}}
\newcommand{\frarg}{\,\sbullet\,}
\newcommand{\BV}{\mathrm{BV}}

\newcommand{\eps}{\epsilon}

\DeclareMathOperator{\Var}{Var}

\DeclareMathOperator{\Diss}{Diss}

\newcommand{\term}[1]{\textit{#1}}

\newcommand{\proofstep}[1]{\medskip\textit{#1}}

\newcommand{\wlj}{{\tilde{u}_{\lambda_j}}}
\newcommand{\w}{{\tilde{u}}}

\newcounter{assumption}
\makeatletter
\newcommand{\nextas}[1]{%
  ~\refstepcounter{assumption}%
   \protected@write \@auxout{}{\string\newlabel{#1}{{(A\theassumption)}{\thepage}{(A\theassumption)}{#1}{}}}%
   \hypertarget{#1}{(A\theassumption)}%
}
\makeatother

\def\Xint#1{\mathchoice 
{\XXint\displaystyle\textstyle{#1}}%
{\XXint\textstyle\scriptstyle{#1}}%
{\XXint\scriptstyle\scriptscriptstyle{#1}}%
{\XXint\scriptscriptstyle\scriptscriptstyle{#1}}%
\!\int} 
\def\XXint#1#2#3{{\setbox0=\hbox{$#1{#2#3}{\int}$} 
\vcenter{\hbox{$#2#3$}}\kern-.5\wd0}} 
\def\dashint{\,\Xint-}

\renewcommand{\eps}{\varepsilon}
\renewcommand{\epsilon}{\varepsilon}
\renewcommand{\phi}{\varphi}
\renewcommand{\hat}{\widehat}

\begin{document}


\title[Two-speed solutions]{Two-speed solutions to non-convex rate-independent systems}

\author{Filip Rindler}
\address{\textit{F.~Rindler:} Mathematics Institute, University of Warwick, Coventry CV4 7AL, UK, and The Alan Turing Institute, British Library, 96 Euston Road, London NW1 2DB London, UK.}
\email{F.Rindler@warwick.ac.uk}

\author{Sebastian Schwarzacher}
\address{\textit{S.~Schwarzacher:} Institute for Applied Mathematics, University of Bonn, Endenicher Allee 60, D-53115 Bonn, Germany, and Katedra matematick\'{e} anal\'{y}zy, Charles University Prague, Sokolovsk\'{a} 83, 186 75 Praha 8, Czech Republic.}
\email{schwarz@karlin.mff.cuni.cz}

\author{Juan J.\ L.\ Vel\'{a}zquez}
\address{\textit{J.\ J.\ L.\ Vel\'{a}zquez:} Institute for Applied Mathematics, University of Bonn, Endenicher Allee 60, D-53115 Bonn, Germany.}
\email{velazquez@iam.uni-bonn.de}


\hypersetup{
  pdfauthor = {Filip Rindler, Sebastian Schwarzacher, Juan J. L. Velazquez},
  pdftitle = {Two-speed solutions to non-convex rate-independent systems},
  pdfsubject = {MSC (2010): 49J40 (primary); 47J20, 47J40, 74H30},
  pdfkeywords = {Rate-independent systems, quasistatic evolution, nonconvex functional, two-speed solution, viscosity approximation, slow-loading limit}
}


\maketitle
\thispagestyle{empty}

\begin{abstract}

We consider evolutionary PDE inclusions of the form
\[
  -\lambda \dot{u}_\lambda + \Delta u - \DD W_0(u) + f \ni \partial \Rcal_1(\dot{u})
  \qquad \text{in $(0,T) \times \Omega$,}
\]
where $\Rcal_1$ is a positively $1$-homogeneous rate-independent dissipation potential and $W_0$ is a (generally) non-convex energy density. This work constructs solutions to the above system in the slow-loading limit $\lambda \todown 0$. Our solutions have more regularity both in space and time than those that have been obtained with other approaches. On the ``slow'' time scale we see strong solutions to a purely rate-independent evolution. Over the jumps, we obtain a detailed description of the behavior of the solution and we resolve the jump transients at a ``fast'' time scale, where the original rate-dependent evolution is still visible. Crucially, every jump transient splits into a (possibly countable) number of rate-dependent evolutions, for which the energy dissipation can be explicitly computed. This in particular yields a global energy equality for the whole evolution process. It also turns out that there is a canonical slow time scale that avoids intermediate-scale effects, where movement occurs in a mixed rate-dependent / rate-independent way. In this way, we obtain precise information on the impact of the approximation on the constructed solution. Our results are illustrated by examples, which elucidate the effects that can occur.

\vspace{4pt}



\vspace{4pt}

\noindent\textsc{Date:} \today{}.
\end{abstract}

\setcounter{tocdepth}{1}
\tableofcontents

\section{Introduction}

Consider the prototypical PDE system (to be interpreted in a suitable sense) 
\begin{equation} \label{eq:PDE_ex}
  \lambda  \dot{u}_\lambda(t) + \frac{\dot{u}_\lambda(t)}{\abs{\dot{u}_\lambda(t)}} - \Delta u_\lambda(t) + \DD W_0(u_\lambda(t)) = f(t), \qquad u_\lambda \colon [0,T] \times \Omega \to \R^m,
\end{equation}
on a time interval $[0,T]$ ($T > 0$) and a bounded $\Crm^{1,1}$-domain $\Omega \subset \R^d$. This PDE system combines rate-independent dissipation (e.g.\ dry friction) and inertial (parabolic) dissipation. When the energy potential $W_0$ is non-convex (e.g.\ a double-well potential), as is often the case in applications, then the behavior of~\eqref{eq:PDE_ex} may display rapid phase transitions, where $u(t)$ moves from one well of $W_0$ to another with speed $\abs{\dot{u}_\lambda} \sim 1/\lambda$. 

It is often desirable to take the \emph{slow-loading limit} of~\eqref{eq:PDE_ex} as $\lambda \todown 0$. This corresponds to the assumption that the rate-dependent dissipative effects only act with infinitesimal speed. In the limit one may conjecture that $u$ follows the degenerate equation
\begin{equation} \label{eq:PDE_0}
 \frac{\dot{u}(t)}{\abs{\dot{u}(t)}} - \Delta u(t) + \DD W_0(u(t)) = f(t), \qquad u \colon [0,T] \times \Omega \to \R^m,
\end{equation}
in a suitable (weak) sense. However, since $\abs{\dot{u}_\lambda} \sim 1/\lambda \to \infty$ over the rapid transitions, we must expect that $u$ in general has jumps in time. In particular, for the total energy process $E(t)$, defined as
\[
  E(t) := \int_\Omega \frac{\abs{\nabla u(t,x)}^2}{2} + W_0(u(t,x)) - f(t,x) \cdot u(t,x) \dd x,
\]
at a jump point $t_0 \in (0,T)$ of $u(t)$, the energy difference
\[
  \delta E(t_0) := E(t_0+) - E(t_0-) = \lim_{t \todown t_0} E(t) - \lim_{t \toup t_0} E(t)
\]
should still depend on the original dynamics from~\eqref{eq:PDE_ex}. In fact, it turns out that the jump path connecting the two jump endpoints is not necessarily straight and thus the total energy dissipation cannot simply be measured as a total variation (with respect to a suitable dissipation distance), as is for instance the case in the by now classical Mielke--Theil energetic solutions~\cite{MielkeTheilLevitas02,MielkeTheil04}. Instead, it turns out that the dissipation of energy over a jump depends on both the path and the speed of the jump transient between $u(t_0-)$ and $u(t_0+)$. This jump transient, which progresses along a ``fast time'' (relative to the ``slow time'' $t$), cannot be expected to be independent of the approximation. Indeed, as we will demonstrate below, some jumps are what we call {\em rate-dependent}, which means that they follow an evolution of the type~\eqref{eq:PDE_ex} with $\lambda = 1$. Thus, from this point of view, the above formulation~\eqref{eq:PDE_0} by itself is \emph{under-specified}.

The rate-independent system~\eqref{eq:PDE_0} above has been studied in great detail in the work of Mielke and collaborators, starting from~\cite{MielkeTheilLevitas02,MielkeTheil04}, and recently culminating in the book~\cite{MielkeRoubicek15book}, to which we refer for motivation, applications and history. In particular, we mention the theory of ``balanced viscosity'' solutions by Mielke--Rossi--Savar\'{e}; see the main works~\cite{MielkeRossiSavare09,MielkeRossiSavare12, MielkeRossiSavare16}. There, the authors develop a powerful framework to construct solutions which satisfy a conservation-of-energy formula. This formula includes a contribution to the energy dissipation originating from the rate-dependent evolution over the jumps, which is computed by means of a variational principle.

While the theory of balanced viscosity solutions encompasses a large number of non-convex and non-smooth functionals, all constructed solutions are of (potentially) \emph{low regularity}. Moreover, unlike in classical PDE theory, it seems to be impossible to establish regularity of such solutions a-posteriori. There are essentially three reasons for this: First, once the solution processes are constructed, all regularity is already ``lost'' and the stability and energy conditions that characterize balanced viscosity solutions are too weak to derive it. Second, because of the potential non-uniqueness of balanced viscosity solutions, there may be no reason to believe that all balanced viscosity solutions are in fact regular (as fast oscillations between several possible solutions may occur). Third, balanced viscosity solutions can be constructed in very general situations, including those involving non-smooth functionals, where no further regularity theory may exist.

On the other hand,~\cite{RindlerSchwarzacherSuli17} developed a solution theory of strong solutions and derived essentially optimal regularity estimates in the case of ``functionally convex'' $W_0$. For fully non-convex energies, the validity of the regularity results of~\cite{RindlerSchwarzacherSuli17} breaks down at jump points and a new analysis is necessary. We also refer to~\cite{ThomasMielke10,Knees10,MielkePaoliPetrovStefanelli10,MielkeZelik14,Minh14} for other regularity results in the theory of rate-independent systems.

In the present work, we impose stronger assumptions on the energy functionals (still allowing for strong non-convexity) than the ones required for the theory of balanced viscosity solutions of Mielke--Rossi--Savar\'{e}~\cite{MielkeRossiSavare09,MielkeRossiSavare12, MielkeRossiSavare16}. As a tradeoff, however, our solutions have higher regularity both in space and time; see Theorem~\ref{thm:main_ex}. For instance, on intervals without jumps our solutions are \emph{strong} in the sense of a variational inequality (which is more information than the energetic balance defining balanced viscosity solutions). Furthermore, we get higher integrability and differentiability properties of the solution in space and time.

Like in the theory of balanced viscosity solutions, we explicitly resolve jumps into a (possibly countable) number of fast jump parts over which the dissipation is rate-dependent as in the original system~\eqref{eq:PDE_ex} with $\lambda = 1$. This effect originates from the possibility that a developing jump in~\eqref{eq:PDE_ex} (i.e.\ with slope $1/\lambda$) may in fact consist of several pieces that are separated on a slower scale than $\lambda$. See Figure~\ref{fig:twoscales_jump} for an illustration. Thus, the jump resolves to a number of fast evolutions, which are, however, disconnected from each other with respect to the fast time.

We also establish convergence of inertial (parabolic) approximations to a (regular) limit solution; see Theorem~\ref{thm:main_approx}.

\begin{figure}[t]
\includegraphics[width=250pt]{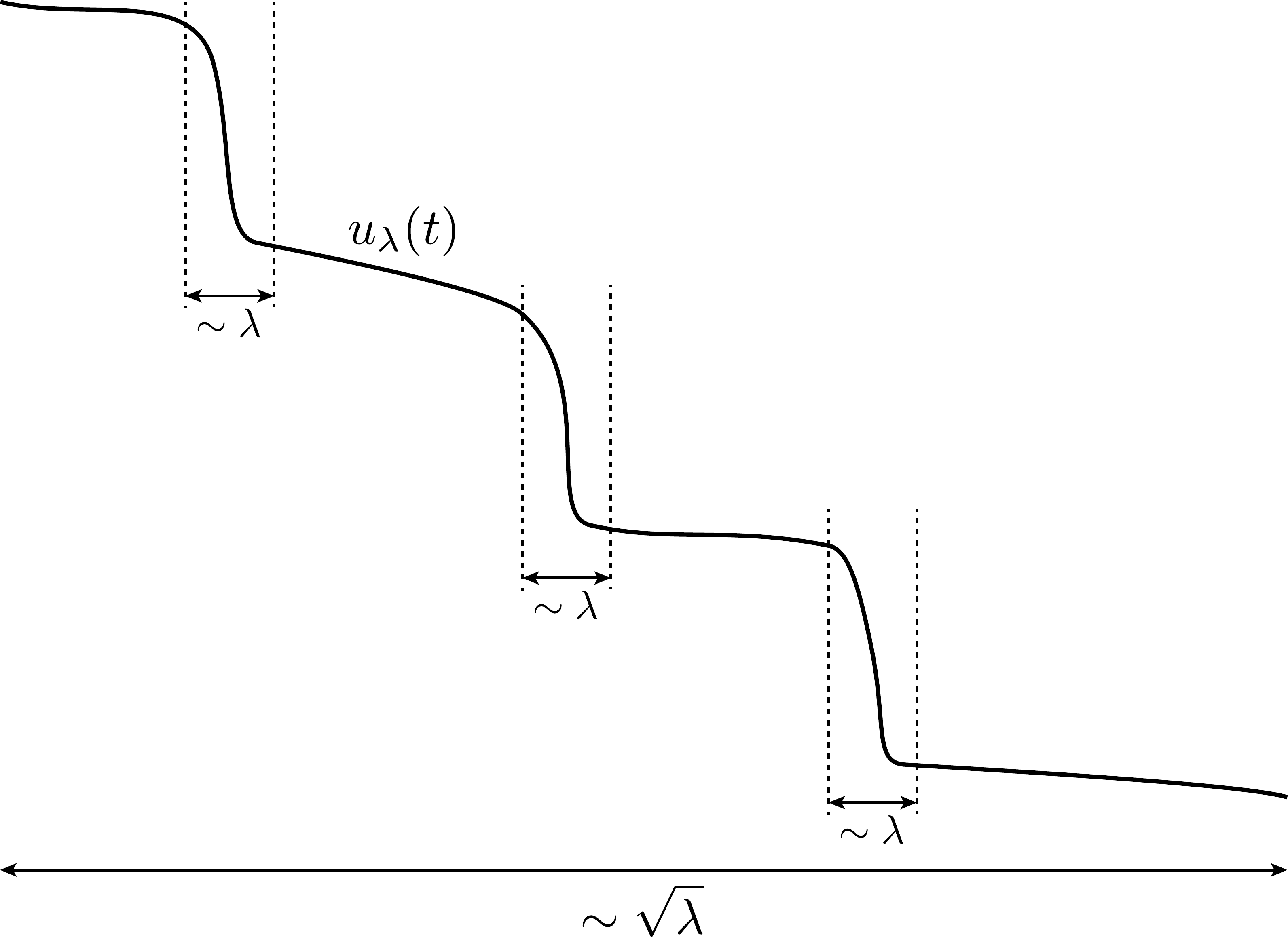}
\caption{A jump with two scales}
\label{fig:twoscales_jump}
\end{figure}

As it turns out, the time scale $t$ above is in fact not very well-suited to describing the rate-independent evolution with (possibly) rate-dependent jump transients. Indeed, in the ``naive'' time scale $t$, there may occur pathological paths of a rate-independent (sliding) nature. Namely, these ``weak shocks'' happen with a speed that is slower than the fast speed of the jump transients, yet faster than the speed $t$ (see also Example~\ref{ex:jumplength1}). Instead, we explicitly construct a canonical slow time $s$, in which the total dissipation is more explicit (a similar idea is in fact already present in~\cite{MielkeRossiSavare09,MielkeRossiSavare12, MielkeRossiSavare16}). In fact, our main existence result, Theorem~\ref{thm:main_ex} is formulated in this new, natural time scale. In this time scale, all rate-independent jumps have been removed; consequently, the solution's energy dissipation will be purely rate-independent except for the jumps on which the energy is dissipated in a (purely) \emph{rate-dependent} manner. For the sake of completeness, we also give a (necessarily less satisfying) formulation of the energy dissipation in the original time scale $t$, see Corollary~\ref{cor:main_ex}.

We mention that the present work can also be understood as progress towards the clarification of which parts of a rate-independent evolution is {\em independent of the approximation} and which part {\em depends on the approximation}. This is achieved by tracing down precisely the impact on the energy from the approximation.

On the technical side we mention the key estimate in Lemma~\ref{lem:key1}. This crucial result entails that the rate-dependent dissipations associated to the processes $u_{\lambda}$ converge to the rate-independent dissipation of the limit process $u$. In particular, there are no contributions to the energy dissipation due to fast and small oscillations of $u_{\lambda}$  away from the jumps (see Proposition~\ref{prop:E_contpoint}).

This paper is organized as follows: We first describe the setup and the main results of this work in Section~\ref{sc:mainresult}, which are then illustrated with examples. After looking at what can be gained from an energy inequality together with stability alone in Section~\ref{sc:energy_ineq}, we turn to the construction of solutions of~\eqref{eq:PDE_ex} for $\lambda > 0$ in Section~\ref{sc:RDevol}. Then, finally, in Section~\ref{sc:twospeed} we construct our full solution process and jump resolutions. It is worth mentioning that the analysis of Section~\ref{sc:energy_ineq} does not depend on the approximation and hence might be of independent interest.

\subsection*{Acknowledgements}
This project has received funding from the European Research Council (ERC) under the European Union's Horizon 2020 research and innovation programme, grant agreement No 757254 (SINGULARITY). F.~R.\ also acknowledges the support from an EPSRC Research Fellowship on Singularities in Nonlinear PDEs (EP/L018934/1). S.S. and J.J.L.V. acknowledge support through the CRC 1060 (The Mathematics of Emergent Effects) of the University of Bonn that is funded through the German Science Foundation (DFG). S.S. further thanks  for the research support PRIMUS/19/SCI/01 and UNCE/SCI/023 of Charles University and the program GJ17-01694Y of the Czech national grant agency.

\section{Setup and main results}
\label{sc:mainresult}

\subsection{Problem formulation}

Consider for $\lambda > 0$ the following PDE system on a time interval $[0,T]$ ($T>0$) and a bounded $\Crm^{1,1}$-domain $\Omega \subset \R^d$, $d=2$ or $3$, for a map $u \colon [0,T] \times \Omega \to \R^m$:
\begin{align} \label{eq:PDE_lambda}
  \left\{
  \begin{aligned}
    -\lambda \dot{u}_\lambda(t) +  \Delta u_\lambda(t) - \DD W_0(u_\lambda(t))+f(t) &\in  \partial \Rcal_1(\dot{u}_\lambda(t))
      &&\text{in $(0,T) \times \Omega$,}\\
    u_\lambda(t)|_{\partial \Omega} &= 0  &&\text{for $t \in (0,T]$,} \\
    \qquad u_\lambda(0) &= u_0  &&\text{in $\Omega$}.
  \end{aligned} \right.
\end{align}
Here, $\Rcal_1 \colon \Lrm^1(\Omega;\R^m) \to [0,\infty)$ is the \term{rate-independent dissipation potential}, which is assumed to be convex and positively $1$-homogeneous (in the introduction we used $\Rcal_1(v)=\int_\Omega \abs{v(x)}\dd x$); $\partial \Rcal_1$ is its subdifferential; $W_0\colon \R^m \to [0,\infty)$ is the \term{energy functional}, which satisfies natural \emph{coercivity} and \emph{growth conditions} that will be specified below; $f \colon [0,T]\times \Omega\to \R^m$ is the \term{external loading (force)}; and $u_0:\Omega \to \R^m$ is the \term{initial value}. It is important to note that $W_0$ may be \emph{non-convex} and could for instance have the form of a double well.

Call
\[
  u_\lambda \in \Lrm^\infty(0,T;(\Wrm^{1,2}_0 \cap \Wrm^{2,2})(\Omega;\R^m))
  \quad\text{with}\quad
  \dot{u}_\lambda \in \Lrm^2((0,T) \times \Omega;\R^m)
\]
a \term{strong solution} to~\eqref{eq:PDE_lambda} if
\[
  -\lambda \dot{u}_\lambda(t) + \Delta u(t) - \DD W_0(u(t)) + f(t) \in \partial \Rcal_1(\dot{u}_\lambda(t))  \qquad\text{for a.e.\ $t \in (0,T)$,}
\]
i.e.,
\begin{equation} \label{eq:Varlam}
  \left\{\begin{aligned}
    &\Rcal_1(\dot{u}_\lambda(t)) + \dprb{- \lambda \dot{u}_\lambda(t) + \Delta u_\lambda(t) - \DD W_0(u_\lambda(t)) + f(t), \xi(t) - \dot{u}_\lambda(t)}  \leq \Rcal_1(\xi(t)) \\
    &\text{for a.e.\ $t\in (0,T)$ and all $\xi \in \Lrm^1(0,T;\Wrm^{1,2}_0(\Omega;\R^m))$,}
  \end{aligned}\right.
\end{equation}
and if
\[
  u_\lambda(0) = u_0  \quad\text{in $\Lrm^2(\Omega;\R^m)$,}
\]
noting that $u$ has regularity $\Wrm^{1,2}(0,T;\Lrm^2(\Omega;\R^m))$, hence the trace on the left time end-point is well-defined.

We remark that above the Laplacian $\Delta$ could be replaced by a (possibly time-dependent)  second-order strongly elliptic linear PDE operator in the spatial variables. For the sake of clarity, in the following we only consider the case of the Laplacian.

Formally setting $\lambda = 0$ in~\eqref{eq:PDE_lambda}, we obtain the associated \term{rate-independent system} (without initial condition)
\begin{align} \label{eq:PDE_RI}
  \left\{
  \begin{aligned}
      \Delta u(t) - \DD W_0(u(t)) + f(t) &\in \partial \Rcal_1(\dot{u}(t))
      &&\text{in $(0,T) \times \Omega$,}\\
    u(t)|_{\partial \Omega} &= 0  &&\text{for $t \in (0,T)$.}
  \end{aligned} \right.
\end{align}
We call a map
\[
  u \in \Lrm^\infty_{\loc}(0,T;(\Wrm^{1,2}_0 \cap \Lrm^\infty)(\Omega;\R^m))
  \quad\text{with}\quad
  \dot{u} \in \Lrm^1(0,T;\Wrm^{1,2}_0(\Omega;\R^m))
\]
a \term{strong solution} to~\eqref{eq:PDE_RI} provided that
\[
  \Delta u - \DD W_0(u) + f \in \Lrm^\infty(0,T;\Lrm^2(\Omega;\R^m))
\]
and
\[
  \Delta u(t) - \DD W_0(u(t)) + f(t) \in \partial \Rcal_1(\dot{u}(t))  \qquad\text{for a.e.\ $t \in (0,T)$.}
\]
By recalling the definition of the subdifferential, this differential inclusion means that
\begin{equation}\label{eq:strongsol_ineq}
  \left\{\begin{aligned}
    &\Rcal_1(\dot{u}(t)) + \dprb{\Delta u(t) - \DD W_0(u(t)) + f(t), \xi(t) - \dot{u}(t)} \leq \Rcal_1(\xi(t))\\
    &\text{for a.e.\ $t \in (0,T)$ and all $\xi \in \Lrm^1(0,T;\Wrm^{1,2}_0(\Omega;\R^m))$.}
  \end{aligned}\right.
\end{equation}
We refer to~\cite[Remark~1.1]{RindlerSchwarzacherSuli17} for some comments on the regularity classes in which we look for solutions. If~\eqref{eq:strongsol_ineq} is satisfied at $t \in (0,T)$, we say that $u$ is a \term{strong solution at $t$} to~\eqref{eq:PDE_RI}. Note that above we required only \emph{local} $\Lrm^\infty$-regularity in time since this is all we can expect if there is a jump at the initial time.

Observe that due to the rate-independent character of~\eqref{eq:strongsol_ineq}, we find for every Lipschitz function $\phi \colon [0,T]\to [0,S]$ with $\phi(0) = 0$ and $\phi(T) = S$ that the function $\w(s):=u(\phi(s))$ satisfies~\eqref{eq:PDE_RI} on $[0,S]$. Indeed if $u$ is a strong solution to~\eqref{eq:PDE_RI}, then $\w$ satisfies
 \[
  \Rcal_1(\partial_s \w(s)) + \dprb{\Delta \w(s) - \DD W_0(\w(s)) + \tilde{f}(s), \xi(s) - \partial_s \w(s)} \leq \Rcal_1(\xi(t)),
\]
where $\tilde{f}(s):=f\circ \phi(s)$.

\subsection{Assumptions} \label{sc:assume}

Unless stated otherwise, the following conditions will be assumed to hold in the rest of the paper:

\begin{enumerate}[({A}1)]
  \item \label{as:domain} Let $d\in\{2,3\}$ and let $\Omega \subset \R^d$ be open, bounded and with boundary of class $\Crm^{1,1}$.
  \item \label{as:R} The rate-independent dissipation (pseudo-)potential $\Rcal_1 \colon \Lrm^1(\Omega;\R^m) \to \R $ is given as
\[\qquad
  \Rcal_1(v) := \int_\Omega R_1(v(x)) \dd x,  \qquad v \in \Lrm^1(\Omega;\R^m),
\]
with $R_1 \colon \R^m \to [0,\infty)$ convex and positively $1$-homogeneous, i.e.\ $R_1(\alpha z) = \alpha R_1(z)$ for any $\alpha \geq 0$ and $z \in \R^m$. Moreover, we assume that $R_1(z)>0$ for all $z \in \R^m\setminus \{0\}$.
  \item \label{as:W} The energy functional $\Wcal_0 \colon \Lrm^q(\Omega;\R^m) \to [0,\infty]$, where $q \in (1,\infty)$, has the form
\[\qquad
  \Wcal_0(u) := \int_\Omega W_0(u(x)) \dd x
\]
with $W_0 \in \Crm^2(\R^m;[0,\infty))$ satisfying the following assumptions for suitable constants $C,\mu > 0$, and all
$v, w \in \R^m$:
\begin{align*}
  \qquad
  C^{-1}\abs{v}^q - C &\leq W_0(v) \leq C(1+\abs{v}^q);  \\
  \abs{\DD W_0(v)}  &\leq C(1+\abs{v}^{q-1});             \\
 -\mu \abs{v-w}^2 &\leq (\DD W_0(v)-\DD W_0(w))\cdot(v-w);  \\
 -\mu \abs{w}^2&\leq \DD^2 W_0(v)[w,w].
\end{align*}
The last assumption means that the non-convexity is not too degenerate (this still of course allows multi-well potentials).

  \item \label{as:f} The force term satisfies $f \in \Wrm^{1,\infty}(0,T;\Lrm^\infty(\Omega;\R^m))$.

  \item \label{as:u0} The initial value satisfies $u_0 \in (\Wrm^{1,2}_0 \cap \Lrm^q)(\Omega;\R^m)$.
\end{enumerate}

We note that the assumed positive convexity and $1$-homogeneity of $R_1$ implies that $R_1$ is in fact globally Lipschitz continuous (see, e.g.,~\cite[Lemma~5.6]{Rindler18book}). Hence, the hypotheses yield that there are $c_1,c_2 > 0$ such that $c_1\abs{z}\leq R_1(z) \leq c_2 \abs{z}$ for all $z \in \R^m$. Observe also that the convexity and $1$-homogeneity imply  that $R_1$ is sublinear since
\[
R_1(a+b)=2R_1\Big(\frac{a}2+\frac{b}2\Big)\leq R_1(a)+R_1(b).
\]

Associated with $\Rcal_1$ we define the \term{rate-independent dissipation} of $u \colon [0,T] \to \Lrm^1(\Omega;\R^m)$ on the subinterval $[s,t] \subset [0,T]$ to be
\[
  \Var_{\Rcal_1}(u;[s,t]) := \sup \setBB{ \sum_{k=0}^{N-1} \Rcal_1(u(s_{k+1}) - u(s_k)) }{ s = s_0 < s < \cdots < s_N = t }
\]
with the convention $\Var_{\Rcal_1}(u;[s,s]) := 0$. If for such a $u$ we have $\Var_{\Rcal_1}(u;[0,T]) < \infty$, we say that $u$ is of \term{bounded dissipation} (or \term{bounded $\Rcal_1$-variation}). Further set
\[
  \Var_{\Rcal_1}(u;(s,t)) := \lim_{\eps \todown 0} \Var_{\Rcal_1}(u;[s+\eps,t-\eps])
\]
and similarly for half-open intervals. 

The above notion of $\Rcal_1$-variation is sensitive to changes of $u$ at isolated points. When we apply the $\Rcal_1$-variation to maps that are only defined almost everywhere, we implicitly understand the $\Rcal_1$-variation to mean the infimum over all maps that are equal almost everywhere. This makes the $\Rcal_1$-variation lower semicontinuous with respect to weak convergence in $\Lrm^1(0,T;\Lrm^1(\Omega;\R^m))$. In any case, this will never cause problems since in all our results the jump points are explicitly resolved and the variation is ultimately only computed for \emph{continuous} maps.

For
\begin{align*}
  \Var(u;[s,t]) &:= \Var_{\norm{\frarg}_{\Lrm^1}}(u;[s,t]) \\
  &\phantom{:}= \sup \setBB{ \sum_{k=0}^{N-1} \normb{u(s_{k+1}) - u(s_k)}_{\Lrm^1} }{ s = s_0 < s < \cdots < s_N = t }
\end{align*}
we obtain the usual notion of \term{variation} for functions with values in $\Lrm^1(\Omega;\R^m)$. Due to the assumptions on $R_1$, we find that
\begin{align*}
c_1\Var(u;[s,t])\leq \Var_{\Rcal_1}(u;[s,t])\leq c_2\Var(u;[s,t])
\end{align*}
Hence, $\Var_{\Rcal_1}(u;[0,T]) < \infty$ if and only if $\Var(u;[0,T]) < \infty$ and in this case we write $u \in \BV(0,T;\Lrm^1(\Omega;\R^m))$. It can be shown that for $u \in \BV(0,T;\Lrm^1(\Omega;\R^m))$ at every $t \in (0,T)$ the left limit $u(t-) := \lim_{s\toup t} u(s)$ and the right limit $u(t+) := \lim_{s\todown t} u(s)$ exist, where the limits are taken with respect to the (strong) $\Lrm^1(\Omega;\R^m)$-topology.

For a countable ordered set of jump points $D = \{q_k\}_{k \in I} \subset J_u$ ($J_u$ being the jump set of $u$), $I \subset \Zbb$, we define the total variation on $[s,t] \setminus D$ as follows:
\[
  \Var_{\Rcal_1}(u;[s,t] \setminus D) := \Var_{\Rcal_1}(u;[s,t]) - \sum_{k \in I} \Rcal_1(u(q_k+)-u(q_k-)).
\]

Typical examples of $R_1$ are $R_1(z)=\sum_{i=1}^m\alpha_i\abs{z^i}$ with $\alpha_i > 0$, but also versions that have different frictions in different directions, such as the (scalar) example
\[
R_1(z)=
\begin{cases}
\alpha \abs{z} &\text{for }z>0,
\\
\beta \abs{z} &\text{for }z\leq 0,
\end{cases}
\]
where $\alpha,\beta > 0$.

\begin{example}[Friction potential]
\label{ex:fric}
Due to its $1$-homogeneity, the function $R_1$ is determined by the shape of the set $\set{z\in \R^m}{R_1(z)\leq 1}$. Indeed, one can associate to any bounded closed convex set $K$ that has $0$ as an interior point, a related friction potential $R_1$ by defining
\[
R_1(z):=\inf\set{s>0}{\frac1s z\in K}.
\]
This function is $1$-homogeneous by definition. It is also convex. To see this let us first assume that $z_1,z_2\in \partial K$. Then $R_1(z_1)=1=R_1(z_2)$ and by the convexity of $K$ we find that
$\theta z_1+(1-\theta) z_2\in K$ for all $\theta\in [0,1]$. Hence,
\[
R_1(\theta z_1+(1-\theta) z_2)\leq 1 = \theta R_1(z_1)+(1-\theta)R_1(z_2).
\]
By the above and the $1$-homogeneity it then follows for general $a,b\in \R^m$ (assuming that $a\neq 0$) that with $\beta:=\theta R_1(a)+(1-\theta) R_1(b)$,
\begin{align*}
R_1\big(\theta a+(1-\theta)b\big)
&=\beta
R_1\biggl(\frac{\theta R_1(a)}{\beta}\cdot\frac{a}{R_1(a)}+\frac{(1-\theta)R_1(b)}{\beta}\cdot\frac{b}{R_1(b)}\biggr)
\\
&\leq \beta\bigg(\frac{\theta R_1(a)}{\beta}
R_1\Big(\frac{a}{R_1(a)}\Big)+\frac{(1-\theta)R_1(b)}{\beta}R_1\Big(\frac{b}{R_1(b)}\Big)\biggr)
\\
&= \theta  R_1(a)+(1-\theta) R_1(b),
\end{align*}
which implies the convexity and hence also the continuity of $R_1$.
\end{example}

For $u \in (\Wrm^{1,2}_0 \cap \Lrm^q)(\Omega;\R^m)$ we also define the \term{regularized energy functional}
\begin{equation} \label{eq:Wcal}
  \Wcal(u) :=  \int_\Omega \frac{\abs{\nabla u(x)}^2}{2} \dd x + \Wcal_0(u) =  \int_\Omega \frac{\abs{\nabla u(x)}^2}{2} + W_0(u(x)) \dd x
\end{equation}
and the \term{total energy functional}
\[
  \Ecal(t,u) := \Wcal(u) - \dprb{f(t),u} = \int_\Omega \frac{\abs{\nabla u(x)}^2}{2} + W_0(u(x)) - f(t,x) \cdot u(x) \dd x. 
\]
In case that $\Wcal$ is \emph{convex} and $u_0$ satisfies the \term{compatibility condition}
\begin{equation} \label{eq:u0}
  \dprb{\Delta u_0 - \DD W_0(u_0) + f(0), \psi} \leq \Rcal_1(\psi)
\end{equation}
for all $\psi \in\Wrm^{1,2}_0(\Omega;\R^m)$, the existence and uniqueness of strong solutions is known (provided the given data satisfies the necessary regularity), see~\cite{RindlerSchwarzacherSuli17}. In the present paper we are concerned with \emph{non-convex} $\Wcal$, in which case solutions might have jumps. Further, in the non-convex case even a smooth global solution $u$ may not attain a given smooth initial condition that satisfies~\eqref{eq:u0}.

\subsection{Main results} \label{ssc:twospeedsol}

Our principal result is the construction of what one may call a ``solution with two different speeds'' by the parabolic approximation~\eqref{eq:PDE_lambda}. The following three cases may occur:
\begin{enumerate}
\item[Case~A] {\bf The continuous parts of the energy evolution}, where the rate-independent evolution follows the right-hand side in a quasi-static manner. Here, we are able to prove that the solution is strong and regular.
\item[Case~B] {\bf The parts where the energy approaches a jump with a rate slower than $\lambda$}. These jumps can be transferred to Case~A by introducing a {\em canonical time scale} that is slowing down the evolution in these parts by modifying the right-hand side.
\item[Case~C] {\bf The parts where the inertial dissipation is contributing.}  These are what we call {\em rate-dependent} jumps of the energy since they conserve the character of the inertial approximation. We resolve the path over the jump by an at most countable collection of solutions of~\eqref{eq:PDE_lambda} with $\lambda=1$ and constant right-hand side.
\end{enumerate}

For
\[
  u\in \BV(0,T;\Lrm^1(\Omega;\R^m))\cap \Lrm^\infty(0,T;(\Wrm^{1,2}_0 \cap \Lrm^q)(\Omega;\R^m))
\]
we define the \term{energy process}
\[
  E(t) := \Ecal(t,u(t)),  \qquad t \in [0,T].
\]

Below, in Theorem~\ref{thm:main_ex}, we will construct a rate-independent process $u$ as above together with processes $\{v^j(t_k, \frarg)\}_{k\in \Jbb,j\in \Ibb_k}$ (see~\ref{itm:iv} below), which we call a \term{two-speed solution} to~\eqref{eq:PDE_RI}, that satisfies the following properties:
\begin{enumerate}[(I)]
 \item \label{itm:i}
 The jump set $\Jbb$ of the energy process $E$ is exactly the (countable) set of negative jumps of $t \mapsto \Wcal(u(t))\in \BV(0,T)$ and is also the set of jumps of $u$ with respect to the $\Lrm^1(\Omega;\R^m)$-norm. Hence, the energy process $t \mapsto E(t)$ lies in $\BV(0,T)$.
 \item \label{itm:ii}
The solution is a strong solution at almost all times in $[0,T) \setminus \Jbb$.
 \item \label{itm:iii}
 At almost all non-jump points $t \in [0,T) \setminus \Jbb$ the \term{local stability}
\begin{align} \label{eq:stability}
\qquad
  -\DD_u \Ecal(t,u(t)) \in \partial \Rcal_1(0) =: \Scal
\end{align}
holds, that is,
\[\qquad
  \int_\Omega \bigl[- \DD W_0(u(t)) + f(t) \bigr] \cdot \psi - \nabla u(t) \cdot \nabla \psi \dd x
  \leq \int_\Omega R_1(\psi) \dd x
\]
for all $\psi \in \Wrm^{1,2}_0(\Omega;\R^m)$.
\item \label{itm:iv}
At every jump $t_k \in \Jbb$ of the energy there is a countable ordered set $\Ibb_k \subset \N$ and \term{jump resolution maps}
\[
v^i(t_k, \frarg)\in \Wrm^{1,2}(-\infty,\infty;\Wrm^{1,2}_0(\Omega;\R^m))\cap \Lrm^\infty(-\infty,\infty;\Wrm^{2,2}(\Omega;\R^m))
\]
for every $i \in \Ibb_k$, satisfying the \term{jump evolution}
\[\qquad
-\partial_\theta v^i(t_k,\theta) +\Delta v^i(t_k,\theta)-\DD W_0(v^i(t_k,\theta))+f(t_k)\in \partial \Rcal_1(\partial_\theta v^i(t_k,\theta))
\]
for $\theta \in (-\infty,\infty)$ in the strong sense, see~\eqref{eq:Varlam}.

\item \label{itm:v} 
For $t_k\in \Jbb$ and $i,j \in \Ibb_k$ with $i\neq j$ it holds that
\begin{align*}\qquad
E(t_k-)& \geq \Ecal(t_k,v^i(t_k,-\infty))>\Ecal(t_k,v^i(t_k,\infty))\\
&\geq \Ecal(t_k,v^j(t_k,-\infty))>\Ecal(t_k,v^j(t_k,\infty))\geq E(t_k+)
\end{align*}
or 
\begin{align*}\qquad
E(t_k-)& \geq \Ecal(t_k,v^j(t_k,-\infty))>\Ecal(t_k,v^j(t_k,\infty))\\
&\geq \Ecal(t_k,v^i(t_k,-\infty))>\Ecal(t_k,v^i(t_k,\infty))\geq E(t_k+),
\end{align*}
where $E(t_k\pm) := \lim_{t\to t_k\pm} E(t)$ and $u(t_k\pm)=\lim_{t\to t_k\pm} u(t)$.
Moreover, there is a constant $C > 0$ such that for all $t_k\in \Jbb$ and $i,j \in \Ibb_k$ with $i\neq j$ and $\Ecal(t_k,v^i(t_k,\infty))\geq  \Ecal(t_k,v^j(t_k,-\infty))$
it holds that
\begin{align*}\qquad
\norm{u(t_k-)-v^i(t_k,-\infty)}_{\Lrm^1} &\leq C\big(E(t_k-)-\Ecal(t_k,v^i(t_k,\infty))\big),\\
\norm{v^i(t_k,\infty)-v^j(t_k,-\infty)}_{\Lrm^1} &\leq C\big(\Ecal(t_k,v^i(t_k,\infty))- \Ecal(t_k,v^j(t_k,-\infty))\big),\\
\norm{v^i(t_k,\infty)-u(t_k+)}_{\Lrm^1} &\leq C\big(\Ecal(t_k,v^i(t_k,\infty))-E(t_k+)\big).
\end{align*}
\item \label{itm:vi}
For all subintervals $[s,t] \subset [0,T)$ with $s,t \in [0,T) \setminus \Jbb$, the following \term{energy balance} holds:
\[\qquad
  \Ecal(t,u(t)) = \Ecal(s,u(s)) - \Diss_+(u;[s,t]) - \int_s^t \dprb{\dot{f}(\tau), u(\tau)} \dd \tau,
\]
where $\Diss_+(u;[s,t])$ denotes the \term{total dissipation},
\begin{align}\qquad
  \Diss_+(u;[s,t]) &:= \Var_{\Rcal_1}(u;[s,t] \setminus \Jbb)
    + \sum_{t_k \in \Jbb \cap [s,t]}\Diss_{\mathrm{jump}}(t_k), 
  \notag\\
  \label{eq:dissjump}
  \Diss_{\mathrm{jump}}(t_k) &:= \sum_{i \in \Ibb_k} \biggl(\int_{-\infty}^\infty \norm{\partial_\theta v^i(t_k,\theta)}_{\Lrm^2}^2 + \Rcal_1(\partial_\theta v^i(t_k,\theta)) \dd \theta \biggr).
  \end{align}
\item \label{itm:vii} If there is a jump at the initial time, meaning that $0 \in \Jbb$, then the map $v^1(0, \frarg)$ has a special structure, namely
\[\qquad
  v^1(0,\theta) = u_0 \quad\text{for $\theta \in (-\infty,0]$.}
\]
\item \label{itm:viii}\label{itm:last} The set
\begin{equation} \label{eq:connect}\qquad
  \cl{\setb{ (s,u(s)) }{ s \in [0,T] \setminus \Jbb } \cup \bigcup_{t_k \in \Jbb} \bigcup_{i \in \Ibb_k} \setb{ (\theta,v^i(t_k,\theta) }{ \theta \in (-\infty,\infty) }},
\end{equation}
where the closure is taken with respect to the $(\R \times \Lrm^1(\Omega;\R^m))$-strong topology, is connected. If $0 \in \Jbb$, then $\theta$ instead is restricted to $[0,\infty)$ for $v^1(0,\theta)$.
\end{enumerate}

We remark that condition~\ref{itm:viii} should be interpreted in the sense that our two-speed solution is indeed a connected process and there are no ``gaps''.

We also note that energy balance can be written in the following conservative form: Upon introducing
\[
  \Gamma(t) := \Ecal(t,u(t)) + \Diss_+(u;[0,t]) + \int_0^t \dprb{\dot{f}(\tau), u(\tau)} \dd \tau
\]
and using the additivity of the dissipation with respect to the interval, we may write~\ref{itm:vi} concisely as
\[
  \Gamma(t) \equiv \Gamma(0)  \qquad\text{for all $t \in [0,T] \setminus \Jbb$.}
\]

We have the following existence theorem for \emph{regular} solutions:

\begin{theorem} \label{thm:main_ex}
Under the assumptions of Section~\ref{sc:assume} there exists an increasing function $\phi \colon [0,T]\to [0,T]$ such that $\phi(0)=0$, $\phi(T) = T$ and with
\[
  \tilde{f}(s):=f\circ\phi(s),  \qquad s \in [0,T],
\]
in place of $f$, there exists a two-speed solution $u$ to~\eqref{eq:PDE_RI} satisfying \ref{itm:i}--\ref{itm:last} with regularity in the following spaces:
\begin{align*}
  &\Lrm^a_\loc((0,T) \times \Omega;\R^m), \\
  &\Lrm^a_\loc(0,T;\Wrm^{1,r}_0(\Omega;\R^m)),  \\
  &\Lrm^2_\loc(0,T;\Wrm^{2,r}(\Omega;\R^m)),  \\
  &\Lrm^\infty(0,T;(\Wrm^{1,2} \cap \Lrm^q)(\Omega;\R^m)),  \\
  &\Lrm^\infty_\loc(0,T;\Wrm^{2,2}(\Omega;\R^m)),  \\
  &\BV(0,T;\Lrm^1(\Omega;\R^m)))
\end{align*}
for all $a\in [1,\infty)$, $r\in [1,2^*)$, where $2^* := \frac{2d}{d-2}$.
\end{theorem}

This theorem implies the following corollary for the original time scale: For every jump point $t_k\in \Jbb$ there exists at most one process
\[
  b^k \in \BV(0,1;\Lrm^1(\Omega;\R^m))\cap \Lrm^\infty(0,1;(\Wrm^{1,2}_0 \cap \Lrm^q)(\Omega;\R^m))
\]
that is a solution (in the sense of~\eqref{eq:strongsol_ineq}) to
\begin{equation} \label{eq:bk_evol}
\Delta b^k(\theta)-\DD W_0(b^k(\theta))+f(t_k)\in \partial \Rcal_1(\partial_\theta b^k(\theta))
\qquad \text{in $(0,1) \times \Omega$} 
\end{equation}
and that satisfies itself~\ref{itm:i}--\ref{itm:last}.
In the condition in Section~\ref{ssc:twospeedsol} we then have to replace the definition of $\Diss_{\mathrm{jump}}$ in~\eqref{eq:dissjump} by the following:
 \begin{align}
  \label{eq:dissjump2}
  \begin{aligned}
  \Diss_{\mathrm{jump}}(t_k) &:= \sum_{i \in \Ibb_k} \biggl(\int_{-\infty}^\infty \norm{\partial_\theta v^i(t_k,\theta)}_{\Lrm^2}^2 + \Rcal_1(\partial_\theta v^i(t_k,\theta))\dd \theta \biggr)\\ 
&\quad\quad + \Var_{\Rcal_1}(b^k;[s,t] \setminus J_{b^k}),
\end{aligned}
  \end{align}
where  $J_{b^k}$ is the jump set of $b^k$. Consequently, jumps in $b^k$ are \emph{not counted}. This is explained as follows: The process $b^k$ represents all the rate-independent evolution at time $t_k$, but with time ``stretched out'' (note the invariance of~\eqref{eq:bk_evol} under time-rescalings). The jumps in $b^k$ are in fact \emph{rate-dependent} pieces and thus resolved by the $v^i$ and counted in a rate-dependent way in~\eqref{eq:dissjump2}. Thus, $b^k$ represents all the (rate-independent) evolution pieces \emph{between} the $v^i$. We also refer to Example~\ref{ex:bump} for further illustration.

\begin{corollary} \label{cor:main_ex}
Under the assumptions of Section~\ref{sc:assume}, there exists a two-speed solution to~\eqref{eq:PDE_RI} with right-hand side $f \colon [0,T]\times \Omega\to \R^m$ satisfying~\ref{itm:i}--\ref{itm:last} with~\eqref{eq:dissjump2} in place of~\eqref{eq:dissjump}. Moreover, the constructed solution has regularity in the spaces listed in Theorem~\ref{thm:main_ex}.
\end{corollary}

Our second main result concerns the approximability of energy-preserving two-speed solutions.

\begin{theorem} \label{thm:main_approx}
The two-speed solution of~\eqref{eq:PDE_RI} constructed in Corollary~\ref{cor:main_ex} is the limit of strong solutions of~\eqref{eq:PDE_lambda} for a sequence $\lambda_j \to 0$ as $j\to \infty$. Indeed, there is a sequence $u_{\lambda_j}$ of solutions of~\eqref{eq:PDE_lambda} for $\lambda = \lambda_j$ such that
\begin{equation}
\label{eq:conv1}
\left\{
\begin{aligned}
u_{\lambda_j}&\to u  \quad\text{in } \Lrm^a_\loc((0,T) \times \Omega;\R^m) \quad\text{for $a\in [1,\infty)$;}\\
u_{\lambda_j}&\to u  \quad\text{in } \Lrm^a_\loc(0,T;\Wrm^{1,r}_0(\Omega;\R^m)) \quad\text{for $a\in [1,\infty)$, $r\in [1,2^*)$;}\\
u_{\lambda_j}&\toweak  u  \quad\text{in }\Lrm^2_\loc(0,T;\Wrm^{2,r}(\Omega;\R^m))\quad\text{ for $r\in [1,2^*)$;}\\
u_{\lambda_j}&\toweakstar u  \quad\text{in }\Lrm^\infty(0,T;(\Wrm^{1,2} \cap \Lrm^q)(\Omega;\R^m));\\
u_{\lambda_j}&\toweakstar u  \quad\text{in }\Lrm^\infty_\loc(0,T;\Wrm^{2,2}(\Omega;\R^m));\\
u_{\lambda_j} &\toweakstar u \quad\text{in $\BV(0,T;\Lrm^1(\Omega;\R^m)))$.}
\end{aligned}
\right.
\end{equation}
Moreover, for each $t_k\in \Jbb$ and $i\in \Ibb_k$ there exist sequences of \term{intermediate speeds} $\tau^{(k,i)}_j \in \R$ with $\tau^{(k,i)}_j \to 0$ as $j\to \infty$, such that for all $L>0$,
\begin{align}
\label{eq:conv2}
u_{\lambda_j}(t_k+\tau^{(k,i)}_j+\lambda_j\theta) \to v^i(t_k,\theta) \quad\text{in }\Wrm^{1,2}(-L,L;\Wrm^{1,2}_0(\Omega;\R^m)).
\end{align}
Finally, with
\[
  \mu^\mathrm{RI}_{\lambda_j}(A):=\int_A \Rcal_1(\dot{u}_{\lambda_j})\dd t, \qquad
  \mu^\mathrm{RD}_{\lambda_j}(A):=\int_A\lambda_j \norm{\dot{u}_{\lambda_j}}_{\Lrm^2}^2\dd t,  \qquad \text{$A\subset [0,T]$ Borel,}
\]
we have that
\begin{align}
\mu^\mathrm{RI}_{\lambda_j} &\toweakstar \Var_{\Rcal_1}(u;\frarg\setminus\Jbb) + \mu^\mathrm{RI}_{\mathrm{jump}},  \label{eq:conv3} \\
\mu^\mathrm{RD}_{\lambda_j} &\toweakstar \mu^\mathrm{RD}  \label{eq:conv4}
\end{align}
as measures on $[0,T]$, where
\begin{align*}
  \mu^\mathrm{RI}_{\mathrm{jump}} &:= \sum_{t_k \in \Jbb} \biggl[\sum_{i \in \Ibb_k} \int_{-\infty}^\infty \Rcal_1(v^i(t_k,\theta)) \dd \theta + \int_0^1\Rcal_1(\partial_\theta b^k(\theta))\dd \theta  \biggr] \delta_{t_k}, \\  
  \mu^\mathrm{RD} &:= \sum_{t_k \in \Jbb} \biggl[ \sum_{i \in \Ibb_k} \int_{-\infty}^\infty \norm{\partial_\theta v^i(t_k,\theta)}_{\Lrm^2}^2 \dd \theta \biggr] \delta_{t_k}.
\end{align*}
\end{theorem}

\begin{remark}[The rate-independent path length]
The singular measure $\mu^\textrm{RI}_{\textrm{jump}}$ with support on the jumps of the energy gives the rate-independent path length, for which
\[
\mu^\textrm{RI}_{\textrm{jump}}(\{t_k\})\geq\Rcal(u(t_k+) - u(t_k-))
\]
and this inequality may be strict. This phenomenon is due to the fact that the (non-convex) elastic potential might enforce a complex path in arbitrary short time scales, creating a contribution that is larger than the $\Rcal_1$-distance; see Example~\ref{ex:jumplength}.
\end{remark}

\begin{remark}[Balanced viscosity solutions]
The characterization of $\mu^\mathrm{RI}$ and $\mu^\mathrm{RD}$ is related to the theory of balanced viscosity solutions of~\cite{MielkeRossiSavare16}. Indeed, the solution on the $\theta$-scale can be associated to an optimal path with respect to the related Finsler energy, see for instance Theorem~B.18 in~\cite{MielkeRossiSavare12} and Proposition~3.19~(3) in~\cite{MielkeRossiSavare16}. From these results one can deduce that our two-speed solutions are also balanced viscosity solutions. We give an explicit calculation in Example~\ref{ex:jumpapprox}. 
\end{remark}

The proofs of Theorem~\ref{thm:main_ex}, Corollary~\ref{cor:main_ex}, and Theorem~\ref{thm:main_approx} will be accomplished in Sections~\ref{sc:energy_ineq}--\ref{sc:twospeed}; see Section~\ref{ssc:proofs} for how these parts fit together to yield the above results.

\subsection{Examples}

In this section we consider four examples illustrating various aspects of the theory.

The first example shows that the dissipation along a jump may be strictly larger than the total variation. It is also elucidated how this solution relates to the balanced viscosity solutions of Mielke--Rossi--Savar\'{e}~\cite{MielkeRossiSavare09,MielkeRossiSavare12, MielkeRossiSavare16}.

\begin{figure}[t]
\includegraphics[width=180pt]{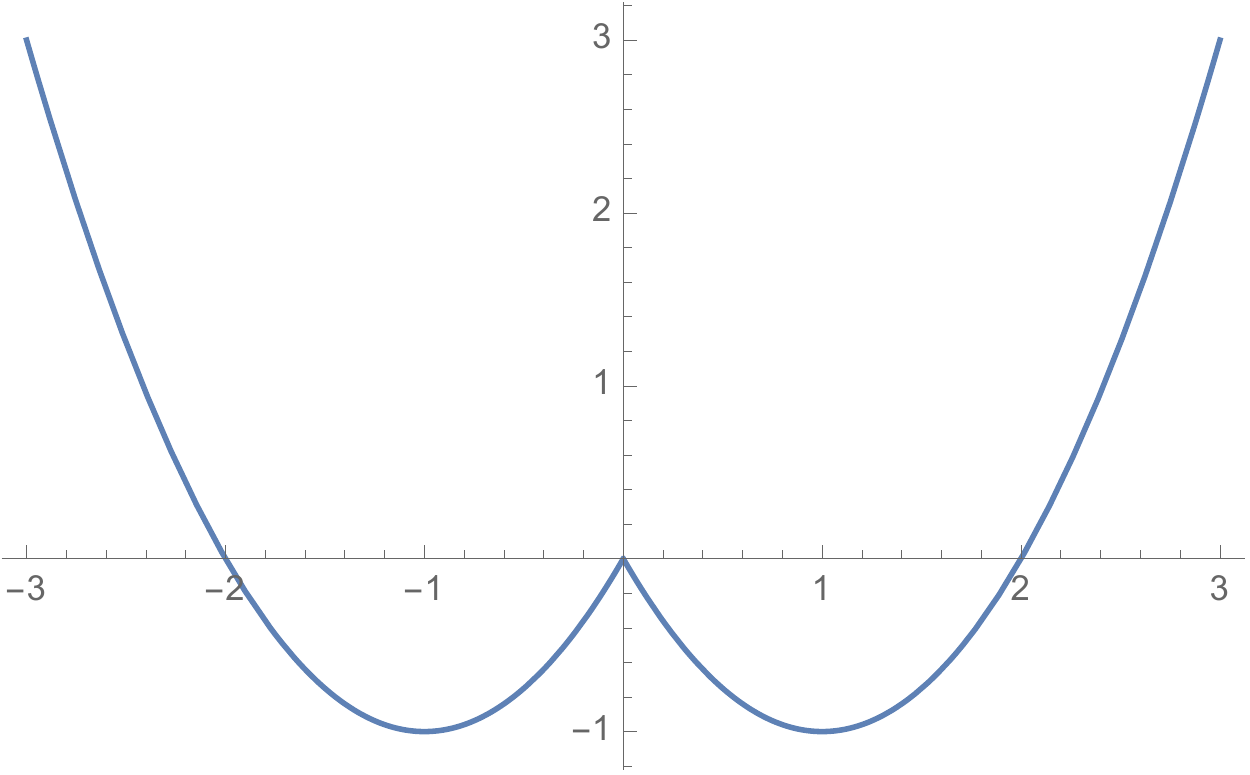}\hspace{10pt}
\includegraphics[width=180pt]{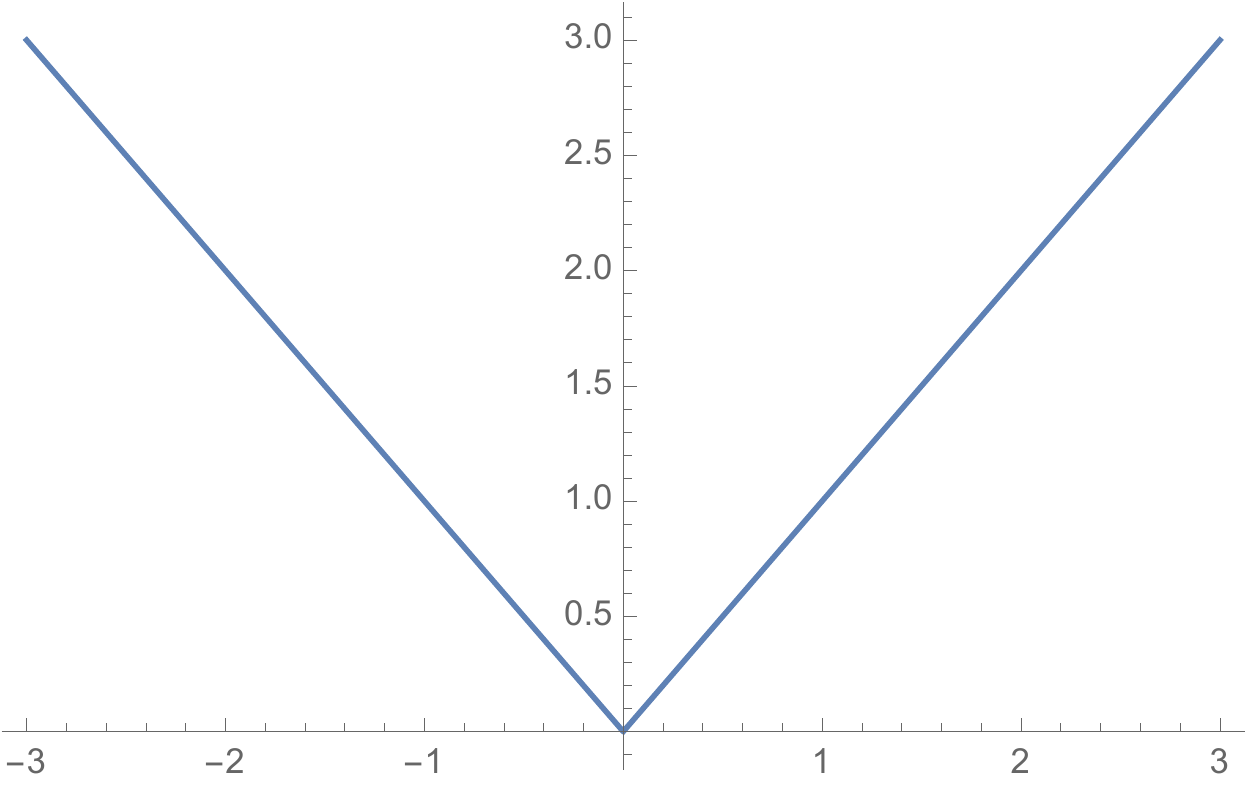}
\caption{$W_0$ and $R_1$ in Example~\ref{ex:jumpapprox}}
\label{fig:ex1}
\end{figure}

\begin{example}[Different notions of solution]
  \label{ex:jumpapprox}
Consider the following zero-dimensional double-well setup, see Figure~\ref{fig:ex1}:
\[
  \Wcal_0(z) = W_0(z) := \min\{ z(z+2), z(z-2) \},  \qquad
  \Rcal_1(z) = R_1(z) := \abs{z}.
\]
We also use the right-hand side
\[
  f(t) := t,  \qquad t \in [0,\infty),
\]
and the initial value $u(0) = -1$.

\begin{figure}[t]
\includegraphics[width=180pt]{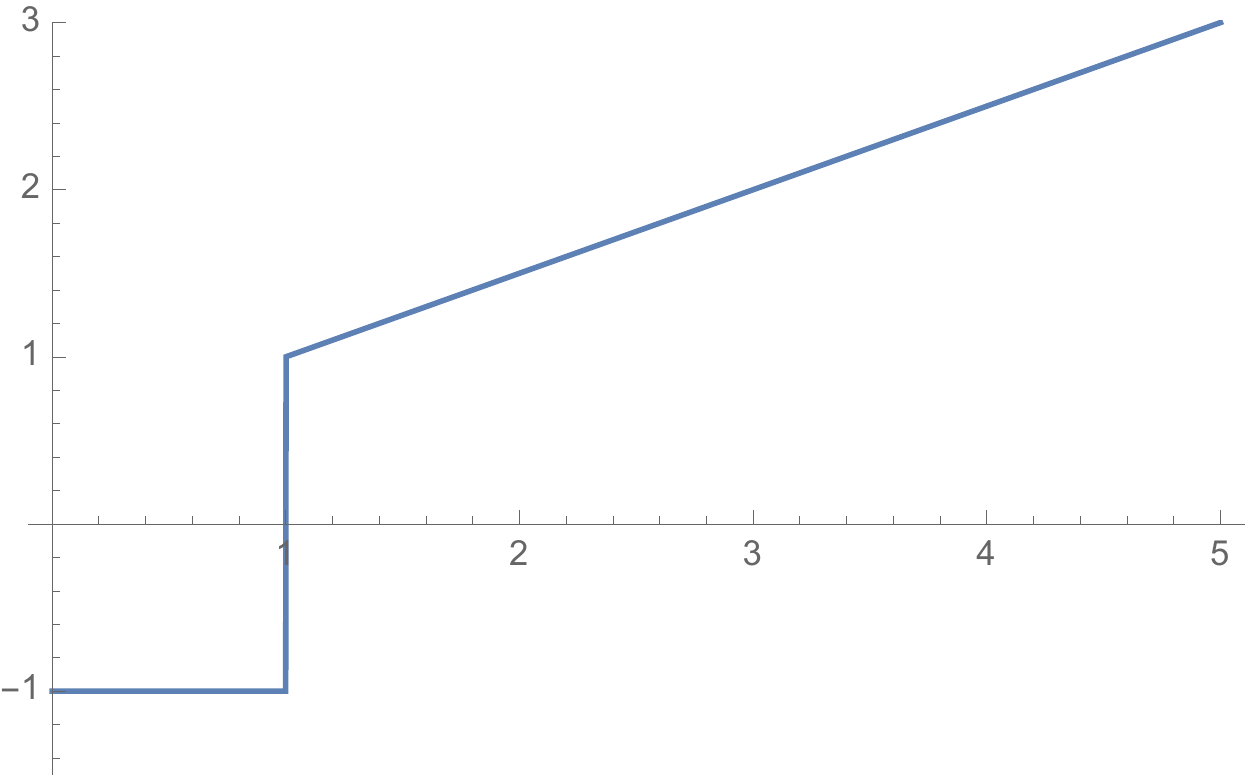}\hspace{10pt}
\includegraphics[width=180pt]{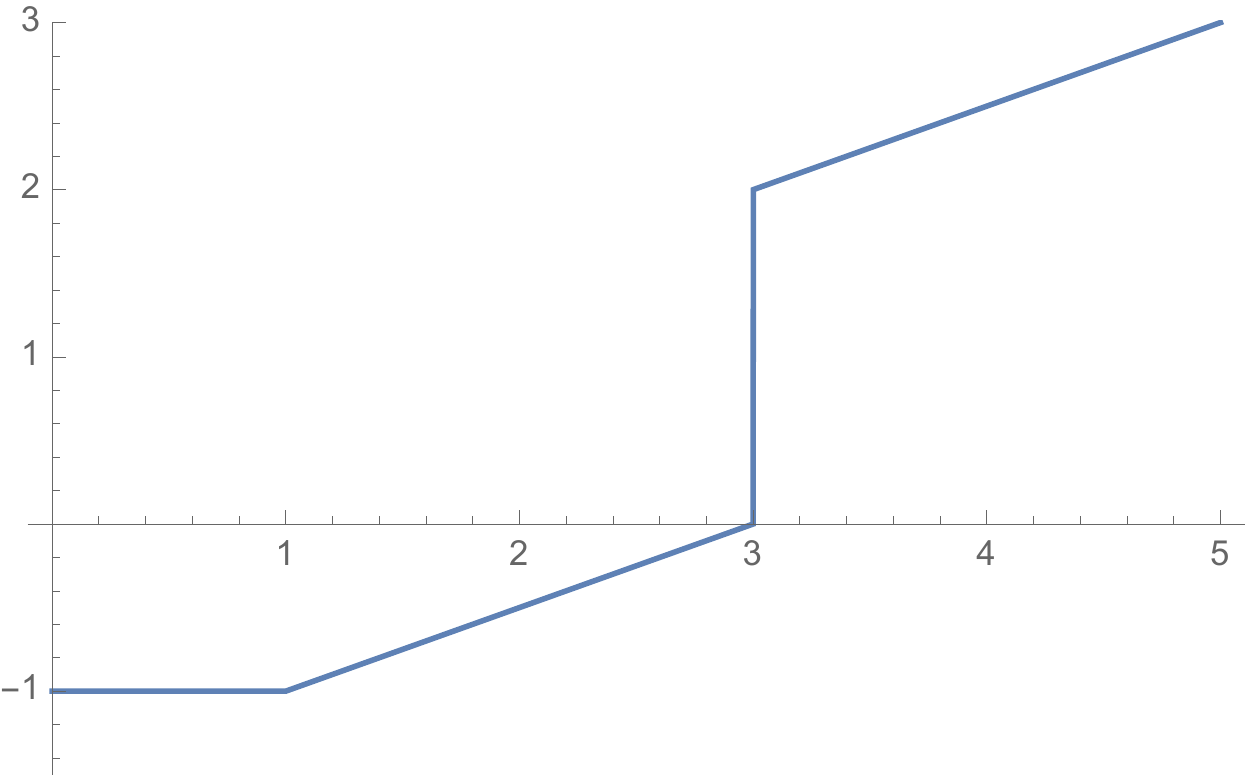}
\caption{$u^{\mathrm{weak}}$ and $u^{\mathrm{ext}}$}
\label{fig:ex_weakstrong}
\end{figure}

It can be easily seen that
\[
  u^{\mathrm{weak}}(t) := \begin{cases}
    -1   & \text{if $t \in [0,1)$,} \\
    \frac{t+1}{2}  & \text{if $t \in [1,\infty)$}
  \end{cases}
\]
is an energetic solution in the sense of Mielke--Theil~\cite{MielkeTheil04} on the infinite time interval $[0,\infty)$, see Figure~\ref{fig:ex_weakstrong}. This means that $u = u^{\mathrm{weak}}$ satisfies the \emph{energy balance}
\begin{equation} \label{eq:energy_balance}
  \Ecal(t,u(t)) - \Ecal(0,u(0)) = \int_0^t \dot{f}(s) \cdot u(s) \dd s - \mathrm{Var}_{R_1}(u;[0,t)),  \qquad t \in [0,\infty),
\end{equation}
as well as the \emph{global stability} inequality
\[
  \Ecal(t,u(t)) \leq \Ecal(t,z) + R_1(z-u(t))  \quad
  \text{for all $z \in \R$,}
\]
where $\Ecal(t,z) := W_0(z) - f(t) \cdot z$ for all $t \in [0,\infty)$, and $\mathrm{Var}_{R_1}$ denotes the total $R_1$-variation.

On the other hand, over the time interval $[0,3)$ also a \emph{maximally strong solution} $u^{\mathrm{strong}}$ exists, namely
\[
  u^{\mathrm{strong}}(t) := \begin{cases}
    -1   & \text{if $t \in [0,1)$,} \\
    \frac{t-3}{2}  & \text{if $t \in [1,3)$,}
  \end{cases}
\]
This solution can be extended to $t \in [0,\infty)$, as a weak solution only, by setting
\[
  u^{\mathrm{ext}}(t) := \begin{cases}
    -1   & \text{if $t \in [0,1)$,} \\
    \frac{t-3}{2}  & \text{if $t \in [1,3)$,} \\
    \frac{t+1}{2}  & \text{if $t \in [3,\infty)$.}
  \end{cases}
\]
This map satisfies the energy balance~\eqref{eq:energy_balance} as well as the \emph{local stability} condition
\[
  -\DD \Ecal(t,u(t)) \in \partial R_1(0), \qquad t \in [0,\infty).
\]
The difference can be explained by considering the \emph{effective} energy functional before the jump, namely $z \mapsto W_0(z) + R_1(z-(-1)) - f(t) \cdot z = z(z+2) + \abs{z+1} - tz$, see Figure~\ref{fig:ex_tilt}. It can be seen that the weak solution $u^{\mathrm{weak}}$ jumps from one well to the other as soon as it can lower the total energy (at time $t = 1$), whereas the maximally strong solution $u^{\mathrm{ext}}$ has to wait with its jump until the potential barrier has vanished (at time $t = 3$). In physical problems, we thus see that the energetic solutions jump in general too early, whereas viscosity approximations and also balanced viscosity solutions (both are equal to $u^{\mathrm{ext}}$) in a physically reasonable way.

\begin{figure}[t]
\includegraphics[width=180pt]{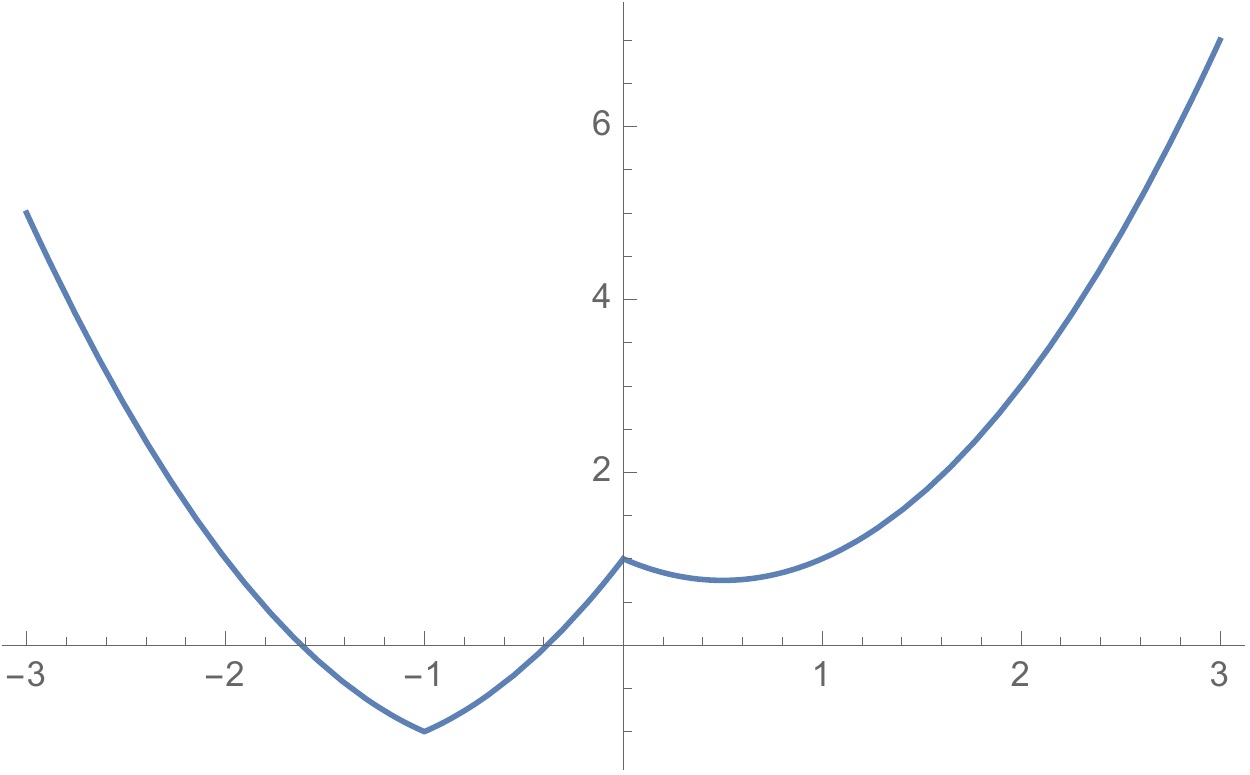}\hspace{10pt}
\includegraphics[width=180pt]{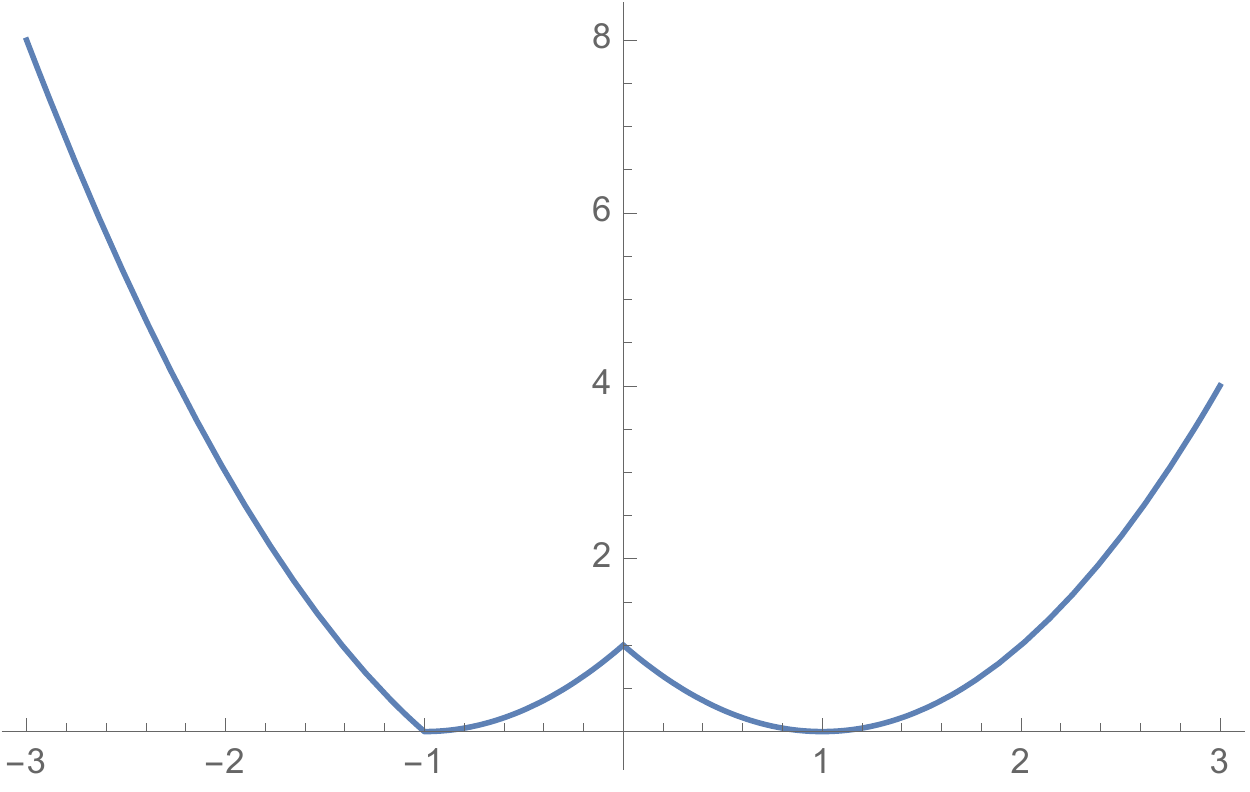}\hspace{10pt}
\includegraphics[width=180pt]{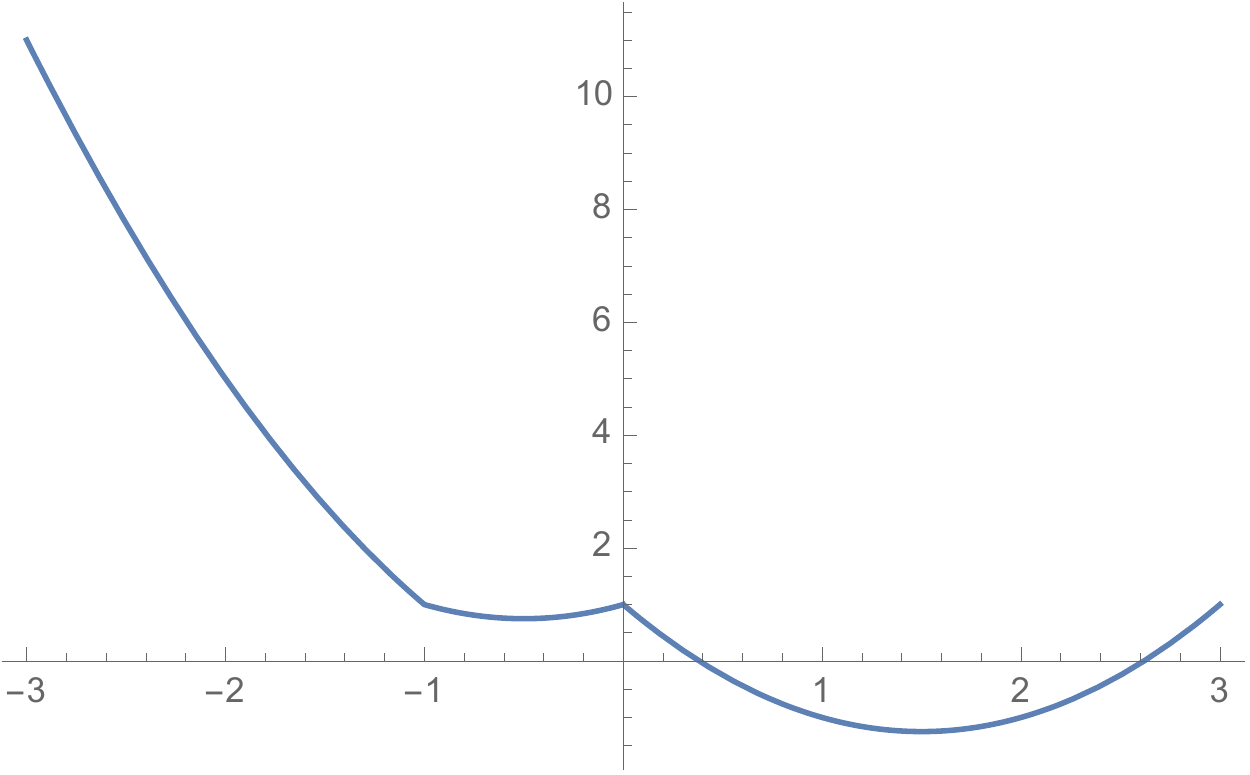}\hspace{10pt}
\includegraphics[width=180pt]{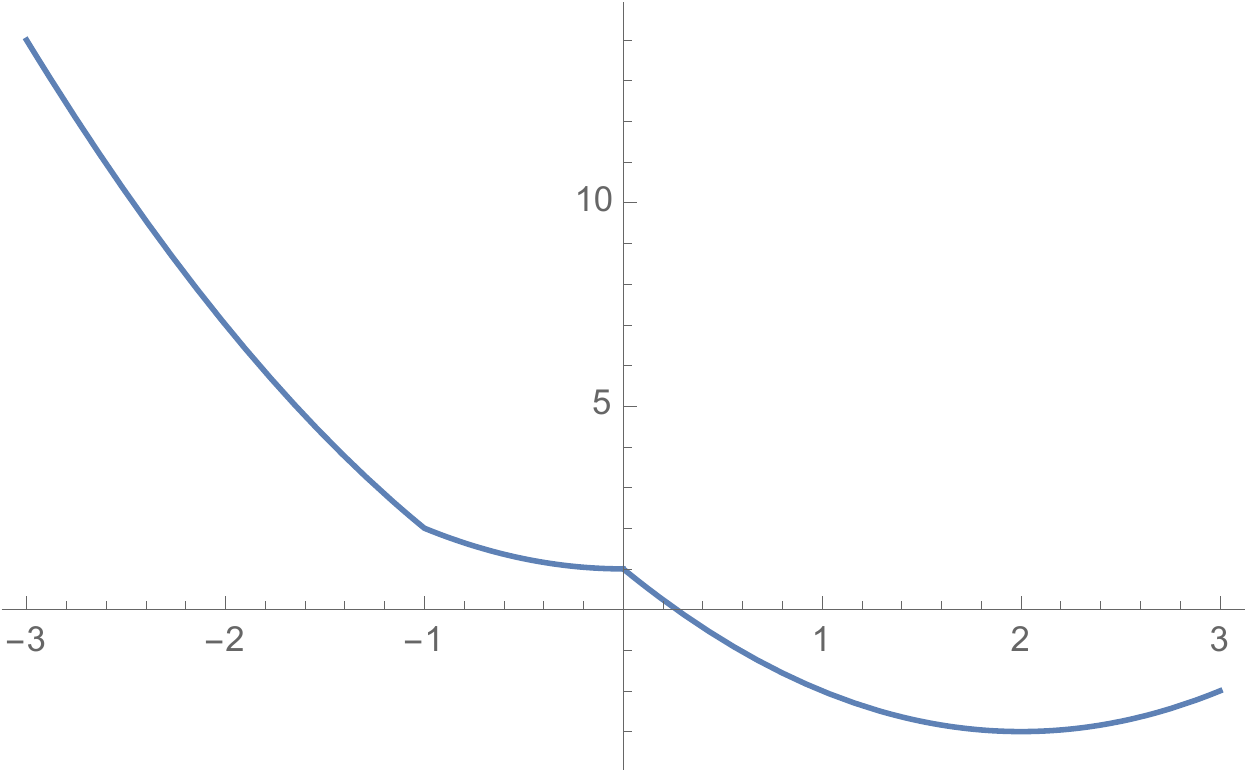}
\caption{The effective potential $z \mapsto W_0(z) + R_1(z-(-1)) - f(t) \cdot z = z(z+2) + \abs{z+1} - tz$ at $t = 0,1,2,3$.}
\label{fig:ex_tilt}
\end{figure}

If we add the viscosity term to our example, we need to solve
\[
  \Sgn(\dot{u}^{\mathrm{jump}}_\lambda(t)) + \lambda \dot{u}^{\mathrm{jump}}_\lambda(t) +\DD W_0(u^{\mathrm{jump}}_\lambda(t)) = f(t)
\]
for $\lambda > 0$. At $t=3$, which we shift to $0$ (to investigate the jump in $u^{\mathrm{ext}}$), we thus need to solve the ODE
\[
  \lambda \dot{u}^{\mathrm{jump}}_\lambda + 2u_\lambda = 4.
\]
It can be checked that the solution is
\[
  u^{\mathrm{jump}}_\lambda(\theta) = 2-2 \ee^{-2\theta/\lambda}  \qquad\text{with}\qquad
  \dot{u}^{\mathrm{jump}}_\lambda(\theta) = \frac{4 \ee^{-2\theta/\lambda}}{\lambda},  \qquad \theta \in (0,\infty),
\]
see Figure~\ref{fig:ex_jump_lambda}. Clearly, they converge to the expected jump at $0$ (corresponding to $t=3$), which we already saw in $u^{\mathrm{ext}}$. Hence, the viscosity approximation ``selects'' $u^{\mathrm{ext}}$ over $u^{\mathrm{weak}}$. 

\begin{figure}[t]
\includegraphics[width=180pt]{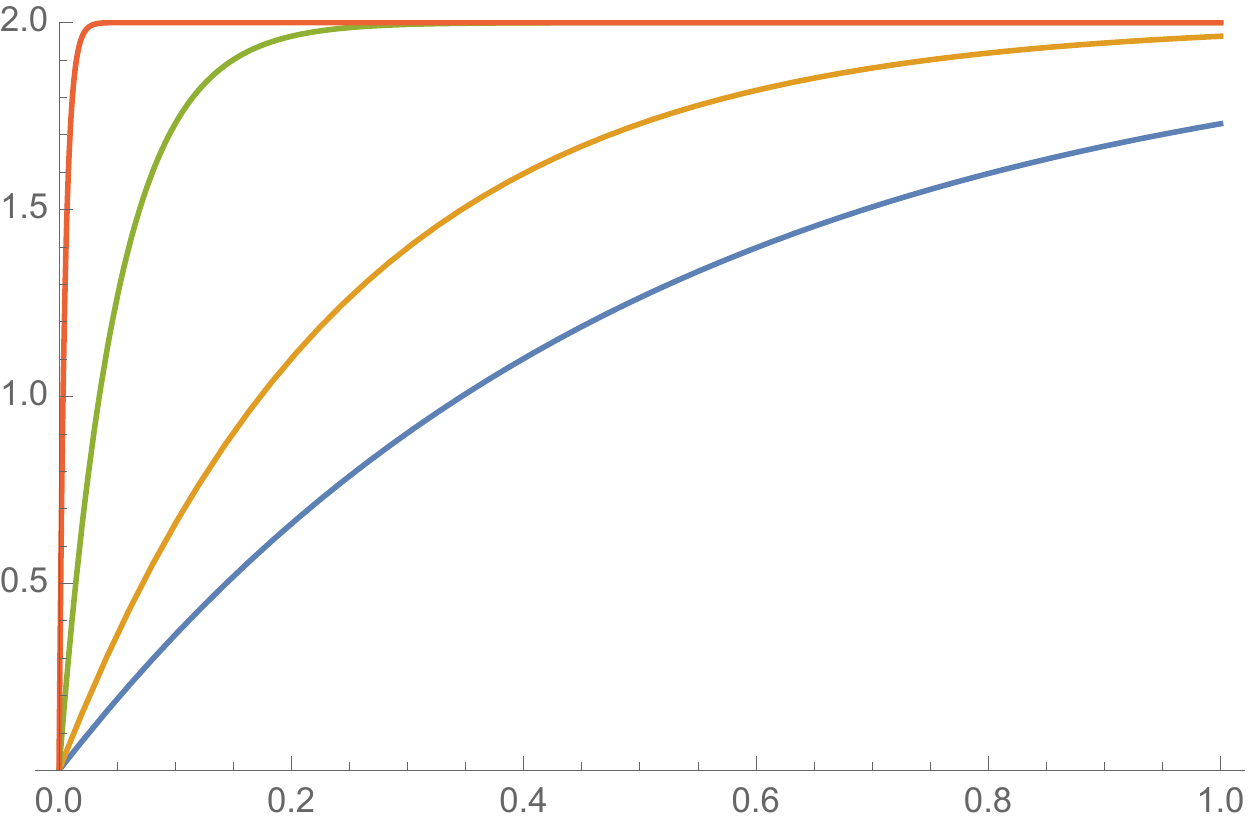}\hspace{10pt}
\includegraphics[width=180pt]{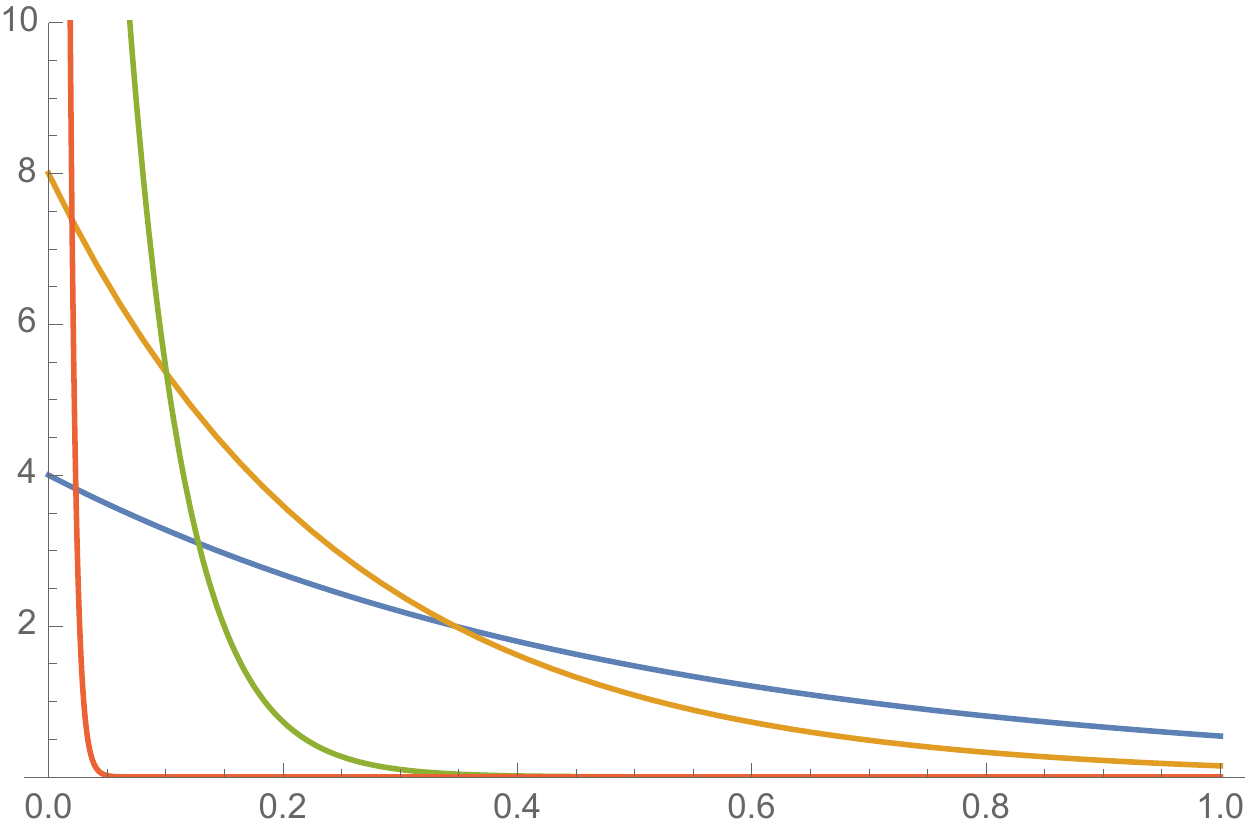}\hspace{10pt}
\caption{The solution $u^{\mathrm{jump}}_\lambda$ and its derivative $(u^{\mathrm{jump}}_\lambda)'$ for $\lambda = 1,0.5,0.1,0.01$}
\label{fig:ex_jump_lambda}
\end{figure}

The total expended rate-independent and rate-dependent energy over the jump are
\begin{align*}
  \mu^\mathrm{RI}(3) &= \int_0^\infty \abs{\dot{u}^{\mathrm{jump}}_\lambda(\theta)} \dd \theta = \int_0^\infty 4 \ee^{-2\theta} \dd \theta = 2, \\
  \mu^\mathrm{RD}(3) &= \int_0^\infty  \abs{\dot{u}^{\mathrm{jump}}_\lambda(\theta)}^2 \dd \theta = \int_0^\infty   16 \ee^{-4 \theta} \dd \theta = 4.
\end{align*}
Their sum is equal to the expected jump energy
\[
  E(3-,u(3-)) - E(3+,u(3+)) = 6.
\]
Observe that by defining $v^{\mathrm{jump}}_\lambda(\theta) = u^{\mathrm{jump}}_\lambda(\lambda \theta)=2-2 \ee^{-2\theta} $, we find that the jump is resolved in one single fast-scale solution $v^{\mathrm{jump}}(\theta)=2-2 \ee^{-2\theta}$. In particular, the dissipation of the fast-scale solution $v^{\mathrm{jump}}$ is equal to the jump in the energy.

In the framework of balanced viscosity solutions of Mielke--Rossi--Savar\'{e}~\cite{MielkeRossiSavare09,MielkeRossiSavare12, MielkeRossiSavare16}, the dissipation is given by the \emph{Finsler dissipation cost}
\begin{align*}
  \Delta_{\ffrak_\lambda}(u_-,u_+) := \inf \setBB{ \int_0^1 \ffrak_\lambda(\zeta,\dot{\zeta}) \dd r }{ &\text{$\zeta \colon [0,1] \to \Lrm^2(\Omega;\R^m)$ absolutely continuous,} \\
  &\text{$\zeta(0) = u_-$, $\zeta(1) = u_+$} },
\end{align*}
where
\[
  \ffrak_\lambda(\zeta,\dot{\zeta}) := \Psi_\lambda(\dot{\zeta}) + \Psi_\lambda^*(-\DD E(\zeta)),  \qquad
  \Psi_\lambda(\dot{\zeta}) := R_1(\dot{\zeta}) + \frac{1}{2} \lambda \abs{\dot{\zeta}}^2
\]
and $\Psi_\lambda^*$ is the Fenchel conjugate of $\Psi_\lambda$. By the Fenchel inequality and a well known theorem in the theory of convex subdifferentials (see, e.g., Theorem~3.32 in~\cite{Rindler18book}), we have that $\ffrak_\lambda(\zeta,\dot{\zeta})$ is minimal (with respect to the choice of $\dot{\zeta}$) if
\[
  -\DD E(\zeta) \in \partial \Psi_\lambda(\dot{\zeta}) = \partial R_1(\dot{\zeta}) + \lambda \dot{\zeta},  \qquad \theta \in (-\infty,\infty),
\]
and in this case
\[
  \ffrak_\lambda(\zeta,\dot{\zeta}) = -\DD E(\zeta) \cdot \dot{\zeta}.
\]
Thus, the jump transient is again identified as $u^{\mathrm{jump}}$ from above and (notice that the path integral is rescaling-invariant)
\begin{align*}
  \Delta_{\ffrak_\lambda}(u(3-),u(3+))
  &= \int_0^\infty -\DD E(u^{\mathrm{jump}}(\theta)) \cdot \dot{u}^{\mathrm{jump}}(\theta) \dd \theta \\
  &= \int_0^\infty - \bigl( 2(2-2\ee^{-2\theta/\lambda})) - 5 \bigr) \cdot \frac{4 \ee^{-2\theta/\lambda}}{\lambda} \dd \theta \\
  &= \int_0^\infty 4 \ee^{-2\theta} + 16 \ee^{-4 \theta} \dd \theta \\
  &= 6,
\end{align*}
which agrees with our calculation for the two-speed dissipation above. So, $u^{\mathrm{ext}}$ is in fact the balanced viscosity solution.
\end{example}

The next example shows that, unlike globally stable energetic solutions, our two-speed solutions (like balanced viscosity solutions) may have a jump at the initial time, which may also follow a non-trivial path even if $f\equiv 0$.

\begin{example}
\label{ex:jumplength}
We take the initial condition $(u(0),v(0))=(0,0)$. We aim to find a potential such that $(u,v)$ moves to $(u(1),v(1))=(0,1)$ by solving the following rate-independent system ($R_1(u,v) := \abs{u} + \abs{v}$):
\begin{align}
\label{eq:length_RI}
\left \{
\begin{aligned}\
&-\DD_u W(u,v) \in \Sgn(\dot{u}),\\
&-\DD_v W(u,v) \in \Sgn(\dot{v}),\\
&u(0)=0, \quad  v(0)=0.
\end{aligned}
\right.
\end{align}
For this, take $\phi\in \Crm^2([\frac13,\frac23],[-1,1])$ such that $\phi(\frac13)=-1$ and $\phi(\frac23)=1$.
Consider
\[
W(u,v) :=\begin{cases}
-u-\frac13 -v &\text{if }v\in (-\infty,\frac13],\\
\phi(v)(u-\frac13)-v-\frac{2}{3} &\text{if }v\in (\frac13,\frac23),\\
u-v-1 &\text{if }v\in [\frac23,\infty),
\end{cases}
\]
which is continuous. 
It can be easily checked that
\[
(u(t), v(t))=\begin{cases}
(t,t) &\text{if }t\in [0,\frac13],\\
(\frac{1}{3},t) &\text{if }t\in (\frac13,\frac23),\\
(1-t,t) &\text{if }t\in [\frac23,1),\\
(0,1) &\text{if }t\in [1,\infty)
\end{cases}
\]
is a solution to~\eqref{eq:length_RI} and the initial value $(u(0),v(0))=(0,0)$ is (locally) stable, i.e.\ $-\DD W(u,v) \in \Scal = [-1,1]^2$. We find
\[
\int_0^1\abs{\dot{u}}+\abs{\dot{v}} \dd t=\frac{5}{3}=-\frac{1}{3}-(-2)=W(0,0)-W(0,1).
\]
 By simply rescaling the solution $(u_\tau(\theta)), (v_\tau(\theta))):=(u(\frac{ \theta}{\tau}),v(\frac{\theta}{\tau}))$, where $\tau > 0$, we find
\[
\int_0^\tau\abs{\partial_\theta{u}_\tau }+\abs{\partial_\theta{v}_\tau} \dd \theta=\frac{5}{3}=W(0,0)-W(u_\tau(\tau),v_\tau (\tau)).
\]
The limit function is
\[
(u(t),v(t)) = \lim_{\tau\to 0}\,(u_\tau(t), v_\tau(t))=:(\tilde{u}(t),\tilde{v}(t))=(0,1) \quad\text{for }t>0
\]
Hence, $(\tilde{u},\tilde{v})$ has a single jump at the initial time $t = 0$ and
\[
\mu^\mathrm{RI}(0)=\frac{5}{3}>1=\Var_{R_1}(\tilde{u};\{0\}).
\]
This implies that we have constructed two different solution of~\eqref{eq:length_RI}, even though in~\eqref{eq:length_RI} there is no external force and the initial condition is (locally) stable. 
\end{example}

Next we demonstrate how  an arbitrarily long (rate-independent) jump path can be approximated by a parabolic approximation. This shows the necessity to include a canonical slow time in the concept of solutions, since otherwise the purely rate-independent measure in a point is more than the jump length. It is noteworthy that in the given situation the pathological movement is induced by the energy potential only and hence does not depend on the (parabolic) approximation.

\begin{figure}[t]
\includegraphics[width=80pt]{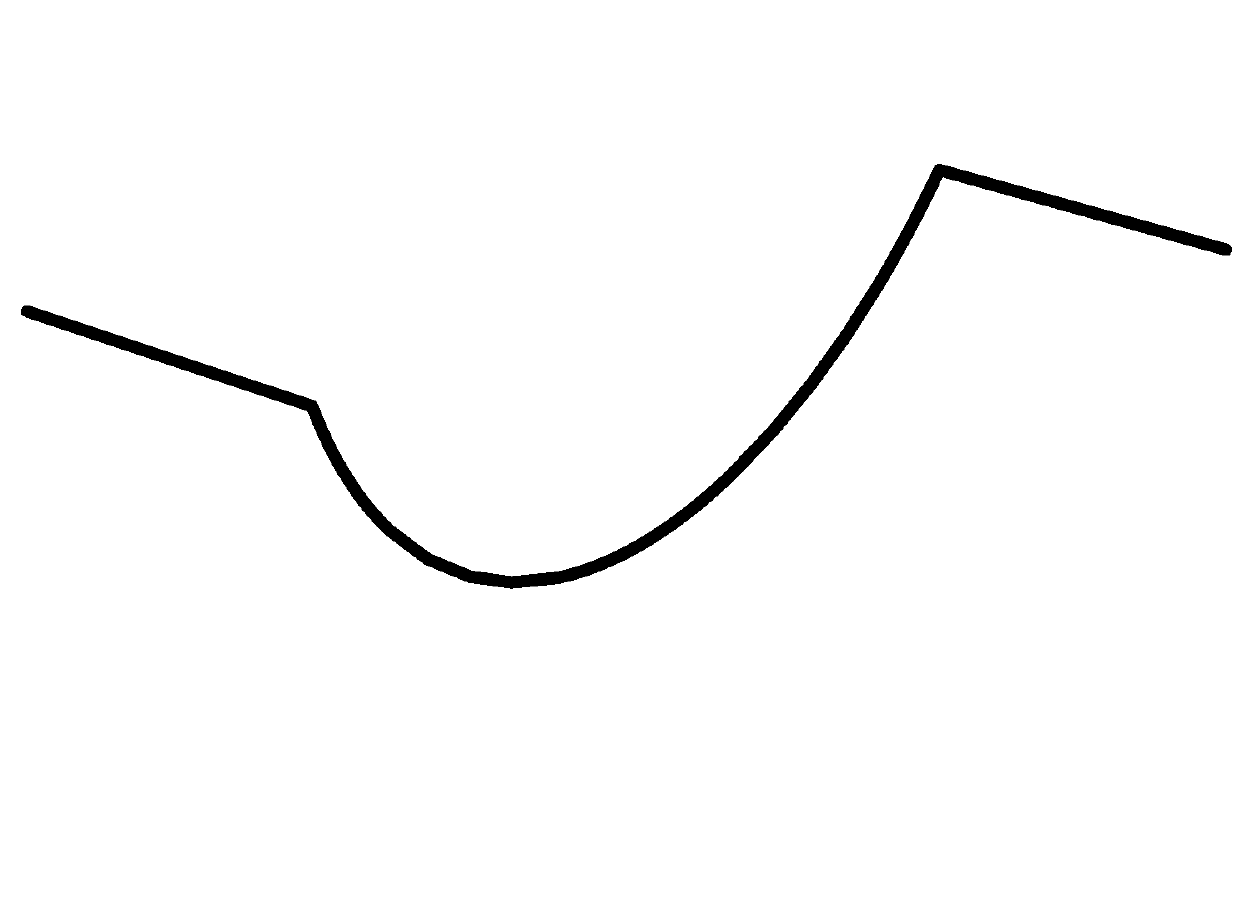}\raisebox{33pt}{,}
\includegraphics[width=80pt]{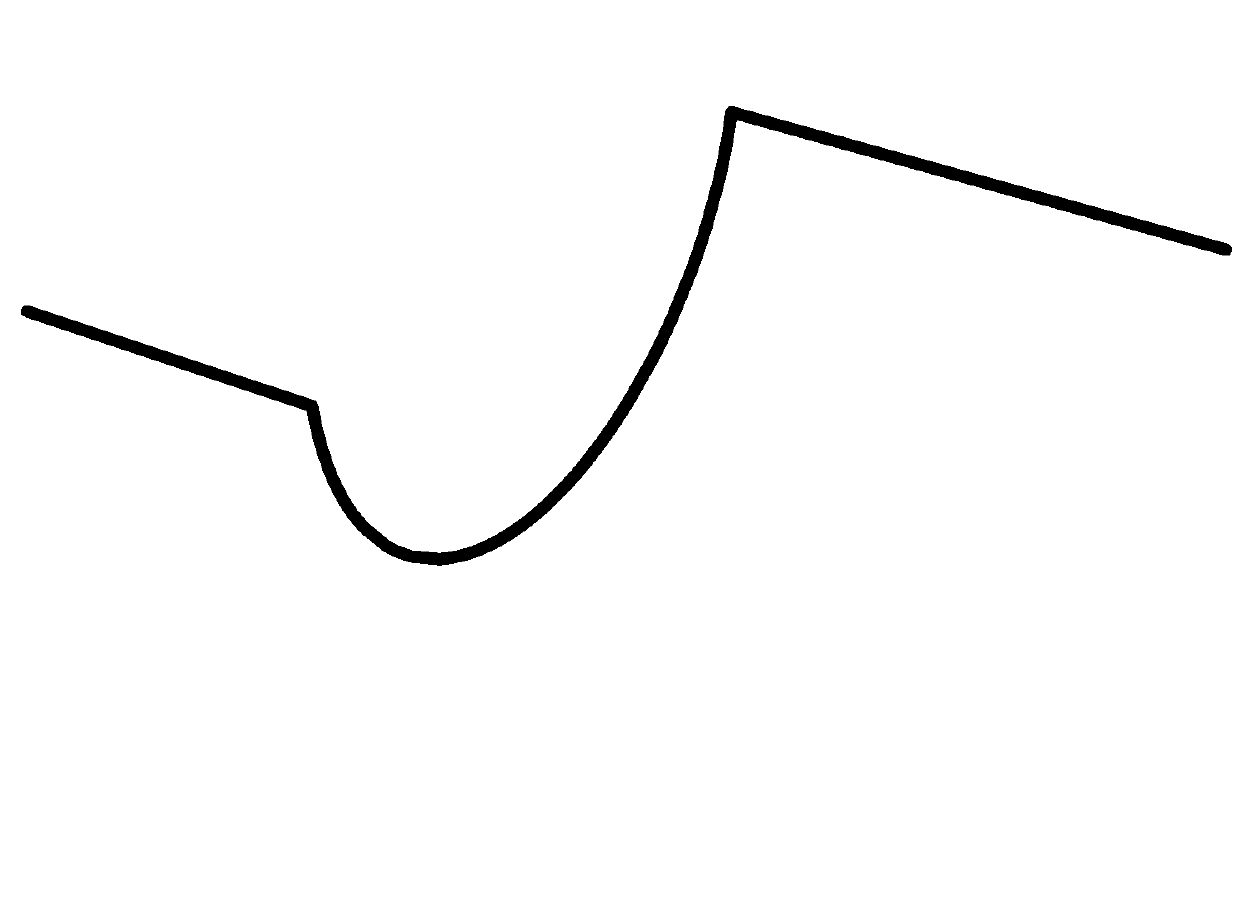}\raisebox{33pt}{,}
\includegraphics[width=80pt]{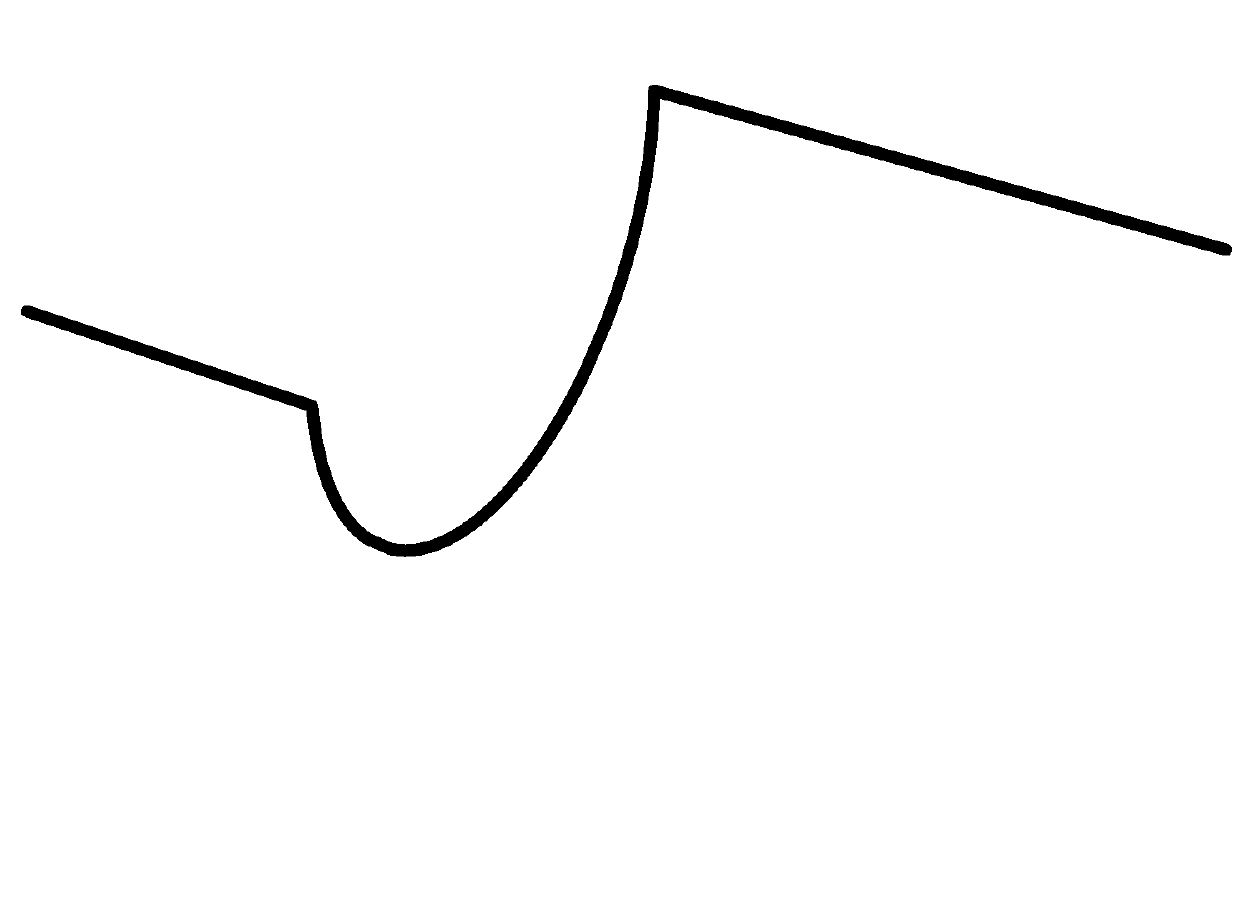}\raisebox{33pt}{$\to$}
\includegraphics[width=80pt]{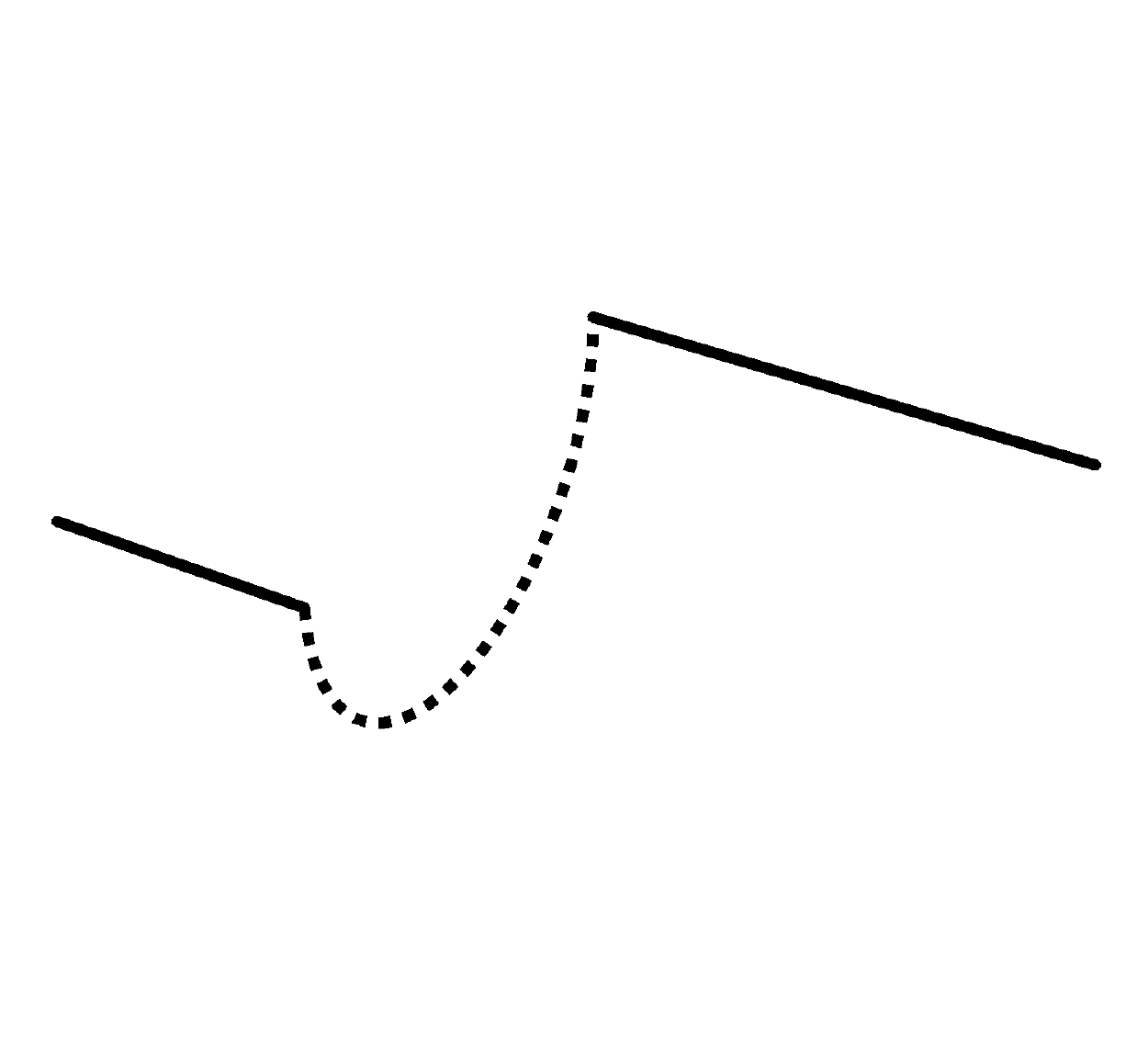}
\caption{An evolution $u_\lambda(t)$ following a non-affine path as $\lambda \todown 0$.}
\label{fig:semicirc}
\end{figure}

\begin{example}
\label{ex:jumplength1}
Consider $\xi \in \Crm^1(0,1)$ such that $\xi(0)=0=\xi(1)$. For $K>0$ we define the functional $W = W_0$ via
\[
W(u,v) := K^2(u-\xi(v))^2+K(v-1)^2.
\]
Moreover, we set $R_1(u,v) := \abs{u} + \abs{v}$ and $f \equiv 0$.
We wish to solve
\begin{align*}
\left \{
\begin{aligned}\
&-\DD_u W(u,v) \in \Sgn(\dot{u}) + \lambda \dot{u},\\
&-\DD_v W(u,v) \in \Sgn(\dot{v}) + \lambda \dot{v},\\
&u(0)=0, \quad  v(0)=0.
\end{aligned}
\right.
\end{align*}
By the energy equality we find that
\[
\int_0^t\abs{\dot{u}_\lambda}+\abs{\dot{v}_\lambda}+\lambda\abs{\dot{u}_\lambda}^2+\lambda\abs{\dot{v}_\lambda}^2\dd s+K^2(u_\lambda(t)-\xi(v_\lambda(t)))^2+K(v_\lambda(t)-1)^2 = K.
\] 
Hence, by increasing $K$ the function $u$ can be forced to be follow $\xi(v)$ closely, see Figure~\ref{fig:semicirc} for an illustration. In particular, when letting $\lambda \todown 0$, we find a jump at $0$, from $(0,0)$ to $(0,1)$. Hence, $\Var((u,v);\{0\})=1$, but $
\mu^\mathrm{RI}(0) \approx \int_0^1\sqrt{1+\abs{\xi'(s)}^2}\dd s$,
which can be chosen to be arbitrarily large.
\end{example}

The final example demonstrates the importance of rescaling the time via the function $\phi$. Since rate-independent solutions can be glued together and accelerated or delayed, it is easy to produce potentials where jumps are collections of various paths with different speeds. The next example exhibits a jump that can be resolved using a change of variables $\phi$ (or be described by the rate-independent evolution $b^k$) and a rate-dependent jump. In some sense it is a combination of Example~\ref{ex:jumpapprox} and Example~\ref{ex:jumplength}. Further combinations are possible. In particular, switching between rate-independent evolutions (collected in $b^k$) and rate-dependent evolutions (in the $v^i$) are easily constructed in a similar manner.

\begin{figure}[t]
\includegraphics[width=180pt]{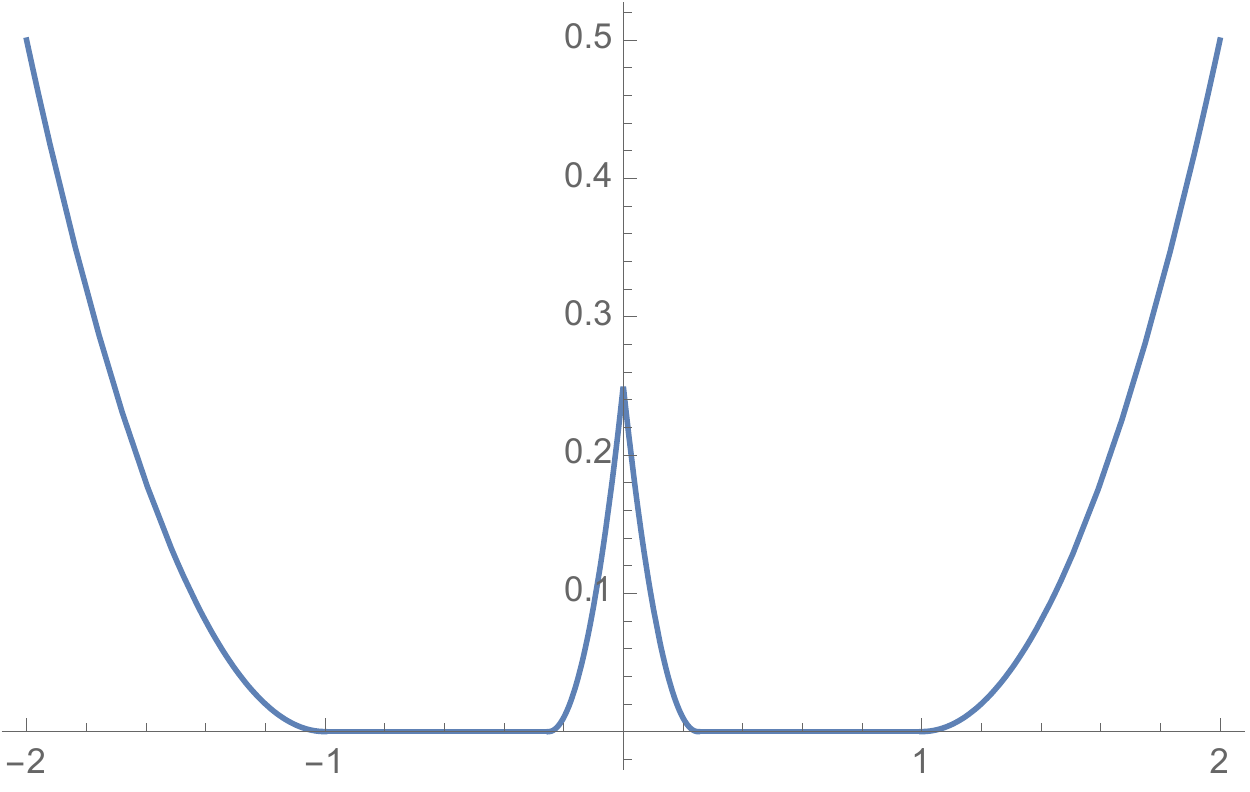}\hspace{10pt}
\includegraphics[width=180pt]{ex_disspot.pdf}
\caption{$W_0$ and $R_1$ in Example~\ref{ex:bump}}
\label{fig:bump}
\end{figure}

\begin{example}
\label{ex:bump}
We introduce the following energy that provides a lot of possible (energy-preserving) solutions for rate-independent systems, see Figure~\ref{fig:bump}:
\[
  W_0(z)  := \begin{cases}
   \frac12 (z+1)^2   & \text{if $z \in (-\infty,-1)$,} \\
   4\min\bigl\{z(z+\frac12),z(z-\frac12)\bigr\}+\frac{1}{4}  & \text{if $z\in (-\frac{1}{4},\frac{1}{4})$,} \\
  \frac12 (z-1)^2& \text{if $z \in (1,\infty)$,}\\
   0 & \text{if $z \in [-1,-\frac{1}{4}]\cup [\frac{1}{4},1]$.}
  \end{cases}
\]
This implies
\[
 \DD W_0(z) = \begin{cases}
    z+1   & \text{if $z \in (-\infty,-1)$,} \\
    8z+2  & \text{if $z\in (-\frac{1}{4},0)$,} \\
    8z-2  & \text{if $z\in (0,\frac{1}{4})$,} \\
   z-1 & \text{if $z \in (1,\infty)$,}\\
   0 & \text{if $z \in [-1,-\frac{1}{4}]\cup [\frac{1}{4},1]$.}
  \end{cases}
\]
We also set
\[
  u_0:=0, \quad  R_1(z) := \abs{z}, \quad f(t):=1.
\]
It can be seen easily that the solution immediately (i.e.\ at $t=0$) transitions rate-dependently from $z=0$ to $z=\frac14$ since the total force $- \DD W_0(z)+f(0)$ is greater than $1$, so we need the rate-dependent dissipation. Then, it further transitions from $z=\frac{1}{4}$ to $z=1$ since then $- \DD W_0(z)+f(0) = 1$, so it balances with the rate-independent dissipation term. Then, for $z > 1$, we have $- \DD W_0(z)+f(0) = 2-z = 1$, so the solution remains at $z=1$. In the situation of Theorem~\ref{thm:main_ex}, time would be rescaled so that there is a rate-dependent jump at $s=0$ (resolved into a $v^1$), then, a rate-independent transition from $s=0$ to $s=s^*$ (which can be chosen in different ways, e.g.\ $s^*=\frac14$ if we choose an arclength-parametrization). In the situation of Corollary~\ref{cor:main_ex}, however, we would stick with the original time and have $b^{t=0}(\theta)$ transitioning from $z=\frac{1}{4}$ to $z=1$ instead.
\end{example}

\section{Energy inequality and stability} \label{sc:energy_ineq}

It is a well-known observation that in order to prove that a constructed process is a solution to a rate-independent system, it suffices to establish only an energy \emph{inequality} together with  the local stability~\eqref{eq:stability} (this holds for instance for Mielke--Theil energetic solutions). For strong solutions, of course one has to require in addition some regularity of the processes. In this section we show how such regularity estimates can be obtained on a large part of the time domain. The key technical result of this section is Lemma~\ref{lem:key1}, which estimates the oscillation of the solution by the oscillation of the energy. Then, a covering argument via a suitable maximal function implies that the solution is indeed strong on a large part of the time domain.

Within this section we assume that we are given
\[
  u\in \Lrm^\infty(0,T;(\Wrm^{1,2}\cap \Lrm^q)(\Omega;\R^m))
\]
such that the following two conditions hold:
\begin{enumerate}[(i)]
  \item the energy process $E(t) := \Ecal(t,u(t))$ satisfies for almost all $s,t \in [0,T]$ with $s < t$ the inequality
\begin{align} \label{eq:subenerg}
\qquad
E(t) - E(s) \leq - \Var_{\Rcal_1}(u;[s,t]) -\int_s^t \dprb{\dot{f}(\tau), u(\tau)} \dd \tau \; ;
\end{align}
  \item the local stability~\eqref{eq:stability} holds for almost every $t\in [0,T]$, namely
\begin{align} \label{eq:stab}
\qquad
  \int_\Omega - \nabla u(t) \cdot \nabla \psi + \bigl[- \DD W_0(u(t)) + f(t) \bigr] \cdot \psi \dd x
  \leq \int_\Omega R_1(\psi) \dd x
\end{align}
for all $\psi \in \Wrm^{1,2}_0(\Omega;\R^m)$.
\end{enumerate}

The first lemma contains the observation, alluded to above, that~\eqref{eq:subenerg} and~\eqref{eq:stab} are equivalent to~\eqref{eq:strongsol_ineq} provided that $u$ has additional regularity.

\begin{lemma}
\label{lem:strongsol}
Assume that $\nabla u(t), \nabla \dot{u}(t) \in \Lrm^2(\Omega;\R^{m \times d})$ or $u(t), \Delta \dot{u}(t) \in \Lrm^2(\Omega;\R^{m \times d})$ for some $t \in [0,T]$ and that~\eqref{eq:subenerg},~\eqref{eq:stab} are satisfied (at this $t$ and for almost every $s < t$). Then, $u$ is a strong solution to~\eqref{eq:PDE_lambda} at $t$, i.e.,~\eqref{eq:strongsol_ineq} holds at $t$.
\end{lemma}

\begin{proof}
Assume that $\nabla u(t), \nabla \dot{u}(t) \in \Lrm^2(\Omega;\R^{m \times d})$; the proof for $u(t), \Delta \dot{u}(t) \in \Lrm^2(\Omega;\R^{m \times d})$ is analogous.

Dividing~\eqref{eq:subenerg} by $t-s$ and letting $s \toup t$ gives
\[
  \Rcal_1(\dot{u}(t)) + \int_\Omega \nabla u(t) \cdot \nabla \dot{u}(t) - \bigl[- \DD W_0(u(t)) + f(t) \bigr] \cdot \dot{u}(t) \dd x
  \leq 0.
\]
Thus, the inequality in~\eqref{eq:strongsol_ineq}, that is,
\begin{align*}
  &\Rcal_1(\dot{u}(t)) - \int_\Omega \nabla u(t) \cdot \bigl( \nabla \xi(t) - \nabla \dot{u}(t) \bigr)  \dd x \\
  &\quad + \int_\Omega \bigl[- \DD W_0(u(t)) + f(t) \bigr] \cdot (\xi(t) - \dot{u}(t)) \dd x \leq \Rcal_1(\xi(t))
\end{align*}
for all $\xi \in \Lrm^1(0,T;\Wrm^{1,2}_0(\Omega;\R^m))$, follows by adding~\eqref{eq:stab} with $\psi := \xi(t)$ to the preceding inequality.
\end{proof}

The proof of the following lemma is a special case of~\cite[Lemma~3.1]{RindlerSchwarzacherSuli17} and is therefore omitted.

\begin{lemma} \label{lem:apriori-space}
Assume that $W_0$ satisfies the assumptions in~\ref{as:W}. Then, any weak solution $u\in (\Wrm^{1,2}_0 \cap \Lrm^q)(\Omega;\R^m)$ of
\[
\left\{
\begin{aligned}
-\Delta u +\DD W_0 (u)&=g \quad\text{ in $\Omega$,}\\
u|_{\partial\Omega}&=0,
\end{aligned} \right.
\]
where $g\in \Lrm^s(\Omega;\R^m)$ for $s\in [2,\infty)$, satisfies
\[
 \norm{\nabla^2u}_{\Lrm^s}\leq C(1+ \norm{g}_{\Lrm^s}+\norm{g}_{\Lrm^{2}}^{q-1}).
\]
\end{lemma}

The first regularity result is a straightforward consequence of the preceding lemma:

\begin{lemma}
\label{lem:Linfty}
If~\eqref{eq:stab} is satisfied at $t \in (0,T)$, then $\nabla^2 u(t)\in \Lrm^s(\Omega)$ for all $s\in [2,\infty)$ and
\[
 \norm{u(t)}_{\Wrm^{2,s}}\leq C\big(1+\norm{f(t)}_{\Lrm^s}+\norm{f(t)}_{\Lrm^2}^{q-1}\big).
\]
If~\eqref{eq:stab} is satisfied at almost every $t\in (0,T)$, then
 \[
 \norm{u}_{\Lrm^\infty(0,T;\Wrm^{2,s}(\Omega))}\leq C\big(1+\norm{f}_{\Lrm^\infty(0,T;\Lrm^s(\Omega))}+\norm{f}_{\Lrm^\infty(0,T;\Lrm^2(\Omega))}^{q-1}\big).
 \] 
\end{lemma}

\begin{proof}
Assume that $u(t)$ satisfies~\eqref{eq:stab}. Then, for all $s \in [2,\infty)$,
\begin{align*}
\norm{-\Delta u(t)+ \DD W_0(u(t))}_{\Lrm^s} &\leq \sup_{\substack{\psi\in \Lrm^{s'}(\Omega)\\\norm{\psi}_{\Lrm^{s'}} \leq 1}} \Rcal(\psi)+ \norm{f(s)}_{\Lrm^s}
\leq C+ \norm{f(t)}_{\Lrm^s}.
\end{align*}
Hence the first result follows by Lemma~\ref{lem:apriori-space} and the Poincar\'e--Friedrichs inequality. The second statement is immediate from the first.
\end{proof}

Note that by the above lemma, the Sobolev embedding theorem, and the dual characterization of $\Lrm^\infty$, any $u$ satisfying \eqref{eq:stab} automatically satisfies
\[
  u \in \Lrm^\infty((0,T) \times \Omega;\R^m).
\]

\subsection{Control of the energy process via the dissipation}

The following key \emph{oscillation lemma} gives a way of proving that the time derivative $\dot{u}$ posseses a spatial gradient at points where the energy is smooth.

\begin{lemma}
\label{lem:key1}
Let $[s,t]\subset [0,T]$ such that $\nabla^2u(s),\nabla^2u(t) \in \Lrm^2(\Omega;\R^{m \times d \times d})$. Moreover, let~\eqref{eq:subenerg} be satisfied in $[s,t]$ and let~\eqref{eq:stab} be satisfied at $s$. Then,
\[
\norm{\nabla u(t) - \nabla u(s)}_{\Lrm^2}^2
\leq C \biggl( \norm{u(t)-u(s)}_{\Lrm^1}^2 + \norm{\dot{f}}_{\Lrm^\infty([s,t]\times\Omega)} \cdot \int_s^t\norm{u(t)-u(\tau)}_{\Lrm^1} \dd \tau \biggr).
\]
\end{lemma}

\begin{proof}
By~\eqref{eq:subenerg} and the definition of $\Var_{\Rcal_1}$,
\begin{align}
\label{eq:varbyenergy}
E(t)-E(s)\leq - \Rcal_1(u(t)-u(s)) -\int_s^{t}\dprb{\dot{f}(\tau),u(\tau)} \dd \tau .
\end{align}
On the other hand we find by~\eqref{eq:stab} at the time $s$, taking $\psi:=u(t)-u(s)$ as test function, that
\begin{align}
\label{eq:statioreg}
\begin{aligned}
&- \int_\Omega \nabla (u(s)) \cdot \bigl( \nabla u(t) - \nabla u(s) \bigr) \dd x - \Rcal_1(u(t)-u(s))
\\
 &\quad \leq  \int_\Omega \DD W_0(u(s)) \cdot (u(t)-u(s)) -f(s) \cdot (u(t)-u(s)) \dd x.
\end{aligned}
\end{align}
 Now we calculate, using~\eqref{eq:varbyenergy} and~\eqref{eq:statioreg},
\begin{align*}
&\int_\Omega \frac{\abs{\nabla u(t) - \nabla u(s)}^2}{2}\dd x
\\
&\quad =\int_\Omega \frac{\abs{\nabla u(t)}^2}{2}-\frac{\abs{\nabla u(s)}^2}{2}\dd x-\int_\Omega \nabla u(s)\cdot \big(\nabla u(t)-\nabla u(s)\big)\dd x
\\
&\quad = E(t)-E(s)
 +\int_\Omega -W_0(u(t))+f(t)\cdot u(t) + W_0(u(s))-f(s)\cdot u(s)\dd x
\\
&\quad\quad -\int_\Omega\nabla u(s)\cdot \big(\nabla u(t)-\nabla u(s)\big)\dd x
\\
&\quad \leq -\int_s^{t} \dprb{\dot{f}(\tau), u(\tau)} \dd \tau+\dprb{f(t)-f(s), u(t)}
 \\
&\quad \quad+
 \int_\Omega W_0(u(s))-W_0(u(t))+\DD W_0(u(s))\cdot(u(t)-u(s))\dd x 
 \\
 &\quad = \int_s^{t} \dprb{\dot{f}(\tau), u(t)-u(\tau)} \dd \tau 
\\
&\quad \quad+ \int_\Omega \biggl( \int_0^1\DD W_0 \bigl( u(t)+\gamma (u(s)-u(t)) \bigr)\dd \gamma \biggr) \cdot(u(s)-u(t)) \dd x \\
&\quad \quad+ \int_\Omega \DD W_0(u(s))\cdot(u(t)-u(s))\dd x
\\
&\quad = \int_s^{t}\dprb{\dot{f}(\tau), u(t)-u(\tau)} \dd \tau \\
&\quad \quad+ \int_\Omega \int_0^1 \int_0^1\DD^2 W_0\Big(u(t)+\gamma (u(s)-u(t)) +\sigma \big(u(s)-u(t)- \gamma (u(s)-u(t))\big)\Big)\\
&\quad\quad\qquad\qquad\qquad\qquad[(1-\gamma)(u(s)-u(t)),u(t)-u(s)] \dd \sigma \dd \gamma\dd x.
\end{align*}
Since we assumed $\nabla^2 u(s), \nabla^2 u(t) \in \Lrm^2(\Omega;\R^{m \times d \times d})$ and $d \in \{2,3\}$, we have shown an $\Lrm^\infty$-bound on $u(t)$. Thus, the above together with the boundedness of $\DD^2 W_0$ on bounded sets (as mentioned above, $u \in \Lrm^\infty((0,T) \times \Omega;\R^m)$ by~\eqref{eq:stab} via~Lemma~\ref{lem:Linfty}) implies
\begin{align*}
&\norm{\nabla u(t) - \nabla u(s)}_{\Lrm^2}^2 \leq C\norm{u(t)-u(s)}_{\Lrm^2}^2  + \norm{\dot{f}}_{\Lrm^\infty([s,t]\times\Omega)} \int_{s}^t\norm{u(t)-u(\tau)}_{\Lrm^1} \dd \tau.
\end{align*}
Now, by interpolation and Sobolev inequality, we find for any $\delta > 0$ a constant $c_\delta > 0$ such that with the Sobolev embedding exponent $2^* := \frac{2d}{d-2}$ we have
\begin{align*}
\norm{u(t)-u(s)}_{\Lrm^2}^2 &\leq C_\delta\norm{u(t)-u(s)}^2_{\Lrm^1}+\delta \norm{u(t)-u(s)}_{\Lrm^{2^*}}^2 
\\
&\leq C_\delta\norm{u(t)-u(s)}_{\Lrm^1}^2+ C\delta \norm{\nabla u(t) - \nabla u(s)}_{\Lrm^2}^2.
\end{align*}
Together with the previous estimate, this implies the result by absorption.
\end{proof}

The next proposition shows that the variation controls the energy process away from large jumps.

\begin{proposition}
\label{prop:E_contpoint}
Let $u\in \Lrm^\infty((0,T) \times \Omega;\R^m)\cap \Lrm^\infty(0,T;\Wrm^{2,2}(\Omega;\R^m))$ satisfy~\eqref{eq:subenerg},~\eqref{eq:stab} in $[0,T]$. Let $\epsilon>0$ and suppose that the energy process $E(t) := \Ecal(t,u(t))$ has no jump larger than $\epsilon$ on the interval $[a,b] \subset [0,T]$ with $a,b$ not jump points. Then there is a constant $C > 0$, which is independent of $\epsilon$, such that
\[
\Var_{\Rcal_1}(u;[a,b])\leq E(a)-E(b)-\int_a^b\dprb{\dot{f}(\tau), u(\tau)} \dd \tau\leq (1+C\epsilon)\Var_{\Rcal_1}(u;[a,b]).
\]
\end{proposition}

\begin{proof}
We assume without loss of generality that~\eqref{eq:subenerg} holds at $s = a$, $t = b$ (since $a,b$ are not jumps).
Hence, the lower bound is~\eqref{eq:subenerg}. We are left to prove
\[
E(a)-E(b)-\int_a^b\dprb{\dot{f}(\tau), u(\tau)} \dd \tau \leq (1+C\epsilon)\Var_{\Rcal_1}(u;[a,b]).
\]
By assumption, all jumps in $E(t)$ are of height at most $\epsilon$. In particular, we find a $\delta>0$ such that for almost all $s,t$ with $a\leq s<t\leq b$ and $\abs{t-s}\leq \delta$, it holds that
  \[
\abs{E(t)-E(s)}\leq 2\epsilon
\]
and
\[
\sup_{\sigma\in [a,b]}\int\limits_{\sigma}^{\min\{\sigma+\delta,b\}}\int_\Omega\abs{\dot{f}(\tau)} \cdot \abs{u(\tau)}\dd x\dd \tau\leq \epsilon.
\]
Take $a = t_0<t_1<\cdots<t_N = b$ such that $t_{j+1}-t_j\leq \delta$ for all $j =0,\ldots,N-1$, and such that that we can apply Lemma~\ref{lem:key1} for $t_j,t_{j+1}$. We abbreviate
\[
  u_j := u(t_j),  \qquad j = 0,\ldots,N.
\]
Then, for all $j\in \{1,..,N\}$ we have by~\eqref{eq:subenerg} that
\begin{align}
\label{eq:Rdelta}
c_1 \norm{u_{j+1}-u_j}_{\Lrm^1} \leq \int_{\Omega}R_1(u_{j+1}-u_j)\dd x\leq 3\epsilon.
\end{align}
Applying~\eqref{eq:stab} at $t_j, t_{j+1}$ with $\psi := u_{j+1}-u_j$, we find (like in the proof of Lemma~\ref{lem:key1}),
\begin{align*}
&E(t_j)-E(t_{j+1})-\int_{t_j}^{t_{j+1}}\dprb{\dot{f}(\tau), u(\tau)} \dd \tau \\
&\quad= \int_\Omega \frac{1}{2} \bigl( \nabla(u_j+u_{j+1})\cdot \nabla(u_j-u_{j+1}) \bigr) + W_0(u_j)-W_0(u_{j+1}) \dd x
\\
&\quad\quad - \dprb{f(t_j), u_j} + \dprb{f(t_{j+1}), u_{j+1}} - \int_{t_j}^{t_{j+1}}\dprb{\dot{f}(\tau), u(\tau)} \dd \tau\\
&\leq \Rcal(u_{j+1}-u_j)
\\
&\quad\quad + \int_\Omega  W_0(u_j)-W_0(u_{j+1}) -\frac{1}{2}\Big(\DD W_0(u_j)+\DD W_0(u_{j+1}) \Big)\cdot(u_j-u_{j+1})\dd x
\\
&\quad\quad + \dprBB{f(t_{j+1})-f(t_j), \frac{u_j + u_{j+1}}{2}} - \int_{t_j}^{t_{j+1}}\dprb{\dot{f}(\tau), u(\tau)} \dd \tau \\
&= \Rcal(u_{j+1}-u_j)
\\
&\quad\quad + \int_\Omega  \int_0^1\DD W_0\big(u_{j+1}+\gamma(u_j-u_{j+1})\big) \cdot (u_j-u_{j+1})\dd \gamma \dd x
\\
&\quad\quad -\int_\Omega\frac{1}{2}\Big(\DD W_0(u_j)+\DD W_0(u_{j+1})\Big)\cdot(u_j-u_{j+1})\dd x
\\
&\quad\quad + \int_{t_j}^{t_{j+1}}\dprBB{ \dot{f}(\tau), \frac{u_j + u_{j+1}}{2}-u(\tau) } \dd \tau. 
\end{align*}
Using the fundamental theorem of calculus, we find
\begin{align*}
&E(t_j)-E(t_{j+1})-\int_{t_j}^{t_{j+1}}\dprb{\dot{f}(\tau), u(\tau)} \dd \tau \\
&\quad= \Rcal(u_{j+1}-u_j)
\\
&\quad\quad + \frac{1}{2} \int_\Omega  \int_0^1\bigg(\int_0^1 \DD^2 W_0\Big(u_j+\sigma\big(u_{j+1}+\gamma(u_j-u_{j+1})-u_j\big)\Big)(\gamma-1)\dd \sigma
\\
&\quad\quad\qquad\qquad\quad +\int_0^1 \DD^2 W_0\Big(u_{j+1}+\sigma\big(\gamma(u_j-u_{j+1})\big)\Big)\gamma \dd\sigma\bigg)
\\
&\quad\quad\qquad\qquad\;\; \times [u_j-u_{j+1}, u_j-u_{j+1}]\dd \gamma\dd x
\\
&\quad\quad + \int_{t_j}^{t_{j+1}}\dprBB{ \dot{f}(\tau), \frac{u_j + u_{j+1}}{2}-u(\tau) } \dd \tau 
\\
&\quad=: \Rcal(u_{j+1}-u_j)+\mathrm{(I)}_j+\mathrm{(II)}_j.
\end{align*}
We estimate $\mathrm{(I)}_j$ using the assumed uniform $\Lrm^\infty$-bounds on $u$ (by Sobolev embedding), the uniform boundedness of $\DD^2 W_0$ on bounded sets (see above), and Poincar\'e--Friedrichs's inequality, as follows:
\begin{align*}
\mathrm{(I)}_j &:=\frac{1}{2} \int_\Omega  \int_0^1\bigg(\int_0^1 \DD^2 W_0\Big(u_j+\sigma\big(u_{j+1}+\gamma(u_j-u_{j+1})-u_j\big)\Big)(\gamma-1)\dd \sigma
\\
&\quad\qquad\qquad\quad +\int_0^1 \DD^2 W_0\Big(u_{j+1}+\sigma\big(\gamma(u_j-u_{j+1})\big)\Big)\gamma \dd\sigma\bigg)
\\
&\quad\qquad\qquad\;\; \times [u_j-u_{j+1}, u_j-u_{j+1}]\dd \gamma\dd x
\\
&\phantom{:}\leq C\norm{\nabla u_{j+1} - \nabla u_j}_{\Lrm^2}^2.
\end{align*}
We estimate further using Lemma~\ref{lem:key1} and~\eqref{eq:Rdelta},
\begin{align*}
\mathrm{(I)}_j
&\leq C \biggl( \norm{u(t_{j+1})-u(t_j)}_{\Lrm^1}^2 + \norm{\dot{f}}_{\Lrm^\infty([t_j,t_{j+1}]\times\Omega)} \cdot \int_{t_j}^{t_{j+1}}\norm{u(t_{j+1})-u(\tau)}_{\Lrm^1} \dd \tau \biggr)
\\
&\leq C\epsilon \Rcal_1(u_{j+1}-u_j)+C\norm{\dot{f}}_{\Lrm^\infty([t_j,t_{j+1}]\times\Omega)} \cdot \int_{t_j}^{t_{j+1}}\norm{u_{j+1}-u(\tau)}_{\Lrm^1} \dd \tau.
\end{align*}
Furthermore, using the uniform $\Lrm^\infty$-bounds on $\dot{f}$,
\[
\mathrm{(II)}_j\leq C \norm{\dot{f}}_{\Lrm^\infty([t_j,t_{j+1}]\times\Omega)} \cdot \int_{t_j}^{t_{j+1}}\normBB{\frac{u_j + u_{j+1}}{2}-u(\tau)}_{\Lrm^1} \dd \tau.
\]
Combining all the above estimates, we arrive at
\begin{align*}
&E(a)-E(b)-\int_a^b\dprb{\dot{f}(\tau), u(\tau)} \dd \tau\\
&\quad 
=\sum_{j=0}^{N-1} \biggl( E(t_j)-E(t_{j+1})-\int_{t_j}^{t_{j+1}}\dprb{\dot{f}(\tau), u(\tau)} \dd \tau \biggr)\\
&\quad \leq (1+C\epsilon)\sum_{j=0}^{N-1} \Rcal_1(u_{j+1}-u_j)
\\
&\quad\quad + C\sum_{j=0}^{N-1} \biggl( \norm{\dot{f}}_{\Lrm^\infty([t_j,t_{j+1}]\times\Omega)} \cdot \int_{t_j}^{t_{j+1}}\norm{u_{j+1}-u(\tau)}_{\Lrm^1} + \norm{u(\tau)-u_j}_{\Lrm^1} \dd \tau \biggr). 
\end{align*}
Using the assumptions on $\Rcal_1$, we proceed by estimating
\begin{align*}
&E(a)-E(b)-\int_a^b\dprb{\dot{f}(\tau), u(\tau)} \dd \tau\\
&\quad
\leq  (1+C\epsilon)\sum_{j=0}^{N-1} \Rcal_1(u_{j+1}-u_j)
\\
&\quad\quad +C\sum_{j=0}^{N-1} \biggl( \norm{\dot{f}}_{\Lrm^\infty([t_j,t_{j+1}]\times\Omega)} \cdot \int_{t_j}^{t_{j+1}}\Rcal_1(u_{j+1}-u(\tau))+\Rcal_1(u(\tau)-u_j) \dd \tau \biggr)
\\
&\quad\leq  (1+C\epsilon)\sum_{j=0}^{N-1} \Var_{\Rcal_1}(u;[t_j,t_{j+1}])
\\
&\quad\quad + C \sum_{j=0}^{N-1} \Bigl( \abs{t_{j+1}-t_j} \cdot \norm{\dot{f}}_{\Lrm^\infty([t_j,t_{j+1}]\times\Omega)} \cdot \Var_{\Rcal_1}(u;[t_j,t_{j+1}]) \Bigr)
\\
&\quad \leq (1+C\epsilon+C\delta\norm{\dot{f}}_{\Lrm^\infty([a,b]\times\Omega;\R^m)})\Var_{\Rcal_1}(u;[a,b]).
\end{align*}
The result follows since we may assume that $C\delta\norm{\dot{f}}_{\Lrm^\infty([a,b]\times\Omega;\R^m)}\leq\epsilon$.
\end{proof}

As a direct consequence we find the energy balance on jump-free intervals:

\begin{corollary}
\label{cor:energ_cont}
Assume that $\nabla^2 u \in \Lrm^\infty(0,T;\Lrm^2(\Omega;\R^{m \times d \times d}))$. Let the energy process $E(t) := \Ecal(t,u(t))$ have no jump on $[a,b]\subset [0,T]$. Then,
\[
\Var_{\Rcal_1}(u;[a,b])=E(a)-E(b)-\int_a^b\dprb{\dot{f}(\tau), u(\tau)} \dd \tau.
\]
\end{corollary}

\subsection{Strong solutions on large portions of the time interval}

Let $\nu$ be a non-negative finite Radon measure on $\Rbb$. We introduce the following maximal function:
\[
\Mcal\nu(t):=\sup_{h>0}\frac{\nu((t-h,t+h))}{2h},  \qquad t \in \R.
\]
It is clear that if $t$ is a singular point of the measure $\nu$, then
$\Mcal\nu(t)=\infty$.

The following is a standard result for such a maximal function:

\begin{lemma}
\label{lem:vitali}
Let $\nu$ be a non-negative finite Radon measure on $\Rbb$ and let $\eta > 0$. Then there exist at most countably many pairwise disjoint non-empty intervals $I_k$ such that 
\[
  \setb{ t }{ \Mcal\nu(t)>\eta }\subset \bigcup_k I_k  \qquad\text{and}\qquad
  \absb{\setb{ t }{ \Mcal\nu(t)>\eta }}\leq \sum_k\abs{I_k}\leq \frac{3}{\eta}\nu(\Rbb).
\]
Moreover, $\nu$ has a Lipschitz-continuous density on the interior of $\set{ t }{ \Mcal\nu(t)\leq\eta }$  with  Lipschitz constant that is similar to $\eta$. 
\end{lemma}

\begin{proof}
Let $t \in \R$ be such that $\Mcal\nu(t)>\eta$. Then, there exists a $h_t > 0$ such that
\[
2 h_t \eta\leq \nu((t-h_t,t+h_t)). 
\]
We may apply Vitali's covering lemma, see~\cite[Lemma~7.3]{Rudin86book}, which entails that there exists a sequence of disjoint intervals $B(t_k,h_{t_k}) =(t_k-h_{t_k},t_k+h_{t_k})$ such that
\[
  \setb{ t }{ \Mcal\nu(t)>\eta } \subset \bigcup_k \, B(t_k,3h_{t_k}).
\]
This implies
\begin{align*}
\eta\absb{\setb{ t }{ \Mcal\nu(t)>\eta }} \leq \sum_k\eta\abs{B(t_k,3h_{t_k})}=3\sum_k2h_{t_k}\eta
\leq  3\sum_k\nu(B(t_k,h_{t_k}))\leq 3\nu(\Rbb).
\end{align*}
We let the $I_k$ be the connected components of the union of the above constructed intervals $B(t_k,3h_{t_k})$. 
The size estimates follow from the estimate above. The Lipschitz continuity follows by the fact that the interior of $\set{ t }{ \Mcal\nu(t)\leq\eta }$ (if it is non-empty) is contained in the absolute continuous part of $\nu$ and the Acerbi--Fusco lemma~\cite{AcerbiFusco84} (also see~\cite[Lemma 1.68]{MalyZiemer97book}).
\end{proof}

We are now in a position to show the main result of this section.

\begin{proposition}
\label{prop:goodset}
Let $u\in \Lrm^\infty((0,T) \times \Omega;\R^m)\cap \Lrm^\infty(0,T;\Wrm^{2,2}(\Omega;\R^m))$ satisfy~\eqref{eq:subenerg},~\eqref{eq:stab} in $[0,T]$. Then, $u \in \BV(0,T;\Lrm^1(\Omega;\R^m))$, $E \in \BV(0,T)$, and the process $u$ is a strong solution for almost every $t\in (0,T)$, i.e.,~\eqref{eq:strongsol_ineq} holds  for almost every $t\in (0,T)$.
\end{proposition}

\begin{proof}
We first observe that the assumed regularity on $u$ implies that $E(t) \leq C < \infty$ for almost all $t \in [0,T]$ (where $C > 0$ does not depend on $t$); moreover, this holds at $t = 0$ by Assumption~\ref{as:u0}. Then, via~\eqref{eq:subenerg}  we obtain that $\Var_{\Rcal_1}(u;[0,T]) < \infty$, hence $u \in \BV(0,T;\Lrm^1(\Omega;\R^m))$. Moreover, the map
\[
  \omega(t) := E(0) - E(t) - \int_0^t \dprb{\dot{f}(\tau), u(\tau)} \dd \tau
\]
satisfies, by Proposition~\ref{prop:E_contpoint}, for all non-jump points $a,b \in [0,T]$ that
\[
  \Var_{\Rcal_1}(u;[a,b]) \leq \omega(b) - \omega(a) \leq C \Var_{\Rcal_1}(u;[a,b]).
\]
This implies in particular that $\omega \in \BV(0,T)$ and then also $E \in \BV(0,T)$.

Define the measure $\nu$ on $[0,T]$ to be the BV-derivative of $\omega$; in particular, for all non-jump points $a,b \in [0,T]$,
\[
 \nu([a,b]) = \omega(b) - \omega(a) = E(a)-E(b) - \int_a^b \dprb{\dot{f}(\tau), u(\tau)} \dd \tau.
\]
We use the covering from Lemma~\ref{lem:vitali} on $\nu$ to show that for every $\epsilon>0$ there exists a Borel set $I_\epsilon \subset [0,T]$ such that $\abs{[0,T]\setminus I_\epsilon}\leq \epsilon$ and at all $t\in I_\epsilon$, the process $u$ is a strong solution at $t$, i.e.,~\eqref{eq:strongsol_ineq} holds at $t$. Indeed, by Lemma~\ref{lem:vitali} it is possible to choose $\eta$ large enough such that
\[
  \absb{[0,T]\setminus \set{t}{\Mcal\nu(t)\leq \eta}}\leq \epsilon.
\]
We define
   \[
   I_\epsilon:=\setb{t}{\Mcal\nu(t)\leq \eta, \; \nabla^2u(t)\in \Lrm^2(\Omega;\R^{m \times d \times d})\text{ and~\eqref{eq:subenerg},~\eqref{eq:stab} is satisfied at $t$, a.e.\ $s$} }.
   \]  Let us next consider a point in $t\in I_\epsilon$. 
By definition of $\Mcal\nu$ and~\eqref{eq:subenerg}, we find that for almost all $h>0$,
\begin{align*}
c_1 \frac{\norm{u(t+h)-u(t-h)}_{\Lrm^1}}{2h}
&\leq \frac{\Var_{\Rcal_1}(u;[t-h,t+h])}{2h} \\
&\leq \frac{ E(t-h)-E(t+h)}{2h} - \dashint_{t-h}^{t+h}\dprb{\dot{f}(\tau), u(\tau)} \dd \tau \\
&\leq  \eta .
\end{align*}
Similarly, we have for almost all $h > 0$ and almost all $\tau \in [t-h,t+h]$ that (using~\eqref{eq:subenerg} again in $[t-h,\tau]$)
\begin{align*}
c_1 \frac{\norm{u(t+h)-u(\tau)}_{\Lrm^1}}{2h} 
&\leq \frac{\Var_{\Rcal_1}(u;[\tau,t+h])}{2h} \\
&\leq \frac{1}{2h} \biggl( E(\tau)-E(t+h) - \int_{\tau}^{t+h}\dprb{\dot{f}(\tau'), u(\tau')} \dd \tau' \biggr) \\
&\leq \frac{1}{2h} \biggl( E(t-h)-E(t+h) - \int_{t-h}^{t+h}\dprb{\dot{f}(\tau'), u(\tau')} \dd \tau' \biggr) \\
&\leq \eta.
\end{align*}
Then, Lemma~\ref{lem:key1} implies for almost every $h>0$,
\begin{align*}
&\frac{\norm{\nabla u(t+h) - \nabla u(t-h)}_{\Lrm^2}^2}{(2h)^2}\\
&\quad \leq C\biggl(\frac{\norm{u(t+h) - u(t-h)}_{\Lrm^1}}{2h}\biggr)^2 + C\norm{\dot{f}}_{\Lrm^\infty([0,T] \times \Omega)} \cdot \dashint_{t-h}^{t+h} \frac{\norm{u(t+h)-u(\tau)}_{\Lrm^1}}{2h}\dd \tau\\
&\quad\leq C\eta^2+C\norm{\dot{f}}_{\Lrm^\infty([0,T] \times \Omega)} 
\eta.
\end{align*}
This implies that $\nabla \dot{u}(t)\in \Lrm^2(\Omega;\R^m)$ and in particular that $u$ is a strong solution at $t$ due to Lemma~\ref{lem:strongsol}. Since $\abs{[0,T]\setminus \bigcup_{N\in \Nbb}\set{t}{\Mcal\nu(t)\leq N}}=0$ the proof is complete.
\end{proof}

\begin{remark}[Uniqueness]
\label{rem:uniq}
We note in passing that the estimate above also improves the uniqueness result obtained in~\cite{RindlerSchwarzacherSuli17}. Indeed, it implies that the uniqueness class for convex energies can be extended to weak solutions in $\BV(0,T;\Lrm^1(\Omega;\R^m))$ satisfying the stability estimate and an energy inequality.
\end{remark}

\section{Rate-dependent evolution}
\label{sc:RDevol}

In this section we will construct solutions $u_\lambda$ to the rate-dependent system~\eqref{eq:PDE_lambda} for any $\lambda > 0$ and establish several estimates ($\lambda$-uniform and with quantitative dependence on $\lambda$), which will be needed in the next section to pass to the limit $\lambda \todown 0$. Some of the arguments in this section are similar in spirit to the work~\cite{MielkeZelik14}.

\subsection{Regularization of $R_1$}

In order to construct $u_\lambda$, we need to regularize $R_1$, which has a kink at the origin. This introduces another parameter $\delta > 0$ and we need $\delta$-uniform estimates to let $\delta \todown 0$.

We set $K:=\set{z \in \R^m}{R_1(z)\leq 1}$ and define $K_\delta:=\bigcup_{z\in K}B(z,\eta)$ with $\eta > 0$ chosen such that
\[
  K_\delta \subset \Bigl(1+\frac{\delta}{c_2}\Bigr)K.
\]
Then, $K_\delta$ is convex, contains $0$ as an interior point, and has a smooth boundary. As is shown in Example~\ref{ex:fric}, we can associate with $K_\delta$ a $1$-homogeneous convex friction potential $\hat{R}_1^\delta$ that is smooth away from the origin and satisfies
\[
  \Bigl(1+\frac{\delta}{c_2}\Bigr)^{-1} R_1(z) \leq \hat{R}_1^\delta(z) \leq R_1(z).
\]
Hence, also using that $\hat{R}_1^\delta(z) \leq R_1(z) \leq c_2\abs{z}$ (where $c_2$ is the upper growth bound of $R_1$), we get
\[
  R_1(z) - \delta \abs{z} \leq \hat{R}_1^\delta(z) \leq R_1(z).
\]
To smooth $\hat{R}_1^\delta$ it around the origin, we choose for $\delta > 0$ a convex function $\phi_\delta\in \Crm^\infty([0,\infty),[0,\infty))$ such that for $s\in [0,\infty)$,
\begin{align*}
[s-2\delta]^{+} \leq \phi_\delta(s) &\leq [s-\delta]^{+},\\
0\leq \phi_\delta'(s) &\leq 1,\\
0\leq \phi_\delta''(s) &\leq C
\end{align*}
for some constant $C > 1$ and where $[\frarg]^{+}$ denotes the positive part.
Then we define
\[
  R_1^\delta(z):=\phi_\delta(\hat{R}_1^\delta(z)),  \qquad z \in \R^m,
\]
for which it holds that
\begin{align}
\label{eq:deltaR} 
 [c_1\abs{z}-2\delta]^{+}\leq  R^\delta_1(z) \leq  c_2\abs{z},  \qquad z \in \R^m.
\end{align}
Recall the following basic fact about subdifferentials of $1$-homogeneous functions (which is standard and easy to see):
\[
  R_1(z) = \dpr{\sigma,z}  \qquad\text{for all $z \in \R^m$ and all $\sigma \in \partial R_1(z)$.}
\]
We may then calculate
\[
\DD R_1^\delta(z)\cdot z- R^\delta_1(z)= \phi'(\hat{R}_1^\delta(z))\hat{R}^\delta_1(z)-\phi(\hat{R}_1^\delta(z)).
\]
The convexity and the other assumptions on $\phi_\delta$ imply for $s \geq 0$ that
\[
0 = -\phi_\delta(0) \leq \phi_\delta'(s)s-\phi_\delta(s)  \leq s- [s-2\delta]^{+}\leq 2\delta. 
\]
Hence,
\begin{equation} \label{eq:R1delta_formula}
\absb{\DD R_1^\delta(z)\cdot z- R^\delta_1(z)} \leq 2\delta,  \qquad z \in \R^m.
\end{equation}

We now take a smooth approximation $(u_{0,\delta}) \subset \Crm^\infty(\cl{\Omega};\R^m)$ of the initial value $u_0$, that is, $u_{0,\delta}\to u_0$ in $(\Wrm^{1,2}_0 \cap \Lrm^q)(\Omega;\R^m)$ as $\delta \todown 0$ (cf.~\ref{as:u0}) and consider the following PDE for $u_{\lambda,\delta} \colon [0,T]\times \Omega\to \R^m$:
\begin{equation} \label{eq:PDE_ulambdadelta}
\hspace{-10pt}
\left\{\begin{aligned}
 \lambda  \dot{u}_{\lambda,\delta} + \DD R^\delta_1(\dot{u}_{\lambda,\delta}) - \Delta u_{\lambda,\delta} + \DD W_0(u_{\lambda,\delta}) &= f  &&\text{in $(0,T) \times \Omega$,} \\
 u_{\lambda,\delta}(t)|_{\partial\Omega}&=0 &&\text{for $t \in (0,T)$,}
 \\ 
 u_{\lambda,\delta}(0) &=u_{0,\delta}  &&\text{in $\Omega$.}
\end{aligned}\right.
\end{equation}
\begin{lemma}
\label{lem:deltaex} 
There exists a strong solution
\[
  u_{\lambda,\delta} \in \Wrm^{2,2}(0,T;\Lrm^2(\Omega;\R^m))\cap \Wrm^{1,\infty}(0,T;\Wrm^{1,2}(\Omega;\R^m))
\]
of~\eqref{eq:PDE_ulambdadelta}. 
\end{lemma}

\begin{proof}
Since the existence theory for~\eqref{eq:PDE_ulambdadelta} is standard, we only give a sketch of the proof. The idea is to implement a fixed-point argument over a linear PDE. This can be achieved by formally differentiating the system~\eqref{eq:PDE_ulambdadelta} to get
\begin{equation} \label{eq:PDE_ulambdadelta_dot}
\hspace{-20pt}
\left\{\begin{aligned}
 \lambda  \ddot{u}_{\lambda,\delta} + \DD^2 R^\delta_1(\dot{u}_{\lambda,\delta})\cdot\ddot{u}_{\lambda,\delta} - \Delta \dot{u}_{\lambda,\delta} + \DD^2 W_0(u_{\lambda,\delta})\cdot \dot{u}_{\lambda,\delta} &= \dot{f}  &&\text{in $(0,T) \times \Omega$,}\\
  \dot{u}_{\lambda,\delta}(t)|_{\partial\Omega}&=0  &&\text{for $t \in (0,T)$,}
 \\ 
 \dot{u}_{\lambda,\delta}(0) =\dot{u}_{0,\delta},\quad  {u}_{\lambda,\delta}(0)&={u}_{0,\delta}  &&\text{in $\Omega$,}
\end{aligned}\right.
\end{equation}
where $\dot{u}_{0,\delta}\in \Wrm^{1,2}_0(\Omega;\R^m)$ is defined via the equation
\begin{equation} \label{eq:initialeq}
\lambda  \dot{u}_{0,\delta} + \DD R^\delta_1(\dot{u}_{0,\delta}) - \Delta u_{0,\delta} + \DD W_0(u_{0,\delta}) = f(0) \qquad \text{in $\Omega$.}
\end{equation}
The mapping
\[
\Phi \colon z\mapsto \lambda  z + \DD R^\delta_1(z) ,  \qquad z \in \R^m,
\]
is smooth and invertible. Hence, $\dot{u}_{0,\delta}\in \Wrm^{1,2}(\Omega;\R^m)$ with bounds depending on $\delta,\lambda$. For later use we also record the following estimate:
\begin{equation} \label{eq:dotu0d_L2}
  \lambda \norm{\dot{u}_{0,\delta}}_{\Lrm^2}^2 \lesssim \norm{\Delta u_{0,\delta} - \DD W_0(u_{0,\delta}) + f(0)}_{\Lrm^2}^2,
\end{equation}
which follows by testing the equation~\eqref{eq:initialeq} with $\dot{u}_{0,\delta}$, using Young's inequality, and a standard absorption argument.

A strong solution to~\eqref{eq:PDE_ulambdadelta_dot} (and hence also for~\eqref{eq:PDE_ulambdadelta}) can be gained by first solving for a given $v\in \Wrm^{1,1}(0,T;\Lrm^1(\Omega;\R^m))\cap \Lrm^\infty([0,T]\times\Omega)$ the following system for $w \colon [0,T] \times \Omega \to \R^m$:
\[
\left\{\begin{aligned}
 \lambda  \ddot{w} + \DD^2 R^\delta_1(\dot{w})\cdot\ddot{w} - \Delta \dot{w} + \DD^2 W_0(v)\cdot \dot{w} &= \dot{f}  &&\text{in $(0,T) \times \Omega$,} \\
  \dot{w}_{\lambda,\delta}(t)|_{\partial\Omega}&=0  &&\text{for $t \in (0,T)$,} \\ 
 \dot{w}(0)=\dot{u}_{0,\delta},\quad  w(0)&={u}_{0,\delta}  &&\text{in $\Omega$.}
\end{aligned}\right.
\]
Then use the Schauder fixed-point theorem to obtain the solution. The natural a-priori estimates can be achieved by taking $\ddot{w}$ as a test function, which gives the following a-priori estimates:
\[
\int_{0}^T\norm{\ddot{u}_{\lambda,\delta}(t)}_{\Lrm^2}^2\dd t+\sup_{t\in [0,T]}\norm{\nabla\dot{u}_{\lambda,\delta}(t)}_{\Lrm^2}^2\leq C_{\lambda,\delta},
\]
from which the assertions follow.
\end{proof}

\subsection{Uniform estimates in $\delta$ and $\lambda$.}

Next we aim to estimate the regularity of a process $u_{\lambda,\delta}$ solving~\eqref{eq:PDE_ulambdadelta}.

\begin{lemma}
\label{lem:apri}
We have the bounds
\begin{align}
  \hspace{-10pt} \lambda \norm{\dot{u}_{\lambda,\delta}}_{\Lrm^2_{t,x}}^2 + \norm{\dot{u}_{\lambda,\delta}}_{\Lrm^1_{t,x}} +\sup_{t\in [0,T]} \Bigl( \norm{u_{\lambda,\delta}(t)}_{\Wrm^{1,2}} + \norm{u_{\lambda,\delta}(t)}_{\Lrm^q} \Bigr) &\leq C,     \label{eq:dotulambdadelta_RD_apriori}    
  \\
  \lambda \int_0^T t \norm{\nabla \dot{u}_{\lambda,\delta}(t)}_{\Lrm^2}^2\dd t &\leq C,    \label{eq:dotdotulambdadelta_RI_apriori_t}
\end{align}
and for all $r \in [1,2^*)$ and $\sigma \in (0,T]$,
\begin{align}
\sup_{t\in[\sigma,T]} \lambda^{2} \norm{\dot{u}_{\lambda,\delta}(t)}_{\Lrm^2}^{2}  + \lambda^2 \int_\sigma^T\norm{\nabla \dot{u}_{\lambda,\delta}(t)}_{\Lrm^2}^2\dd t  &\leq C \biggl(1+\frac{\lambda}{\sigma}\biggr),           \label{eq:dotdotulambdadelta_RI_apriori} \\
  \int_\sigma^T \norm{u_{\lambda,\delta}(t)}_{\Wrm^{2,r}}^2\dd t
     + \sup_{t\in[\sigma,T]} \norm{u_{\lambda,\delta}(t)}_{\Wrm^{2,2}}^{2}
&\leq C \biggl(1+\frac{\lambda}{\sigma}\biggr)^{\max\{q-1,1\}}.  \label{eq:ulambdadelta_Wp2}
\end{align}
If $d = 3$, then~\eqref{eq:ulambdadelta_Wp2} also holds for $r = 2^* = 6$. Here, the constants $C>0$ depend on $r$, $\norm{{f}}_{\Wrm^{1,\infty}(\Lrm^\infty)}$, $\norm{u_0}_{\Wrm^{1,2}}$, but are independent of $\delta$ and $\lambda$ (assuming $\delta, \lambda \in (0,1)$). Moreover, $u_{\lambda,\delta}$ satisfies the energy balance
\begin{align}\hspace{-10pt}
&\Ecal(t,u_{\lambda,\delta}(t)) - \Ecal(0,u_{\lambda,\delta}(0))\notag\\
&\quad = - \int_0^t \lambda \norm{\dot{u}_{\lambda,\delta}(\tau)}_{\Lrm^2}^2 + \Rcal_1(\dot{u}_{\lambda,\delta}(\tau)) \dd \tau -\int_0^t\dprb{\dot{f}(\tau),u_{\lambda,\delta}(\tau)}\dd \tau  \label{eq:ulambdadelta_energy}
\end{align}
for almost all $t \in [0,T]$.

If we assume additionally that
\begin{equation} \label{eq:new-ass}
\left\{\begin{aligned}
  &\Delta u_0 -\DD W_0(u_0)\in \Lrm^2(\Omega,\R^m), \\
  &\Delta u_{0,\delta} - \DD W_0(u_{0,\delta}) \to \Delta u_0 -\DD W_0(u_0) \quad\text{in $\Lrm^2$ as $\delta \todown 0$,}
\end{aligned}\right.
\end{equation}
then 
\begin{align}
\sup_{t\in[0,T]} \lambda^2 \norm{\dot{u}_{\lambda,\delta}(t)}_{\Lrm^2}^2 + \lambda \int_0^T \norm{\nabla \dot{u}_{\lambda,\delta}(t)}_{\Lrm^2}^2 \dd t
&\leq C \bigl(1 + \lambda Z_0 \bigr), \label{eq:ulambdadelta_Linfty1}
\\
\int_0^T \norm{u_{\lambda,\delta}(t)}_{\Wrm^{2,r}}^2\dd t
     + \sup_{t\in[0,T]} \norm{u_{\lambda,\delta}(t)}_{\Wrm^{2,2}}^{2}
&\leq C \bigl(1 + \lambda Z_0 \bigr)^{\max\{q-1,1\}}, \label{eq:ulambdadelta_Linfty2}
\end{align}
where $Z_0 := \norm{\Delta u_0 - \DD W_0(u_0) + f(0)}_{\Lrm^2}^2$.
\end{lemma}

We remark that there is a smoothing effect in the equation, giving improved regularity away from the starting time $t = 0$.

\begin{proof}
We multiply the system~\eqref{eq:PDE_ulambdadelta} with various test functions and collect the resulting estimates.
 
\proofstep{Testing with $\dot{u}_{\lambda,\delta}$.}
Multiplying~\eqref{eq:PDE_ulambdadelta} with $\dot{u}_{\lambda,\delta}(t)$ and integrating over $\Omega$ gives
\[
\lambda \norm{\dot{u}_{\lambda,\delta}}_{\Lrm^2}^2 + \int_\Omega \DD R_1^\delta(\dot{u}_{\lambda,\delta}) \dot{u}_{\lambda,\delta} \dd x + \frac{\di}{\di t} \bigl(\Wcal(u_{\lambda,\delta})-\dprb{f,u_{\lambda,\delta}}\bigr)=-\dprb{\dot{f},u_{\lambda,\delta}},
\]
where we recall the defintion of $\Wcal$ in~\eqref{eq:Wcal}. Integrating over a time interval $[0,t] \subset [0,T]$, this implies the energy equality
\begin{align*}
&\int_0^t \lambda \norm{\dot{u}_{\lambda,\delta}(\tau)}_{\Lrm^2}^2  \dd \tau + \int_0^t \int_\Omega \DD R_1^\delta(\dot{u}_{\lambda,\delta}(\tau)) \dot{u}_{\lambda,\delta}(\tau) \dd x \dd \tau  \notag\\
&\quad\quad +\Wcal(u_{\lambda,\delta}(t))-\dprb{f(t),u_{\lambda,\delta}(t)} \notag\\
&\quad = \Wcal(u_{\lambda,\delta}(0)) - \dprb{f(0),u_{\lambda,\delta}(0)} -\int_0^t\dprb{\dot{f}(\tau),u_{\lambda,\delta}(\tau)}\dd \tau.  
\end{align*}
Hence, using~\eqref{eq:deltaR},~\eqref{eq:R1delta_formula}, Assumption~\ref{as:W}, and taking the supremum over all $t \in [0,T]$,
\begin{align*}
&\lambda \norm{\dot{u}_{\lambda,\delta}}_{\Lrm^2_{t,x}}^2 +c_1 \norm{\dot{u}_{\lambda,\delta}}_{\Lrm^1_{t,x}}+\sup_{t\in [0,T]}\Big(\norm{\nabla u_{\lambda,\delta}(t)}_{\Lrm^2}^2+\norm{ u_{\lambda,\delta}(t)}_{\Lrm^q}^q\Big)  \notag\\
&\quad
\lesssim \norm{\nabla u_{0,\delta}}_{\Lrm^2}^2+\norm{u_{0,\delta}}_{\Lrm^q}^q 
+\Bigl( \norm{\dot{f}}_{\Lrm^1(\Lrm^2)}
+\sup_{t\in [0,T]}\norm{f(t)}_{\Lrm^2} \Bigr) \norm{u_{\lambda,\delta}}_{\Lrm^\infty(\Lrm^2)} + \delta . 
\end{align*}
Then, using the Young and Poincar\'e--Friedrichs inequalities followed by an absorption of $\norm{u_{\lambda,\delta}}_{\Lrm^\infty(\Lrm^2)}^2$ into the left-hand side, we arrive at
\begin{align*}
&\lambda \norm{\dot{u}_{\lambda,\delta}}_{\Lrm^2_{t,x}}^2 + \norm{\dot{u}_{\lambda,\delta}}_{\Lrm^1_{t,x}}+ \sup_{t\in [0,T]}\Big(\norm{\nabla u_{\lambda,\delta}(t)}_{\Lrm^2}^2+\norm{ u_{\lambda,\delta}(t)}_{\Lrm^q}^q\Big)  \\
&\quad
\lesssim \norm{\nabla u_{0,\delta}}_{\Lrm^2}^2+\norm{u_{0,\delta}}_{\Lrm^q}^q
+\norm{\dot{f}}_{\Lrm^1(\Lrm^2)}^2
+\sup_{t\in [0,T]}\norm{f(t)}_{\Lrm^2}^2 + \delta.
\end{align*}
This concludes the proof of estimate~\eqref{eq:dotulambdadelta_RD_apriori}.

\proofstep{Testing with $\ddot{u}_{\lambda,\delta}$.}
Let $\eta\in \Crm^2_0((0,T);[0,\infty))$. Using $-\partial_t(\eta \dot{u}_{\lambda,\delta})$ as a test function in~\eqref{eq:PDE_ulambdadelta}, we find
\begin{align*}
\int_0^T\int_\Omega{\dot{f}}{\dot{u}_{\lambda,\delta} \eta} \dd t &=\int_0^T\int_\Omega \lambda \ddot{u}_{\lambda,\delta} \dot{u}_{\lambda,\delta} \eta - \DD R_1^\delta(\dot{u}_{\lambda,\delta}) \dot{u}_{\lambda,\delta}\partial_t\eta - \partial_t( R^\delta_1(\dot{u}_{\lambda,\delta}))\eta \dd x\dd t
\\
&\quad +\int_0^T\int_\Omega 
\abs{\nabla \dot{u}_{\lambda,\delta}}^2\eta+\DD^2 W_0(u_{\lambda,\delta})[\dot{u}_{\lambda,\delta},\dot{u}_{\lambda,\delta}]\eta\dd x\dd t
\\
&=\int_0^T\int_\Omega  -\Big(\frac{\lambda}{2}\abs{\dot{u}_{\lambda,\delta}}^2 + \DD R_1^\delta(\dot{u}_{\lambda,\delta})\dot{u}_{\lambda,\delta} - R^\delta_1(\dot{u}_{\lambda,\delta})\Big)\partial_t\eta \dd x\dd t
\\
&\quad +\int_0^T\int_\Omega 
\abs{\nabla \dot{u}_{\lambda,\delta}}^2\eta+\DD^2 W_0(u_{\lambda,\delta})[\dot{u}_{\lambda,\delta},\dot{u}_{\lambda,\delta}]\eta\dd x\dd t.
\end{align*}
Approximating $\eta(t):=\frac{t}{b}\ONE_{[0,b]}(t)$ (for which $\DD_t \eta = \frac{1}{b}\ONE_{[0,b]} - \delta_b$) smoothly and employing~\eqref{eq:R1delta_formula}, we get
\begin{align*}
  &\int_\Omega \frac{\lambda}{2} \abs{\dot{u}_{\lambda,\delta}(b)}^2\dd x + \int_0^b\int_\Omega \frac{t}{b}\abs{\nabla \dot{u}_{\lambda,\delta}(t)}^2 \dd x\dd t  \\
  &\quad \leq C \delta + \int_0^b \int_\Omega \frac{\lambda}{2b} \abs{\dot{u}_{\lambda,\delta}(t)}^2 \dd x \dd t - \int_0^b \int_\Omega \frac{t}{b} \DD^2W_0(u_{\lambda,\delta}(t))[\dot{u}_{\lambda,\delta}(t),\dot{u}_{\lambda,\delta}(t)] \dd x \dd t \\
  &\quad\quad + \int_0^b \int_\Omega \frac{t}{b} \abs{\dot{f}(t)} \cdot \abs{\dot{u}_{\lambda,\delta}(t)} \dd x \dd t.
\end{align*}
Then, since $-\mu \abs{w}^2 \leq \DD^2 W_0(v)[w,w]$ by assumption~\ref{as:W}, and using Young's inequality,
\begin{align}
&\frac{\lambda}{2} \int_\Omega \abs{\dot{u}_{\lambda,\delta}(b)}^2 \dd x+ \int_0^b\int_\Omega \frac{t}{b}\abs{\nabla \dot{u}_{\lambda,\delta}(t)}^2 \dd x\dd t
\notag\\
&\quad \lesssim 1 + \biggl(1+\frac{\lambda}{b}\biggr) \int_0^b\int_\Omega \abs{\dot{u}_{\lambda,\delta}}^2 \dd x \dd t + \int_0^b\int_\Omega\abs{\dot{f}}^2\dd x\dd t .  \label{eq:TheEst}
\end{align}
Taking $b = T$ in~\eqref{eq:TheEst} and multiplying by $\lambda$, we obtain the estimate
\[
  \lambda \int_0^T\int_\Omega t \abs{\nabla \dot{u}_{\lambda,\delta}(t)}^2 \dd x\dd t \lesssim 1 + \lambda \norm{\dot{u}_{\lambda,\delta}}_{\Lrm^2_{t,x}}^2 + \norm{\dot{f}}_{\Lrm^2_{t,x}}^2.
\]
The right-hand side is uniformly bounded by~\eqref{eq:dotulambdadelta_RD_apriori} and Assumption~\ref{as:f}. From this we deduce that
\[
  \lambda \int_0^T t \norm{\nabla \dot{u}_{\lambda,\delta}(t)}_{\Lrm^2}^2\dd t \leq C,
\]
which is~\eqref{eq:dotdotulambdadelta_RI_apriori_t}. This also immediately implies
\[
  \lambda^2 \int_\sigma^T\norm{\nabla \dot{u}_{\lambda,\delta}(t)}_{\Lrm^2}^2\dd t \leq C \frac{\lambda}{\sigma}.
\]
Moreover, for any $b \geq \sigma$, we multiply~\eqref{eq:TheEst} by $\lambda$ to get
\[
  \lambda^2 \norm{\dot{u}_{\lambda,\delta}(b)}_{\Lrm^2}^2 \lesssim 1 + \biggl(1+\frac{\lambda}{\sigma}\biggr) \lambda \norm{\dot{u}_{\lambda,\delta}}_{\Lrm^2_{t,x}}^2 + \norm{\dot{f}}_{\Lrm^2_{t,x}}^2.
\]
As before, the right-hand side is uniformly bounded by~\eqref{eq:dotulambdadelta_RD_apriori} and Assumption~\ref{as:f}.  We can then take the supremum over all $b \in [\sigma,T]$ to obtain
\[
  \sup_{t\in[\sigma,T]} \lambda^2 \norm{\dot{u}_{\lambda,\delta}(t)}_{\Lrm^2}^2 \leq C \biggl(1+\frac{\lambda}{\sigma}\biggr).
\]
This completes the proof of~\eqref{eq:dotdotulambdadelta_RI_apriori}.

\proofstep{Elliptic regularity.}
In order to prove~\eqref{eq:ulambdadelta_Wp2} we use the fact that the integrability properties of the right-hand side transfer to $\nabla^2u_{\lambda,\delta}$. Indeed, we find
\begin{align*}
 \abs{- \Delta u_{\lambda,\delta} + \DD W_0(u_{\lambda,\delta})} &= \abs{f-\lambda  \dot{u}_{\lambda,\delta} - \DD R^\delta_1(\dot{u}_{\lambda,\delta})} \\
 &\leq C+\abs{f}+\lambda\abs{\dot{u}_{\lambda,\delta}}
\end{align*}
We now use Lemma~\ref{lem:apriori-space} and~\ref{as:f} to find that for $r \in [2,\infty)$ if $d = 2$ or $r \in [2,2^*]$ if $d = 3$ and almost every $t \in [0,T]$,
\begin{align}
\label{eq:ellipticdotu}
\norm{u_{\lambda,\delta}(t)}_{\Wrm^{2,r}}\lesssim 1 + \norm{\lambda \dot{u}_{\lambda,\delta}(t)}_{\Lrm^r}+ \norm{\lambda \dot{u}_{\lambda,\delta}(t)}_{\Lrm^2}^{q-1}.
\end{align}
Now,~\eqref{eq:dotdotulambdadelta_RI_apriori} implies
\[
  \sup_{\tau\in [\sigma,T]}\norm{\lambda \dot{u}_{\lambda,\delta}(\tau)}_{\Lrm^2}^{2q-2} \leq C\biggl(1+\frac{\lambda}{\sigma}\biggr)^{q-1},
\]
and, for $r \in [1,2^*)$ and also $r = 2^*=6$ if $d=3$, via~\eqref{eq:dotdotulambdadelta_RI_apriori},
\[
\int_{\sigma}^T\norm{\lambda \dot{u}_{\lambda,\delta}(\tau)}_{\Lrm^r}^2\dd \tau
\lesssim \lambda^2 \int_{\sigma}^T\norm{\nabla \dot{u}_{\lambda,\delta}(\tau)}_{\Lrm^2}^2\dd \tau
\leq C \biggl(1+\frac{\lambda}{\sigma}\biggr).
\]
Combining these estimates yields the first part of~\eqref{eq:ulambdadelta_Wp2}; the extension to the exponents $r \in [1,2)$ follows by embedding. The second part follows similarly with $r=2$ and using~\eqref{eq:dotdotulambdadelta_RI_apriori}.

\proofstep{Energy balance.}
The energy balance~\eqref{eq:ulambdadelta_energy} follows from the a-priori estimates by multiplying the equation by $\dot{u}_{\lambda,\delta}$ and arguing similarly as in the proof of~\eqref{eq:dotdotulambdadelta_RI_apriori}.

\proofstep{Estimates that are uniform as $t\to 0$.}
If we assume~\eqref{eq:new-ass}, then in the arguments leading to~\eqref{eq:TheEst} we can instead approximate $\eta(t):=\ONE_{[0,b]}(t)$ (for which $\DD_t \eta = \delta_0- \delta_b$) smoothly to find
\begin{align*}
&\lambda \int_\Omega \abs{\dot{u}_{\lambda,\delta}(b)}^2 \dd x+ \int_0^b\int_\Omega \abs{\nabla \dot{u}_{\lambda,\delta}(t)}^2 \dd x\dd t \\
&\quad \lesssim 1 + \lambda \int_\Omega \abs{\dot{u}_{\lambda,\delta}(0)}^2 \dd x + \int_0^b\int_\Omega \abs{\dot{u}_{\lambda,\delta}}^2 \dd x \dd t + \int_0^b\int_\Omega\abs{\dot{f}}^2\dd x\dd t \\
&\quad \lesssim 1 + \norm{\Delta u_{0,\delta} - \DD W_0(u_{0,\delta}) + f(0)}_{\Lrm^2}^2 + \int_0^b\int_\Omega \abs{\dot{u}_{\lambda,\delta}}^2 \dd x \dd t + \int_0^b\int_\Omega\abs{\dot{f}}^2\dd x\dd t,
\end{align*}
where in the last line we used~\eqref{eq:dotu0d_L2}. Then, as before, multiply by $\lambda$, and use~\eqref{eq:dotulambdadelta_RD_apriori} together with Assumption~\ref{as:f}, to get
\[
\lambda^2 \norm{\dot{u}_{\lambda,\delta}(b)}_{\Lrm^2}^2 + \lambda \int_0^b \norm{\nabla \dot{u}_{\lambda,\delta}(t)}_{\Lrm^2}^2 \dd t
\lesssim 1 + \lambda \norm{\Delta u_{0,\delta} - \DD W_0(u_{0,\delta}) + f(0)}_{\Lrm^2}^2.
\]
We can then take the supremum over all $b \in [0,T]$ and employ the convergence $\Delta u_{0,\delta} - \DD W_0(u_{0,\delta}) \to \Delta u_0 -\DD W_0(u_0)$ in $\Lrm^2$ as $\delta \todown 0$, to obtain~\eqref{eq:ulambdadelta_Linfty1}.

Moreover, for $r \in [1,2^*)$ and also $r = 2^*=6$ if $d=3$, we get from~\eqref{eq:ulambdadelta_Linfty1} that
\begin{align*}
\int_0^T\norm{\lambda \dot{u}_{\lambda,\delta}(\tau)}_{\Lrm^r}^2\dd \tau
&\lesssim \lambda^2 \int_0^T\norm{\nabla \dot{u}_{\lambda,\delta}(\tau)}_{\Lrm^2}^2\dd \tau \\
&\lesssim \lambda \Bigl( 1 + \lambda \norm{\Delta u_{0,\delta} - \DD W_0(u_{0,\delta}) + f(0)}_{\Lrm^2}^2 \Bigr)
\end{align*}
and
\[
  \sup_{\tau\in [0,T]}\norm{\lambda \dot{u}_{\lambda,\delta}(\tau)}_{\Lrm^2}^{2q-2}
  \lesssim \Bigl(1 + \lambda \norm{\Delta u_0 - \DD W_0(u_0) + f(0)}_{\Lrm^2}^2\Bigr)^{q-1}.
\]
Via Lemma~\ref{lem:apriori-space} and~\ref{as:f}, we have at almost every $t \in [0,T]$ that (see~\eqref{eq:ellipticdotu})
\[
\norm{u_{\lambda,\delta}(t)}_{\Wrm^{2,r}}\lesssim 1 + \norm{\lambda \dot{u}_{\lambda,\delta}(t)}_{\Lrm^r}+ \norm{\lambda \dot{u}_{\lambda,\delta}(t)}_{\Lrm^2}^{q-1},
\]
so that~\eqref{eq:ulambdadelta_Linfty2} follows from the above estimates.
\end{proof}

\subsection{Existence of $u_\lambda$.}
 
Before we establishes the existence of a solution at the scale $\lambda$ we collect a few technical compactness and interpolation results which we will use extensively in the next sections.
We begin by recalling the following result, which follows for instance from the interpolation estimate in~\cite[Theorem 2.13]{Triebel02}.

\begin{lemma}
\label{lem:triebel}
Let $\Omega$ be a bounded Lipschitz domain in $\R^d$ and let
$g\in(\Wrm^{k,s} \cap \Lrm^a)(\Omega)$. Then, for $l,k \in [0,\infty)$ and $s, a \in [1,\infty)$ such that
\[
l \leq k \qquad\text{and}\qquad
\frac{1}{\gamma}= \frac{k-l}{ka}+\frac{l}{ks},
\]
it holds that
\[
\norm{g}_{\Wrm^{l,\gamma}}\leq C\norm{g}_{\Wrm^{k,s}}^\frac{l}{k}\norm{g}_{\Lrm^{a}}^\frac{k-l}{k},
\]
where the constant depends on the exponents and the dimension only.
\end{lemma}
This lemma combined with the usual Sobolev embeddings~\cite[Theorem~2.5.1 and Remark~2.5.2]{Ziemer89book} implies for all $m \in [0,\infty)$, $\alpha \in [1,\infty)$ that satisfy
\[
  m\leq l  \qquad\text{and}\qquad  \frac{1}{\alpha}-\frac{m}{d}\geq \frac{1}{\gamma}-\frac{l}{d} = \frac{k-l}{ka}+\frac{l}{ks} - \frac{l}{d}
\]
the estimate
\begin{align}
\label{eq:triebel}
\norm{g}_{\Wrm^{m,\alpha}}\leq C\norm{g}_{\Wrm^{k,s}}^\frac{l}{k}\norm{g}_{\Lrm^{a}}^\frac{k-l}{k} .
\end{align}

This lemma can also be used to get the following convergence result.

\begin{lemma}
\label{lem:compact}
Let $\Omega$ be a bounded Lipschitz domain in $\R^d$ and assume for $a,b,s,\rho \in [1,\infty)$ and $k \in [0,\infty)$ that
\[
  g_j\to g \quad\text{in $\Lrm^b(0,T;\Lrm^a(\Omega))$} \qquad\text{and}\qquad
  \sup_{j\in \Nbb}\, \norm{g_j}_{\Lrm^\rho(0,T;\Wrm^{k,s}(\Omega))}\leq C
\]
and assume that for some $m \in \N \cup \{0\}$, $\alpha \in [1,\infty)$, and $\theta\in [0,1)$ we have
\[
m\leq \theta k  \qquad\text{and}\qquad
 \frac{1}{\alpha}-\frac{m}{d}\geq \frac{1-\theta}{a}+\theta\biggl(\frac1s-\frac{k}d\biggr).
\]
Then, for all
\[
\beta\in \biggl[1,\frac{\rho b}{b\theta+\rho(1-\theta)}\biggr]
\]
it holds that $g_j\to g$ in $\Lrm^\beta(0,T;\Wrm^{m,\alpha}(\Omega))$.
\end{lemma}
\begin{proof}
Take $h_j:=g_j-g$. We let $l:=\theta k$ and define $\gamma$ via
\[
  \frac{1}{\gamma}=\frac{k-l}{ka}+\frac{l}{ks}=\frac{1-\theta}{a}+\frac{\theta}{s}.
\]
This implies that
\[
\frac{1}{\gamma}-\frac{l}{d}=\frac{1-\theta}{a}+\theta\biggl(\frac1s-\frac{k}d\biggr)\leq \frac{1}{\alpha}-\frac{m}{d}.
\]
Now,~\eqref{eq:triebel} gives
\begin{align*}
\int_0^T\norm{h_j(t)}_{\Wrm^{m,\alpha}}^\beta\, \dd t \leq \int_0^T\norm{h_j(t)}_{\Wrm^{k,s}}^{\theta \beta}\cdot\norm{h_j(t)}_{\Lrm^{a}}^{(1-\theta)\beta} \dd t.
\end{align*}
Since $\frac{\beta (1-\theta)}{b}+\frac{b-\beta (1-\theta)}{b}=1$, H\"older's inequality then yields that
\begin{align*}
\int_0^T\norm{h_j(t)}_{\Wrm^{m,\alpha}}^\beta\, \dd t \leq \bigg(\int_0^T\norm{h_j(t)}_{\Wrm^{k,s}}^\frac{b\theta \beta}{b-\beta(1-\theta)}\dd t\bigg)^\frac{b-\beta(1-\theta)}{b} \cdot \bigg(\int_0^T\norm{h_j(t)}_{\Lrm^{a}}^b \dd t\bigg)^\frac{\beta(1-\theta)}{b}
\end{align*}
Since $\frac{b\theta \beta}{b-\beta(1-\theta)}\leq \rho$, the right-hand side converges to $0$ due to the uniform bounds on $h_j$ in $\Lrm^\rho(0,T;\Wrm^{k,s}(\Omega))$.
\end{proof}

The following proposition establishes the existence of a solution at scale $\lambda$.

\begin{proposition} \label{prop:apriori}
For every $\lambda>0$ there exists a strong solution to~\eqref{eq:PDE_lambda}. Moreover, letting $\delta \to 0$ (and holding $\lambda$ fixed), the sequence $(u_{\lambda,\delta})_{\delta > 0}$ of solutions to~\eqref{eq:PDE_ulambdadelta} has a subsequence that converges weakly in $\Wrm^{1,2}(0,T;\Lrm^2(\Omega;\R^m)) \cap \Lrm^2(0,T;\Wrm^{2,2}(\Omega;\R^m))$ to a strong solution $u_\lambda$ of~\eqref{eq:PDE_lambda} that satisfies the a-priori estimates
\begin{align}
    \lambda \norm{\dot{u}_\lambda}_{\Lrm^2_{t,x}}^2 + \norm{\dot{u}_\lambda}_{\Lrm^1_{t,x}} +\sup_{t\in [0,T]} \Bigl( \norm{u_\lambda(t)}_{\Wrm^{1,2}} + \norm{u_\lambda(t)}_{\Lrm^q} \Bigr) &\leq C,     \label{eq:dotulambda_RD_apriori}    
  \\
  \lambda \int_0^T t \norm{\nabla \dot{u}_\lambda(t)}_{\Lrm^2}^2\dd t &\leq C,    \label{eq:dotdotulambda_RI_apriori_t}
\end{align}
and for all $r \in [1,2^*)$ and $\sigma \in (0,T]$,
\begin{align}
\sup_{t\in[\sigma,T]} \lambda^{2} \norm{\dot{u}_\lambda(t)}_{\Lrm^2}^{2}  + \lambda^2 \int_\sigma^T\norm{\nabla \dot{u}_\lambda(t)}_{\Lrm^2}^2\dd t  &\leq C \biggl(1+\frac{\lambda}{\sigma}\biggr),           \label{eq:dotdotulambda_RI_apriori} \\
  \int_\sigma^T \norm{u_\lambda(t)}_{\Wrm^{2,r}}^2\dd t
     + \sup_{t\in[\sigma,T]} \norm{u_\lambda(t)}_{\Wrm^{2,2}}^{2}
&\leq C \biggl(1+\frac{\lambda}{\sigma}\biggr)^{\max\{q-1,1\}}.  \label{eq:ulambda_Wp2}
\end{align}
If $d = 3$, then~\eqref{eq:ulambda_Wp2} also holds for $r = 2^* = 6$. Here, the constants $C>0$ depend on $r$, $\norm{{f}}_{\Wrm^{1,\infty}(\Lrm^\infty)}$, $\norm{u_0}_{\Wrm^{1,2}}$, but are independent of $\lambda$ (assuming $\lambda \in (0,1)$). Moreover, $u_\lambda$ satisfies the energy balance
\begin{align}
&\Ecal(t,u_\lambda(t)) - \Ecal(0,u_0) \notag\\
&\quad = - \int_0^t \lambda \norm{\dot{u}_\lambda(\tau)}_{\Lrm^2}^2 + \Rcal_1(\dot{u}_\lambda(\tau)) \dd \tau -\int_0^t\dprb{\dot{f}(\tau),u_\lambda(\tau)}\dd \tau \label{eq:ulambda_energy}
\end{align}
for almost all $t \in [0,T]$.

If we assume additionally that $\Delta u_0 -\DD W_0(u_0)\in \Lrm^2(\Omega,\R^m)$, then 
\begin{align}
\sup_{t\in[0,T]} \lambda^2 \norm{\dot{u}_\lambda(t)}_{\Lrm^2}^2 + \lambda \int_0^T \norm{\nabla \dot{u}_\lambda(t)}_{\Lrm^2}^2 \dd t
&\leq C \bigl(1 + \lambda Z_0 \bigr), \label{eq:ulambda_Linfty1}
\\
\int_0^T \norm{u_\lambda(t)}_{\Wrm^{2,r}}^2\dd t
     + \sup_{t\in[0,T]} \norm{u_\lambda(t)}_{\Wrm^{2,2}}^{2}
&\leq C \bigl(1 + \lambda Z_0 \bigr)^{\max\{q-1,1\}}, \label{eq:ulambda_Linfty2}
\end{align}
where $Z_0 := \norm{\Delta u_0 - \DD W_0(u_0) + f(0)}_{\Lrm^2}^2$. In this case the solution to~\eqref{eq:PDE_lambda} is also unique.
\end{proposition}

\begin{proof}
The proof is accomplished by passing to the limit $\delta \todown 0$ in~\eqref{eq:PDE_ulambdadelta}.

\proofstep{Existence.}
Let $u_{\lambda,\delta}$ be a solution of~\eqref{eq:PDE_ulambdadelta} for given $\lambda,\delta > 0$. In case we assume $\Delta u_0 -\DD W_0(u_0)\in \Lrm^2(\Omega,\R^m)$ for~\eqref{eq:ulambda_Linfty1},~\eqref{eq:ulambda_Linfty2}, we also require that 
for the smooth approximation $(u_{0,\delta}) \subset \Crm^\infty(\cl{\Omega};\R^m)$ of the initial value $u_0$ additionally the convergence in~\eqref{eq:new-ass} holds.

Let $\xi \in \Lrm^1(0,T;\Wrm^{1,2}_0(\Omega;\R^m)$ and test the equation with $\xi - \dot{u}_{\lambda,\delta}$. This implies for almost all $[a,b] \subset (0,T)$  that
\begin{align*}
  &\int_a^b\int_\Omega \DD R_1^\delta(\dot{u}_{\lambda,\delta}) \cdot (\xi-\dot{u}_{\lambda,\delta}) \dd x\dd t
  \\
   &\quad = \int_a^b\int_{\Omega}\bigl[- \lambda \dot{u}_{\lambda,\delta} + \Delta u_{\lambda,\delta} - \DD W_0(u_{\lambda,\delta}) + f \bigr] \cdot (\xi - \dot{u}_{\lambda,\delta}) \dd x \dd t .
\end{align*}
By the convexity we know that $\DD R_1^\delta(z) \cdot (y-z) \leq R_1^\delta(y)-R_1^\delta(z)$. Hence,
\begin{align}
  &\int_a^b \int_\Omega R_1^\delta(\dot{u}_{\lambda,\delta})\dd x
   + \int_{\Omega}\bigl[- \lambda \dot{u}_{\lambda,\delta} + \Delta u_{\lambda,\delta} - \DD W_0(u_{\lambda,\delta}) + f \bigr] \cdot (\xi - \dot{u}_{\lambda,\delta}) \dd x   \notag\\
   &\quad \leq \int_a^b \int_\Omega R_1^\delta(\xi)\dd x.  \label{eq:deltaVar}
\end{align}
The a-priori estimates of Lemma~\ref{lem:apri} imply that there is a sequence of $\delta$'s (not explicitly denoted) such that
\[
  u_{\lambda,\delta}\toweak u_{\lambda}  \quad\text{in $\Lrm^2(0,T;\Wrm^{1,2}_0(\Omega;\R^m))\cap \Wrm^{1,2}(0,T;\Lrm^{2}(\Omega;\R^m))$} \qquad \text{as $\delta \todown 0$.}
\]
The classic Aubin--Lions compactness lemma then implies that for a subsequence (not explicitly labeled) $u_{\lambda,\delta}\to u_{\lambda}$ in $\Lrm^2((0,T) \times \Omega;\R^m)$.

Moreover, on every interval $[a,b] \subset (0,T)$ we have a uniform bound on the $u_{\lambda,\delta}$ in the space
\[
  \Lrm^\infty(a,b;\Wrm^{2,2}(\Omega;\R^m))
\]
As $d = 2,3$, the space $\Wrm^{2,2}(\Omega)$ is compactly embedded into $\Crm^\alpha(\Omega)$ for some $\alpha \in (0,1)$. Thus, we find by Lemma~\ref{lem:compact} that a subsequence convergences strongly in
\[
  \Lrm^\beta((a,b) \times \Omega;\R^m)\cap \Lrm^2(a,b;\Wrm^{1,2}(\Omega;\R^m))  \qquad\text{for any $\beta<\infty$.}
\]
Hence one may pass to the (lower) limit with~\eqref{eq:deltaVar} for almost every $a,b$ as above. For the term $\Delta u_{\lambda,\delta}\cdot \dot{u}_{\lambda,\delta}$, we use that, as $\delta \todown 0$, for almost every $a,b$ it holds that
\begin{align*}
&-\int_a^b \int_\Omega \Delta u_{\lambda,\delta}\cdot \dot{u}_{\lambda,\delta}\dd x \dd t
=\int_\Omega\frac{\abs{\nabla u_{\lambda,\delta}(b)}^2}{2}- \frac{\abs{\nabla u_{\lambda,\delta}(a)}^2}{2}\dd x 
\\
&\quad\to \int_\Omega\frac{\abs{\nabla u_{\lambda}(b)}^2}{2}- \frac{\abs{\nabla u_{\lambda}(a)}^2}{2}\dd x
=  -\int_a^b \int_\Omega \Delta u_{\lambda,\delta}\cdot \dot{u}_{\lambda}\dd x \dd t.
\end{align*}
For the term involving $R_1^\delta(\dot{u}_{\lambda,\delta})$, we use the weak lower semicontinuity of the functional $\Rcal_1$ and the fact that $\abs{R_1^\delta(z) - R_1(z)} \leq C\delta(1 + \abs{z})$ to see that
\begin{align*}
 \int_a^b \int_\Omega R_1(\dot{u}_{\lambda}(t))\dd x \dd t
 &\leq \liminf_{\delta\to 0}\int_a^b \int_\Omega R_1(\dot{u}_{\lambda,\delta}(t))\dd x \dd t \\
 &\leq \liminf_{\delta\to 0}\int_a^b \int_\Omega R_1^{\delta}(\dot{u}_{\lambda,\delta}(t))\dd x \dd t.
\end{align*}
All other terms converge in a straightforward manner. Thus, a solution $u_\lambda$ that satisfies~\eqref{eq:Varlam} is constructed. The a-priori estimates follow directly from Lemma~\ref{lem:apri} together with the weak lower semicontinuity of the respective norms. The energy equality~\eqref{eq:ulambda_energy} follows by passing to the limit in~\eqref{eq:ulambdadelta_energy} at almost every $s,t \in [0,T]$.

\proofstep{Uniqueness.}
Let us assume that we have two solutions to~\eqref{eq:PDE_lambda} (we suppress the fixed $\lambda > 0$ in the following)
\[
  v,w \in \Lrm^2(0,T;\Wrm^{1,2}_0(\Omega;\R^m))\cap \Wrm^{1,2}(0,T;\Lrm^{2}(\Omega;\R^m))
\]
that satisfy~\eqref{eq:Varlam} and $w(0)=u(0)$. Moreover, we assume that~\eqref{eq:ulambda_Linfty1},~\eqref{eq:ulambda_Linfty2} hold.

We use $\dot{v}$ as a test function in the variational inequality~\eqref{eq:deltaVar} for $w$ and $\dot{w}$ as test function $\xi$ for $v$ and then add the two resulting inequalities. This implies (almost everywhere in time)
\begin{align*}
 & \int_\Omega R_1(\dot{w}) + R_1(\dot{v})+\bigl[\lambda (\dot{w}-\dot{v}) - \Delta(w-v)  + \DD W_0(w)-\DD W_0(v) \bigr] \cdot (\dot{w}-\dot{v}) \dd x
  \\
  &\quad  \leq \int_\Omega R_1(\dot{v}) + R_1(\dot{w})\dd x.
\end{align*}
Then,
\begin{align*}
  &\int_\Omega \lambda \abs{\dot{w}-\dot{v}}^2 \dd x + \frac{\di}{\di t} \int_\Omega \frac{\abs{\nabla (w-v)}^2}2 \dd x \\
  &\quad \leq \int_\Omega \Big(\DD W_0(w)-\DD W_0(v) \Big) \cdot (\dot{v}-\dot{w}) \dd x
  \\
  &\quad \leq C\int_\Omega \int_0^1\abs{\DD^2 W_0(\sigma w+(1-\sigma)v)}\dd \sigma \cdot \abs{v-w} \abs{\dot{v}-\dot{w}} \dd x.
\end{align*}
By~\eqref{eq:ulambda_Linfty2} and the Sobolev embedding theorem, we find that $u,v$ are bounded functions, hence $\abs{\DD^2 W_0(\sigma w+(1-\sigma)v})$ is uniformly bounded, and 
\begin{align*}
\int_\Omega \lambda \abs{\dot{w}-\dot{v}}^2 \dd x + \frac{\di}{\di t} \int_\Omega \frac{\abs{\nabla (w-v)}^2}{2} \dd x\leq C\int_\Omega \abs{v-w} \abs{\dot{v}-\dot{w}} \dd x
\end{align*}
Using Young's inequality, we can absorb the term $\abs{\dot{v}-\dot{w}} $ to the left-hand side and find by Poincar\'e--Friedrichs's inequality
\begin{align*}
\frac{\di}{\di t} \int_\Omega \abs{\nabla (v-w)}^2 \dd x\leq C \int_\Omega \abs{\nabla(v-w)}^2  \dd x.
\end{align*}
Now Gronwall's lemma (and the zero boundary values of $v,w$) implies that $v\equiv w$.
\end{proof}

\section{Construction of a two-speed solution} \label{sc:twospeed}

In this section we construct a two-speed solution by letting $\lambda \todown 0$ and performing a careful investigation of the behavior around jump points.

\subsection{Limit passage $\lambda \todown 0$}
\label{ssc:lambda}

We start with a convergence lemma.

\begin{lemma}
\label{lem:bv-comp}
Let $s \in [1,\infty)$ and assume that
\[
  (v_k)_k \subset \Lrm^\infty(0,T;\Wrm^{1,s}(\Omega))\cap \BV(0,T;\Lrm^1(\Omega))
\]
is uniformly bounded (in these spaces). Then, there is a (non-relabeled) subsequence such that
\begin{enumerate}[(i)]
\item $v_k\to v$ in $\Lrm^{1}((0,T)\times \Omega)$;
\item $v_k\to v$ in $\Lrm^{\beta_1}(0,T;\Lrm ^{\alpha_1} (\Omega))$ for all $\beta_1\in [1,\infty)$ and all $\alpha_1\in [1,\tilde{q})$, where $\tilde{q}=\frac{d s}{d-s}$ for $s<d$ and $\tilde{q}=\infty$ otherwise;
\item If moreover $(v_k)_k$ is uniformly bounded in $\Lrm^\rho(0,T;\Wrm^{k,s}(\Omega))$ for a $k \in \N \cup \{0\}$ and $\rho \in [1,\infty]$, then
\[\qquad
  v_k\to v  \quad\text{in $\Lrm^{\beta_2}(0,T;\Wrm^{m,\alpha_2} (\Omega))$}
\]
for all $m\in 0,\ldots,k-1$ and $\alpha_2\in [1,\tilde{q})$ with $\tilde{q}=\frac{d s}{d-(k-m)s}$ for $s<\frac{d}{k-m}$ and $\tilde{q}=\infty$ otherwise, and $\beta_2\in \bigl[1, \frac{\rho}{\theta + \rho(1-\theta)}\bigr]$, where $\theta\in [0,1)$ is defined by
  \[\qquad
\frac{1}{\tilde{q}}-\frac{m}{d}<\frac{1}{\alpha_2}-\frac{m}{d}=1-\theta+\theta\biggl(\frac{1}{s}-\frac{k}{d}\biggr).
 \]
\end{enumerate}
\end{lemma}
\begin{proof}
For $N \in \N$ let $G_N := \{t_l^{(N)}\}_{l=1,\ldots,2^N}$ with $t_l^{(N)} \in (0,T)$ be nested grids (i.e.\ $G_N \subset G_{N+1}$) such that none of the $t_l^{(N)}$ lie on jump points on any of the $v_k$ (note that every jump set is countable), $\norm{v_k(t_l^{(N)})}_{\Wrm^{1,s}}$ is bounded uniformly in $l,k,N$ and, with $t_0^{(N)} := 0$, $t_{N+1}^{(N)} := T$,
\[
  h^{(N)} := \max_{l=0,\ldots,2^N} \absb{t_{l+1}^{(N)} - t_l^{(N)}} \to 0  \qquad\text{as $N \to \infty$.}
\]
By Rellich's compactness theorem in conjunction with a diagonal argument, we find that there is a subsequence of the $v_k$'s (not explicitly labeled) such that $(v_k(t_l^{(N)}))_k$ converges in $\Lrm^1$ as $k\to\infty$ for all $l,N\in \N$.

Let $\epsilon>0$ and fix $N$ such that $h^{(N)}\leq \epsilon$. Then, there exists an $n_\epsilon \in \N$ such that for $j,k\geq n_\epsilon$ and $l\in \{1, \ldots,2^N\}$,
\[
\max_{l=0,\ldots,2^N} \norm{v_j(t_l^{(N)})-v_k(t_l^{(N)})}_{\Lrm^1} \leq \epsilon.
\]
We then estimate
\begin{align*}
\int_0^T\norm{v_j(t)-v_k(t)}_{\Lrm^1}\dd t
&\leq \sum_{l=0}^{2^N} \int_{t_l^{(N)}}^{t_{l+1}^{(N)}}\norm{v_j(t)-v_k(t)}_{\Lrm^1}\dd t\\
&\leq \sum_{l=0}^{2^N} \int_{t_l^{(N)}}^{t_{l+1}^{(N)}}\norm{v_j(t)-v_j(t_l^{(N)})}_{\Lrm^1}\dd t
+h^{(N)}\sum_{l=0}^{2^N}\norm{v_j(t_l^{(N)})-v_k(t_l^{(N)})}_{\Lrm^1}\\
&\quad +\sum_{l=0}^{2^N} \int_{t_l^{(N)}}^{t_{l+1}^{(N)}}\norm{v_k(t)-v_k(t_l^{(N)})}_{\Lrm^1}\dd t
\\
&\leq \int_{0}^{h^{(N)}}\sum_{l=0}^{2^N}\norm{v_j(t_l^{(N)}+s)-v_j(t_l^{(N)})}_{\Lrm^1}\dd s
+T\epsilon
\\
&\quad +\int_{0}^{h^{(N)}}\sum_{l=0}^{2^N}\norm{v_k(t_l^{(N)}+s)-v_k(t_l^{(N)})}_{\Lrm^1}\dd s \\
&\leq T\epsilon + h^{(N)}\norm{v_j}_{\BV(\Lrm^1)}+h^{(N)}\norm{v_k}_{\BV(\Lrm^1)} \\
&\leq C\epsilon
\end{align*}
This implies $v_j\to v$ in $\Lrm^{1}((0,T) \times \Omega)$, i.e.,~(i).

For~(ii), use $a=b=1$, $\rho \in \N$, $k=1$, $s=s$, $\beta=\beta_1$, $m=0$ and $\alpha=\alpha_1$ in Lemma~\ref{lem:compact}. Then, the result follows by fixing $\theta\in (0,1)$ such that 
\[
\frac{1}{\tilde{q}} < \frac{1}{\alpha_1}=1-\theta+\theta\Big(\frac{1}{s}-\frac{1}{d}\Big).
\]
That this is possible follows by the fact that $\alpha_1=1$ relates to $\theta=0$ and $\alpha_1=\tilde{q}$ relates to $\theta=1$. We then get $v_j\to v$ in $\Lrm^{\beta_1}(0,T;\Lrm^{\alpha_1}(\Omega))$ for all 
\[
  \beta_1 \in \biggl[1, \frac{\rho}{b\theta + \rho(1-\theta)}\biggr].
\]
Letting first $\theta \toup 1$ and then $\rho \to \infty$ yields the claim.

For~(iii), take $a=b=1$, $\rho=\rho$, $k=k$, $s=s$, $\beta=\beta_2$, $m=m$ and $\alpha=\alpha_2$ in Lemma~\ref{lem:compact}. The result follows by fixing $\theta\in [0,1)$ such that 
\[
\frac{1}{\tilde{q}}-\frac{m}{d} < \frac{1}{\alpha_2}-\frac{m}{d}=1-\theta+\theta\Big(\frac{1}{s}-\frac{k}{d}\Big).
 \]
 Here the endpoint $\alpha_2=\tilde{q}$ relates to $\theta=1$.
 We calculate (with the usual convention $\frac1{0}=\infty$)
 \[
 \theta=\frac{\frac{1}{\alpha_2}-\frac{m}{d}-1}{\frac{1}{s}-\frac{k}{d}-1}.
 \]
Once $\theta$ is fixed, we may choose
 \[
  \beta_2\in \biggl[1, \frac{\rho}{\theta + \rho(1-\theta)}\biggr].
 \]
 This completes the proof.
\end{proof}

We now consider the behavior of the solutions $u_\lambda$ constructed in Proposition~\ref{prop:apriori} as $\lambda \todown 0$. By the uniform estimates in Proposition~\ref{prop:apriori} and
Lemma~\ref{lem:bv-comp}, we obtain the following convergences for a (non-relabeled) subsequence:
\begin{equation}
\label{eq:l_conv}
\left\{
\begin{aligned}
u_\lambda&\to u  \quad\text{in } \Lrm^a_\loc((0,T) \times \Omega;\R^m) \quad\text{for $a\in [1,\infty)$;}\\
u_\lambda&\to u  \quad\text{in } \Lrm^a_\loc(0,T;\Wrm^{1,r}_0(\Omega;\R^m)) \quad\text{for $a\in [1,\infty)$, $r\in [1,2^*)$;}\\
u_\lambda&\toweak  u  \quad\text{in }\Lrm^2_\loc(0,T;\Wrm^{2,r}(\Omega;\R^m))\quad\text{ for $r\in [1,2^*)$;}\\
u_\lambda&\toweakstar u  \quad\text{in }\Lrm^\infty(0,T;(\Wrm^{1,2} \cap \Lrm^q)(\Omega;\R^m));\\
u_\lambda&\toweakstar u  \quad\text{in }\Lrm^\infty_\loc(0,T;\Wrm^{2,2}(\Omega;\R^m));\\
u_\lambda &\toweakstar u \quad\text{in $\BV(0,T;\Lrm^1(\Omega;\R^m)))$.}
\end{aligned}
\right.
\end{equation}
Indeed, the a-priori estimates~\eqref{eq:dotulambda_RD_apriori}--\eqref{eq:ulambda_Wp2} yield uniform boundedness in the spaces
\[
  \BV(\Lrm^1), \quad \Lrm^\infty(\Wrm^{1,2} \cap \Lrm^q), \quad \Lrm^2_\loc(\Wrm^{2,r}), \quad \Lrm^\infty_\loc(\Wrm^{2,2}).
\]
Then we apply Lemma~\ref{lem:bv-comp}~(iii) with $s = 2$, $\rho = 2$, $k = 2$, $m = 1$, and $\alpha_2 \to \tilde{q} = 2^*$ ($\theta \to 1$) to obtain the strong convergence in $\Lrm^a_\loc(0,T;\Wrm^{1,r}_0(\Omega;\R^m))$ for $a\in [1,\infty)$, $r\in [1,2^*)$. Finally, the strong convergence in $\Lrm^a_\loc((0,T) \times \Omega;\R^m)$ for $a\in [1,\infty)$ follows by embedding since $r$ can be chosen to be larger than $d$.

The above convergences in particular imply that
\[
\Ecal(t,u_{\lambda}(t)) \to \Ecal(t,u(t)) \qquad\text{for almost every $t\in (0,T]$}.
\]
Moreover, by the weak* lower semicontinuity of the variation in $\BV(0,T;\Lrm^1(\Omega;\R^m))$, we find from the energy balance~\eqref{eq:ulambda_energy} the energy inequality~\eqref{eq:subenerg} for almost all $s,t \in [0,T]$.

\subsection{Jump evolutions}

In this section we construct the resolution for the jump transients. We introduce the following rescaling of $u_\lambda$:
\[
v_\lambda(t,\theta):=u_\lambda(t+\lambda\theta),  \qquad t \in [0,T], \; \theta \in \biggl(-\frac{t}{\lambda},\frac{T-t}{\lambda}\biggr).
\]

We first give a sketch of the procedure to follow. Observe (by a change of variables) that $v_\lambda$ satisfies the following PDE in $\theta$ for fixed $t$:
\[
\left\{
\begin{aligned}
\partial_\theta v_\lambda(t,\theta) + \partial \Rcal_1(\partial_\theta  v_\lambda(t,\theta)) - \Delta  v_\lambda(t,\theta) + \DD W_0( v_\lambda(t,\theta)) &\ni f(t+\lambda \theta),
\\
 v_\lambda(t,-t/\lambda)&=u_0
\end{aligned}
\right.
\]
for almost every $t \in (0,T)$, $\theta \in (-\frac{t}{\lambda},\frac{T-t}{\lambda})$. Here, the above system is understood in the same way as~\eqref{eq:Varlam}. Heuristically, for fixed $\theta$,
\[
f(t+\lambda\theta)=f(t) + \lambda\theta \dot{f}(t)+\BigO(\lambda^2).
\]
Hence, as $\lambda \todown 0$ we will be able to show (see Lemma~\ref{lem:energ2} below) that the limit process $v(t,\frarg)$ is a strong solution of
\begin{equation} \label{eq:theta}
  \partial_\theta v(t,\theta) + \partial \Rcal_1(\partial_\theta  v(t,\theta)) - \Delta  v(t,\theta) + \DD W_0( v(t,\theta)) \ni f(t)
\end{equation}
in the sense analogue to~\eqref{eq:Varlam}.

One then expects that these $v(t,\frarg)$ parameterize the jumps. However, this picture is incomplete as one may need to deal with countably many separate evolutions that together constitute the evolution over the jump. These evolutions are separated on an intermediate scale.

We start with a lemma collecting some quantitative estimates at the fine scale $\theta$.

\begin{lemma}
\label{lem:energ2}
Let $t\in (0,T]$ and let $\tau_j\to t$, $L_j\to \infty$ be such that $\lambda_j L_j\to 0$ and $[\tau_j-\lambda_jL_j,\tau_j+\lambda_j L_j]\subset (0,T)$.
Then there is a (non-relabeled) subsequence such that for every $L>0$ the maps
\[
  v^j(\theta):=u_{\lambda_j}(\tau_j+\lambda_j\theta),  \qquad \theta \in (-L,L),
\]
converge to a process
\[
  v \in \Lrm^\infty(-\infty,\infty;(\Wrm^{1,2}_0 \cap \Wrm^{2,2})(\Omega;\R^m))  \quad\text{with}\quad
  \dot{v} \in \Lrm^2((-\infty,\infty)\times\Omega;\R^m)
\]
in the following sense:
\begin{equation}
\label{eq:v_conv}
\left\{
\begin{aligned}
v^j&\to v  \quad\text{in } \Lrm^a_\loc((-\infty,\infty)\times \Omega;\R^m) \quad\text{for $a\in [1,\infty)$;}\\
v^j&\to v  \quad\text{in } \Lrm^a_\loc(-\infty,\infty;\Wrm^{1,r}_0(\Omega;\R^m)) \quad\text{for $a\in [1,\infty)$, $r\in [1,2^*)$;}\\
v^j&\toweak  v  \quad\text{in }\Lrm^2_\loc(-\infty,\infty;\Wrm^{2,r}(\Omega;\R^m))\quad\text{ for $r\in [1,2^*)$;}\\
v^j&\toweak  v  \quad\text{in }\Wrm^{1,2}_\loc(-\infty,\infty;\Wrm^{1,2}(\Omega;\R^m));\\
v^j&\toweakstar v  \quad\text{in }\Lrm^\infty(-\infty,\infty;(\Wrm^{1,2} \cap \Lrm^q)(\Omega;\R^m));\\
v^j&\toweakstar v  \quad\text{in }\Lrm^\infty_\loc(-\infty,\infty;\Wrm^{2,2}(\Omega;\R^m));\\
v^j&\toweakstar v \quad\text{in }\Wrm^{1,\infty}_\loc(-\infty,\infty;\Lrm^2(\Omega;\R^m)).
\end{aligned}
\right.
\end{equation}
Moreover, $v$ is a strong solution to~\eqref{eq:theta} and satisfies the energy balance
\begin{align}
\Ecal(t,v(b))-\Ecal(t,v(a)) &= \Wcal(v(b))-\Wcal(v(a)) \notag\\
&= -\int_a^b \norm{\partial_\theta v(\theta)}_{\Lrm^2}^2 + \Rcal_1(\partial_\theta v(\theta))  \dd \theta 
\leq 0
\label{eq:v_energy}
\end{align}
for almost all intervals $[a,b] \subset (-\infty,\infty)$.
\end{lemma}

\begin{proof}
Let $L>0$. There exists a $j_0\in \N$ such that $[\tau_j-\lambda_jL,\tau_j+\lambda_j L]\subset [\tau_j-\lambda_jL_j,\tau_j+\lambda_j L_j] \subset (0,T)$ for all $j\geq j_0$. Hence,
\[
\partial_\theta v^j + \partial \Rcal_1(\partial_\theta v^j) -\Delta v^j+\DD W_0(v^j) \ni f(\tau_j+\lambda_j \frarg)  \quad   \text{in }(-L,L)\times \Omega.
\]
The $\lambda$-uniform estimates from Proposition~\ref{prop:apriori} give the uniform boundedness of $v^j$ in the space
\[
  \Lrm^2(-L,L;\Wrm^{1,2}_0(\Omega;\R^m)) \cap \Lrm^\infty(-L,L;\Wrm^{2,2}(\Omega;\R^m))
\]
Moreover, employing the uniform boundedness of $\lambda \norm{\dot{u}_{\lambda_j}}_{\Lrm^2_{t,x}}^2$, see~\eqref{eq:dotulambda_RD_apriori}, we get that
\[
  \int_{-L}^{L} \abs{\partial_\theta v^j(\theta)}^2 \dd \theta
  = \lambda^2 \int_{-L}^{L} \abs{\dot{u}_{\lambda_j}(\tau_j + \lambda_j\theta)}^2 \dd \theta
  = \lambda \int_{\tau_j-\lambda_jL}^{\tau_j+\lambda_jL} \abs{\dot{u}_{\lambda_j}(t)}^2 \dd t
  \leq C.
\]
Thus we conclude that the $v^j$ are also uniformly bounded in
\[
  \Wrm^{1,2}(-L,L;\Lrm^{2}(\Omega;\R^m)).
\]
Like in the proof of Proposition~\ref{prop:apriori} (via the Aubin--Lions compactness lemma and Lemma~\ref{lem:compact}), we can then see that a (non-relabeled) subsequence of the $v^j$ converges to a process $v \in \Lrm^2(0,T;\Wrm^{1,2}_0(\Omega;\R^m))\cap \Wrm^{1,2}(0,T;\Lrm^{2}(\Omega;\R^m))$ weakly in that space and also strongly in
\[
  \Lrm^\beta((-L,L) \times \Omega;\R^m)\cap \Lrm^2(-L,L;\Wrm^{1,2}(\Omega;\R^m))  \qquad\text{for any $\beta<\infty$.}
\]
Thus, again like in the proof of Proposition~\ref{prop:apriori}, we may then pass to the limit in the equation to see that $v$ is a strong solution of
\[
\partial_\theta v + \partial \Rcal_1(\partial_\theta v) -\Delta v+\DD W_0(v) \ni f(t) \quad   \text{in }(-L,L)\times \Omega,
\]
which also satisfies the energy balance~\eqref{eq:v_energy} on $[-L,L]$. Observe that in particular $f(\tau_j+\lambda_j\theta) \to f(t)$ uniformly in $[-L,L]\times \Omega$ by~\ref{as:f}.

A solution can be found for any $L$ and we need to show that these solutions agree on their joint interval of existence. Denote the solution on the interval $[-L,L]$ by $v_L$. Since by~\eqref{eq:ulambda_Wp2}, $v_\lambda$ is in $\Lrm^\infty((-L+\delta,L);\Wrm^{2,2}(\Omega;\R^m))$ for arbitrary $\delta > 0$, we infer that
\[
  \Delta v_L(-L+\delta) -\DD W_0(v_L(-L+\delta))\in \Lrm^2(\Omega,\R^m).
\]
Moreover, for any $M > L$, we also have $v_M(-L+\delta) = v_L(-L+\delta)$. Hence, by the uniqueness part of Proposition~\ref{prop:apriori}, on the interval $(-L+\delta,L)$ the solutions $v_L$ and $v_M$ agree. As we can choose $\delta > 0$ arbitrarily, $v_L = v_M$ on the joint interval of existence $(-L,L)$. Thus, we can construct the global solution by concatenation.

The additional convergences and the energy balance (locally) in $(-\infty,\infty)$ are satisfied due to the uniformity of the rescaled estimates in Proposition~\ref{prop:apriori} and compactness arguments like we employed at the beginning of this section for~\eqref{eq:l_conv}.
\end{proof}

\begin{lemma}
\label{lem:energ2b}
In the situation of the previous lemma, with $E(t) := \Ecal(t,u(t))$ the following implication holds: If
\[
  \Ecal(\tau_j-\lambda_jL_j,u_{\lambda_j}(\tau_j-\lambda_jL_j))\to \limsup_{s\toup t} E(s)
\]
and
\[
  \Ecal(\tau_j+\lambda_jL_j,u_{\lambda_j}(\tau_j+\lambda_jL_j))\to  \liminf_{s \todown t} E(s),
\]
then
\begin{align}
\label{eq:energ-v}
  0 \leq \Ecal(t,v(-\infty))-\Ecal(t,v(\infty))\leq \limsup_{s\toup t} E(s) - \liminf_{s \todown t} E(s).
\end{align}
In particular, if $t \in (0,T)$ is a continuity point of the energy process $E(t)$, then $\Ecal(t,v(\theta))= \Ecal(t,u(t))$ for all $\theta \in (-\infty,\infty)$ and consequently $v\equiv u(t)$.  
\end{lemma}

\begin{proof}
We set (see~\eqref{eq:v_energy})
\[
\gamma := \Ecal(t,v(-\infty))-\Ecal(t,v(\infty))=\Wcal(v(-\infty))-\Wcal(v(\infty)) \geq 0.
\]
For $l\in\N$ fix $L_l>0$ such that
 \[
 \gamma-\frac{1}{l}\leq \Ecal(t,v(-L))-\Ecal(t,v(L))  \qquad \text{for all $L\geq L_l$.}
 \]
The functions $\theta \mapsto \Ecal(\tau_j + \lambda_j\theta, v^j(\theta))$ converge uniformly on bounded sets to $\theta \mapsto \Ecal(t,v(\theta))$ as $j\to \infty$, which is implied by the convergence $v^j \toweak  v$ in $\Wrm^{1,2}_\loc(-\infty,\infty;\Wrm^{1,2}(\Omega;\R^m))$ together with the $\Lrm^\infty_\loc(-\infty,\infty;\Wrm^{2,2}(\Omega;\R^m))$-boundedness of $v^j$, see~\eqref{eq:v_conv}. Indeed, the first convergence gives by the Aubin--Lions compactness lemma that $v^j \to v$ in $\Lrm^\infty_\loc(-\infty,\infty;\Lrm^2(\Omega;\R^m))$, which together with the boundedness in $\Lrm^\infty_\loc(-\infty,\infty;\Wrm^{2,2}(\Omega;\R^m))$ yields by interpolation (see, e.g.,~\cite[Remark 3.2]{RindlerSchwarzacherSuli17} together with the Arzela--Ascoli theorem)
\[
  v^j \to v  \quad\text{in $\Lrm^\infty_\loc(-\infty,\infty;\Wrm^{1,2}(\Omega;\R^m))$.}
\]
Thus the claimed uniform convergence above follows.
Hence, there is an index $j_0\in \N$ such that we have $\lambda_j L_l \leq 1/l$, $\abs{t-\tau_j}\leq 1/l$, and
\[
  \gamma -\frac{2}{l}\leq \Ecal( \tau_j-\lambda_j L_l, v^j(-L_l) ) - \Ecal( \tau_j+\lambda_j L_l, v^j(L_l) )  \qquad
  \text{for all $j\geq j_0$.}
\]
This implies using the uniform continuity of $f$ that for $h \in [\frac{2}{l},\frac{3}{l}]$ (also using~\eqref{eq:v_energy})
\begin{align*}
\gamma -\frac{3}{l}
&\leq \Ecal( t,u_{\lambda_j}(\tau_j-\lambda_j L_l) ) - \Ecal( t,u_{\lambda_j}(\tau_j+\lambda_j L_l) )
\\
&= \Wcal(u_{\lambda_j}(\tau_j-\lambda_j L_l)) - \Wcal(u_{\lambda_j}(\tau_j+\lambda_j L_l))
\\
&\leq \Wcal(u_{\lambda_j}(t-h)) - \Wcal(u_{\lambda_j}(t+h)).
\end{align*}
Then, as $j\to \infty$,
\[
\gamma -\frac{3}{l}\leq \dashint_{\frac{2}{l}}^\frac{3}{l}\Wcal (u_{\lambda_j}(t-s))-\Wcal(u_{\lambda_j}(t+s))\dd s \to \dashint_{\frac{2}{l}}^\frac{3}{l}\Wcal (u(t-s))-\Wcal(u(t+s)) \dd s.
\]
 This implies~\eqref{eq:energ-v} by the uniform continuity of $f$.

If $t$ is a continuity point of $E(t)$, we find that $\Wcal(v(\theta))= E(t)$ for all $\theta\in (-\infty,\infty)$. By~\eqref{eq:v_energy} this is only possible if $v$ is constant in time.
\end{proof}

\subsection{Parabolic points}
\label{ssc:parab}

We recall from Example~\ref{ex:jumplength} and Example~\ref{ex:jumplength1} that jumps may have a rate-independent dissipation that is strictly larger than the total variation of the limit time derivative. In the following we investigate the points where the parabolic term contributes and introduce for the rate-independent jumps the corresponding ``stretchings''.

Let $N_j \in \N$ with $N_j \to \infty$ as $j \to \infty$. We divide $[0,T]$ into $N_j$ intervals of size $\lambda_j := \frac{T}{N_j}$ and set
\[
  t_{k}^{\lambda_j} := k\lambda_j = k\frac{T}{N_j},  \qquad k = 0,\ldots,N_j.
\]
The main task in the following will be to analyze the change of the energy in each interval $[t_{k}^{\lambda_j},t_{k+1}^{\lambda_j}]$. 

We define the following \term{energy loss process} (associated to our solutions $u_{\lambda_j}$):
\begin{align*}
\omega^{\lambda_j}(t) &:=\Ecal(0,  u_0)  
-\Ecal( t, u_{\lambda_j}(t)) -\int_{0}^t \dprb{\dot{f}(\tau),u_{\lambda_j}(\tau)} \dd \tau\\
&\phantom{:}= \Wcal(u_0) - \Wcal(u_{\lambda_j}(t)) + \int_{0}^t \dprb{f(\tau),\dot{u}_{\lambda_j}(\tau)} \dd \tau,  \qquad t \in [0,T].
\end{align*}
These functions are continuous and increasing by~\eqref{eq:ulambda_energy}.

The goal is to classify the jumps that develop in $\omega^{\lambda_j}$ when $\lambda_j \todown 0$ as either rate-dependent jumps (where the parabolic approximation matters) or rate-independent jumps (where the approximation gives no contribution). We will do this inductively for $m=1,2,\ldots$ and at each step take
subsequences of our $\lambda_j$'s, which we will however still denote as
$\lambda_j$.

Suppose first that $m=1$. We will say that we have a \term{parabolic sequence} $\{t_{k_j}^{\lambda_j}\}_{j=1}^{\infty}$ for
$m=1$ if for some $j^{(1)}\in \Nbb$ it holds that
\[
\omega^{\lambda_j}\left(t_{k_j+1}^{\lambda_j}\right)  - \omega^{\lambda_j}\left(t_{k_j}^{\lambda_j}\right)  \geq\frac{1}{m}=1  \qquad\text{for all $j \in \{ j^{(1)},j^{(1)}+1,\ldots\}$}
\]
and
\[
  t_{k_j}^{\lambda_j}\to t^{(1)}\in[0,T]  \qquad\text{as $j\to\infty$.}
\]
Over the interval $[t_{k_j}^{\lambda_j},t_{k_j+1}^{\lambda_j}]$ the processes $u_{\lambda_j}$ will develop a rate-dependent jump since the related transient slope is of order $\lambda_j$. We will assume in the following that $t^{(1)} \in\left(0,T\right)$. If $t^{(1)}=0$ or $T$ we need to consider problems in half-infinite intervals, and make suitable adaptations to the arguments.

It now follows by Lemma~\ref{lem:energ2} that for a parabolic sequence $\{t_{k_j}^{\lambda_j}\}_{j=1}^{\infty}$ with $m=1$ there exists a
sequence of integers $L_j\to\infty$, chosen below, such that $\lambda_jL_j\to 0$ and the functions
\[
v_j(\theta) :=u_{\lambda_j}\left(t_{k_j}
^{\lambda_j}+\lambda_j\theta\right)
\]
converge in the sense of~\eqref{eq:v_conv} to a strong solution
of the equation
\[
\partial_{\theta}v(\theta)+ \partial\Rcal_1(\partial_{\theta}v(\theta))-\Delta u+\DD W_0(v(\theta))  =f(t^{(1)}),  \qquad\theta\in  (-\infty,\infty),
\]
We label this solution as $v^{(1)}$. The total change of energy of
this function between $\theta=-\infty$ and $\theta=\infty$ can be computed to be
\begin{align*}
\Ecal(t^{(1)},v^{(1)}(-\infty))-\Ecal(t^{(1)},v^{(1)}(\infty))
&=\phantom{:} \int_{-\infty}^{\infty}  \norm{\partial_\theta v^{(1)}(\theta)}_{\Lrm^2}^2 + \Rcal_1(\partial_\theta v^{(1)}(\theta)) \dd \theta \\
&=:d^{(1)}\geq\frac{1}{m}=1.
\end{align*}
Indeed, observe that since $\omega^{\lambda_j}$ is increasing,
\begin{align*}
  \frac{1}{m} &\leq \omega^{\lambda_j}\left(t_{k_j+1}^{\lambda_j}\right)  - \omega^{\lambda_j}\left(t_{k_j}^{\lambda_j}\right) \\
  &\leq \omega^{\lambda_j}\left(t_{k_j}^{\lambda_j} + \lambda_j L_j\right) - \omega^{\lambda_j}\left(t_{k_j-L_j}^{\lambda_j} - \lambda_j L_j\right)\\
  &= \Wcal \left( v_j(-L_j) \right) - \dprb{f(t_{k_j}^{\lambda_j}-\lambda_j L_j), v_j(-L_j)}  - \Wcal \left( v_j(+L_j) \right) + \dprb{f(t_{k_j}^{\lambda_j}+\lambda_j L_j), v_j(L_j)} \\
  &\quad - \int_{t_{k_j}^{\lambda_j}-\lambda_j L_j}^{t_{k_j}^{\lambda_j}+\lambda_j L_j} \dprb{\dot{f}(\tau),u_{\lambda_j}(\tau)} \dd \tau
\end{align*}
and the last expression converges via~\eqref{eq:v_conv},~\ref{as:f} to
\[
  d^{(1)} = \Ecal(t^{(1)},v^{(1)}(-\infty))-\Ecal(t^{(1)},v^{(1)}(\infty)).
\]
The expression involving the dissipation integrals then follows from~\eqref{eq:v_energy}.

We can furthermore assume that the $L_j$ are chosen in such a way that
\[
\left\vert \omega^{\lambda_j}\left(t_{k_j}^{\lambda_j} + \lambda_j L_j\right) - \omega^{\lambda_j}\left(t_{k_j-L_j}^{\lambda_j} - \lambda_j L_j\right) - d^{(1)}\right\vert \leq\frac{1}{j}d^{(1)}.
\]
Note that here we may need to take another (non-relabeled) subsequence of $\lambda_j$.

Define the first set of \term{parabolic indices},
\[
A_j^{(1)}:=\setB{  t_{\ell}^{\lambda_j}}{\left\vert
\ell-k_j\right\vert \leq L_j}
\]
and consider the new collection of time points
\[
\Sigma_j^{(1)}:=\setB{  t_{k}^{\lambda_j} }{ k=0,1, \ldots,N_j } \setminus A_j^{(1)}.
\]

We now examine if in this sequence there are additional parabolic points
with $m=1$. If there are, we iterate the argument (taking a further subsequence of the $\lambda_j$'s if needed). The corresponding new parabolic sequence, which can be represented as $\bigl\{  t_{k_j^{(2)}}^{\lambda_j^{(2)}}\bigr\}$, converges to another time $t^{(2)}$. We can then repeat again the argument. The corresponding limit function $v^{(2)}$ has a dissipation $d^{(2)}$ relative to a sequence $L_j$ as above, which we label as $L_j^{(2)}$.
 We then define
\[
A_j^{(2)}:=\setB{  t_{\ell}^{\lambda_j}\in\Sigma
_j^{(1)}}{\left\vert \ell-k_j\right\vert \leq L_j^{(2)}}.
\]
Observe that we can choose $L_j^{(2)}\to\infty$ in such a way that $A_j^{(2)}$ does not overlap with the previous set $A_j^{(1)}$. Indeed, if this were not the case, it would mean that the sequences yielding the first parabolic point $t^{(1)}$ and the second parabolic point $t^{(2)}$ are at a distance of order $\leq C\lambda_j$ (uniformly for $j$ large enough) and hence can be merged into a single evolution.

We then iterate until finishing with all the parabolic points for $m=1$. The
number of such points is finite, because each of them yields a dissipation of
energy at least $1$.

We now define the stretching function at the level $m=1$. After removing all
the sets $A_j^{(\ell)}$ associated to parabolic points with
$m=1$, we are left with
\[
\Sigma_j^{(\ell)}=\Sigma_j^{\left(\ell-1\right)
}\setminus A_j^{(\ell)},
\]
where $\ell$ is the number of removed rate-dependent jumps with $m=1$.
Notice that we have new sequences and implicitly do a relabelling each time we take another subsequence. We then denote alternatively $Z_j^{1}:=\Sigma_j^{(\ell)}$ to emphasize that we are at the end of the level $m=1$.

We can now define the \term{stretching functions} $\psi_j^{1} \colon [0,T]\to [0,\infty)$ as the piecewise affine functions that satisfy (the superindex $1$ denotes $m=1$):
\begin{align*}
\psi_j^{1}\left(t_{k+1}^{\lambda_j}\right)  -\psi_j^{1}\left(
t_{k}^{\lambda_j}\right)  
 &  := \begin{cases}
\left(t_{k+1}^{\lambda_j}-t_{k}^{\lambda}\right)  
 +\left(\omega^{\lambda_j}\left(t_{k+1}^{\lambda_j}\right)  -\omega^{\lambda_j}\left(t_{k}^{\lambda_j}\right)
\right) 
 &\text{if }t_{k}^{\lambda_j}\in Z_j^{1},
 \\
\left(t_{k+1}^{\lambda_j}-t_{k}
^{\lambda_j}\right)  &\text{if }t_{k}^{\lambda_j}\notin Z_j^{1},
 \end{cases}
\\
\psi_j^{1}\left(0\right)  & :=0.
\end{align*}
At all other intermediate points, $\psi_j^{1}$ are defined by linear interpolation.

The idea is that we stretch out all possible jumps of the energy which are approached in a rate-independent manner, meaning that the speed of their evolution is much slower than $1/\lambda_j$. 
Observe that there is an $i^{(1)}\in \Nbb$ such that $m \lambda_j \leq 1$ and for $t_{k}^{\lambda_j}\in Z_j^{1}$, $j\in \{i^{(1)},i^{(1)}+1,i^{(1)}+2,...\}$ it holds that
\[
\omega^{\lambda_j}\left(t_{k+1}^{\lambda_j}\right)  -\omega^{\lambda_j}\left(t_{k}^{\lambda_j}\right)\leq \frac{2}{m}=2,
\]
since otherwise one could construct another parabolic sequence that is associated to a parabolic point with jump height at least $1/m=1$. Hence we find for $j\in \{i^{(1)},i^{(1)}+1,i^{(1)}+2,...\}$
\begin{equation}
\label{eq:lipschitz}
\abs{\partial_t\psi_j^{1}} \leq 1+\frac{2}{m \lambda_j}\leq \frac{3}{m \lambda_j}=\frac{3}{\lambda_j}.
\end{equation}
The addition of the term $\left(t_{k+1}^{\lambda_j}-t_{k}^{\lambda_j}\right)$ guarantees also a lower bound on the derivative, such that
\[
1\leq \abs{\partial_t\psi_j^{1}}\leq \frac{3}{\lambda_j}.
\]

Notice that the functions $\psi_j^{1}$ are strictly increasing and uniformly bounded in the interval $[0,T]$ due to the finiteness in the change of the energy in that
interval. Notice also that the functions $\psi_j^{1}$ have large variation where there is a large amount of rate-independent dissipation. 
Defining
\[
  \phi_j^{1}(s):=\left(\psi_j^{1}\right)^{-1}(s)  \qquad\text{and}\qquad 
s^{\lambda_j}_k:=\psi(t^{\lambda_j}_k),
\]
we have for $t^{\lambda_j}_k\in Z_j^{1}$ that
\begin{align*}
s^{\lambda_j}_{k+1}-s^{\lambda_j}_{k} = \left(t_{k+1}^{\lambda_j}-t_{k}^{\lambda_j}\right)  
 +\left(\omega^{\lambda_j}\left(t_{k+1}^{\lambda_j}\right) - \omega^{\lambda_j}\left(t_{k}^{\lambda_j}\right)\right).
\end{align*}

We now iterate over $m$. We remove the parabolic points at the level $m=2,$ which are characterized by the condition
\[
\omega^{\lambda_j}\left(t_{k_j+1}^{\lambda_j}\right)  -\omega^{\lambda_j} \left(t_{k_j}^{\lambda_j}\right)  \geq\frac{1}{m}=\frac{1}{2},
\]
for $j$ large enough.
This defines additional parabolic limits $v^{(r)}$, $r=\ell+1,\ell+2,\ldots$, for which the corresponding dissipations $d^{(r)} \geq \frac{1}{m} = \frac12$ satisfy
\begin{equation}
\label{A1}
\left\vert \omega^{\lambda_j}\left(t_{k_j}^{\lambda_j} + \lambda_j L_j\right)
-\omega^{\lambda_j}\left(t_{k_j}^{\lambda_j} - \lambda_j L_j\right)  -d^{(r)}\right\vert \leq\frac{1}{j}d^{(r)}.
\end{equation}
The sequence $\{d^{(r)}\}$ constructed in this way is independent of the subsequence of $\lambda_j \to0$ that we chose above.

The sets $A_j^{\left(
\ell+1\right)},\ A_j^{(\ell+2)},\ldots$ are defined exactly as
above at the new rate-dependent jumps, and then also
\begin{align*}
\Sigma_j^{(\ell+1)}  &  :=\Sigma_j^{(\ell)
}\setminus A_j^{(\ell)},\\
\Sigma_j^{(\ell+2)}  &  :=\Sigma_j^{(\ell+1)
}\setminus A_j^{(\ell+1)},\\
&\ldots
\end{align*}
Eventually, we will exhaust the rate-dependent jumps at level $m=2$ at some $\ell = \ell^{(2)}$. We then denote the corresponding set $\Sigma_j^{\left(\ell^{(2)}\right)}$ by $Z_j^{2}$.

As before, we define the functions $\psi_j^{2},$ exactly in the same form as above, namely as the piecewise affine functions satisfying
\begin{align*}
\psi_j^{2}\left(t_{k+1}^{\lambda_j}\right)  -\psi_j^{2}\left(
t_{k}^{\lambda_j}\right)  
 &  := \begin{cases}
\left(t_{k+1}^{\lambda_j}-t_{k}^{\lambda}\right)  
 +\left(\omega^{\lambda_j}\left(t_{k+1}^{\lambda_j}\right)  -\omega^{\lambda_j}\left(t_{k}^{\lambda_j}\right)
\right) 
 &\text{if }t_{k}^{\lambda_j} \in Z_j^{2},
 \\
\left(t_{k+1}^{\lambda_j}-t_{k}
^{\lambda_j}\right)  &\text{if }t_{k}^{\lambda_j}\notin Z_j^{2},
 \end{cases}
\\
\psi_j^{2}\left(0\right)  & :=0.
\end{align*}
Iterating, we exhaust all parabolic points for $m=1,2,\ldots$, with the
corresponding parabolic limits $v^{(r)}$, dissipations $d^{(r)}$, and the stretching function $\psi_j^{m}$ at level $m$.

In order to avoid overlap of the parabolic jump intervals, it is convenient to choose $L_j$ and the values of
$\lambda_j$ in such a way that the total measure of the intervals around the
jump points converge to zero. So, we additionally require
\[
\lambda_j\sum_{k=1}^{\infty}L_j^{(k)}\to
0  \qquad\text{as $j \to \infty$,}
\]
where $L_j^{(k)}$ is the $L_j$ corresponding to the $k$'th parabolic point. We furthermore need that $L_j^{(k)}\to\infty$ for each
fixed $k$. We can take for instance $L_j^{(k)}\leq 2^{-k}(\lambda_j)^{-1/2}$.

Notice furthermore that by the finite dissipation of energy and the assumed bounds on $f$ it holds that
\[
\sum_{r=1}^\infty d^{(r)}<\infty
\]
and
\begin{align}
\label{eq:inverse}
\omega^{\lambda_j}\left (\phi_j^{m}(s^{\lambda_j}_{k+1})\right)-\omega^{\lambda_j}\left (\phi_j^{m}(s^{\lambda_j}_{k})\right)\leq s^{\lambda_j}_{k+1}-s^{\lambda_j}_{k}.
\end{align}

\subsection{The canonical slow time scale}
\label{ssc:s}

In this section we construct a new time scale in which the rate-independent dissipation (in the jumps) takes place in times of order one.

Notice that by construction we have the following monotonicity:
\[
  \psi_j^{m}(t) \geq \psi_j^{m+1}(t),  \qquad t\in [0,T]
\]
due to the fact that the sets $Z_j^{m}$ are decreasing in $m$.

We want to take the limit as $j\to\infty$ and $m\to\infty$ at the same time. It is convenient to
work with the inverse functions
\[
  \varphi_j^{m}:=\left(\psi_j^{m}\right)  ^{-1}.
\]
These inverses always exist, because the
functions $\psi_j^{m}$ are strictly increasing.
Due to the fact that $\frac{\di}{\di t} \psi_j^{m} \geq 1$ we find that the $\varphi_j^{m}$ are uniformly (in $j$ and $m$) Lipschitz. 
Moreover, we have, using the monotonicity above,
\[
\varphi_j^{m}(s) \leq \varphi_j^{m+1}(s).
\]
Observe that by the construction of the piecewise affine functions $\psi_j^{m}, \varphi_j^{m}$ we can estimate their Lipschitz constants. Indeed, there is an $i^{(m)}\in \Nbb$ such that for $t_{k}^{\lambda_j}\in Z_j^{1}$, and $j\in \{i^{(m)},i^{(m)}+1,i^{(m)}+2,...\}$
\[
\omega^{\lambda_j}\left(t_{k+1}^{\lambda_j}\right)  -\omega^{\lambda_j}\left(t_{k}^{\lambda_j}\right)\leq \frac{2}{m},
\]
since otherwise one could construct another parabolic sequence that is associated to a parabolic point with jump $\geq \frac{1}{m}$. Hence we find analogously to~\eqref{eq:lipschitz} that
for $j\in \{i^{(m)},i^{(m)}+1,i^{(m)}+2,...\}$,
\begin{align}
\label{eq:lipschitz1}
1\leq  \partial_{t}\psi_j^{m} \leq \frac{3}{m\lambda_j}  \qquad\text{and}\qquad
\frac{m\lambda_j}{3}\leq \partial_{s}\varphi_j^{m}  \leq1.
\end{align}
Since the functions $\varphi_j^{m}$ are uniformly Lipschitz, they converge uniformly as $j\to\infty$ for each fixed $m$. On the
other hand, given that $\partial_{t}\psi_j^{m}\geq1$ we find
$\varphi_j^{m}(s)  \leq s$ for $s\geq0$. Since we know that the domain of definition of $\varphi_j^{m}$ is uniformly bounded by the total energy dissipation plus $T$, we can extend these functions to a common interval of finite length. 
Then, if we first take
the limit $j\to\infty$ and then $m\to\infty$ we obtain, using the Arzela--Ascoli compactness theorem, the
convergence to some limit function $\varphi$ uniformly in the interval
$s\in\left[  0,S_{0}\right]$, where $S_0$ is taken as the smallest number such that $\varphi(S_{0})=T$. Observe that $S_{0}$ is $T$ plus the dissipated energy that is not captured by the parabolic resolution functions $v^{(\ell)}$.

 Using a diagonal argument, we obtain the existence of a subsequence
$m_j\to\infty$ as $j\to\infty$ such that
\[
\varphi_j^{m_j}\to\varphi \quad\text{uniformly in $[0,S_{0}]$.}
\]
In fact, $\norm{ \varphi_j^{m_j}-\varphi} _{L^{\infty}([0,S_{0,j}])}\to0$ as $j\to\infty,$ where $S_{0,j}$ is the interval of definition for $\varphi_j^{m_j}$.

In the following we study the properties of the processes
\[
\widetilde{\omega}_j(s)  :=\omega^{\lambda_j}\left(\varphi_j^{m_j
}(s)  \right),  \qquad s \in [0,S_{0}].
\]
Our construction implies that if $t_i^{\lambda_j}\leq t_1\leq t_2\leq t_k^{\lambda_j}$ and $t_{j}^{\lambda_j}\not\in Z^{m_j}_j$ for all $j\in \{i,..,k\}$, then for $a:=\psi(t_1)$ and $b:=\psi(t_2)$ we have by~\eqref{eq:inverse} that
\begin{align}
\label{eq:etildalambda}
\absb{\widetilde{\omega}_j(a)-\widetilde{\omega}_j(b)}\leq b-a .
\end{align}
The functions $\widetilde{\omega}_j$ contain the information about the dissipation of energy at the rate-dependent jumps. 
 More precisely, we have for yet another subsequence of the $\lambda_j$'s that for some function $\widetilde{\omega} \colon [a,b] \to \R$,
 \[
\widetilde{\omega}_j \to\widetilde{\omega} \quad\text{a.e.\ in $[a,b]$} \qquad\text{as $j\to\infty$.}
\]
Moreover, the convergence of the
functions $\widetilde{\omega}_j$ is uniform due to~\eqref{eq:etildalambda} except at a countable set of points, which are characterized in the following way: We recall that we have associated to each rate-dependent jump
a time $t^{(\ell)}$. We recall that there is a sequence $t_{k_j}^{\lambda_j}\to t^{(\ell)}$ (the dependence on $\ell$ will be suppressed in the following) and set
\[
s_j=s_j^{(\ell)} := \psi_j^{m_j}\left(t_{k_j}^{\lambda_j}\right).
\]
We define
\[
  s^{(\ell)}:=\liminf_{j\to\infty}\psi_j^{m_j}\left(t_{k_j}^{\lambda_j}\right).
\]
By taking another subsequence, we can ensure that $s_j^{(\ell)}\to s^{(\ell)}$ as $j\to \infty$. We claim that the limit values arising in this way are precisely the rate-dependent jump times in the new time $s$.

To see the claim, we first remark that these points are jump points of $\widetilde{\omega}$. We will further show that these jump points are the only jump points of $\widetilde{\omega}$ and that all these jumps are of a rate-dependent (inertial/parabolic) nature. Indeed, the times $s^{(\ell)}$ are the only points where we have discontinuities of the
limit energy $\widetilde{\omega}$ due to the construction of the functions $\psi
_j^{m_j}$, which guarantees that away from the rate-dependent jumps the
functions $\widetilde{\omega}_j$ are uniformly continuous by~\eqref{eq:etildalambda}. Of course, the jumps may occur in a dense subset.

It is crucial to observe that we can have several $s^{(\ell)}$ taking the same value. This means that there are several rate-dependent jumps taking place around $s^{(\ell)}$. Then, of course also the old time is the same for all these jumps.

On the other hand, we might have rate-dependent jumps at different values of $s$, but associated to the same value of $t$. Notice that the times $t$ are given in terms of $s$ by means of the function $t=\varphi(s)$ and $\varphi$ can have flat regions (a countable number of them). This would represent having, at a given time $t$, intertwined rate-dependent (inertial) and rate-independent (slide) time intervals.

\subsection{Convergence in the canonical slow time scale}

In this subsection we pass to the limit in $j$ with
\[
  \wlj(s):=u_{\lambda_j}(\varphi_j^{m_j}(s)),  \qquad s \in [0,S_{0,j}].
\]
Rewriting the equation~\eqref{eq:PDE_lambda} using
the new variable $s=\psi^{m_j}_j(t)$, we obtain
\begin{align} \label{A3}
  \left\{
  \begin{aligned}
    -\frac{\lambda_j}{\partial_{s}\varphi_j^{m_j}} \partial_{s}\wlj +  \Delta \wlj - \DD W_0(\wlj)+\tilde{f}_j &\in  \partial \Rcal_1(\partial_s \wlj)
      &&\text{in $(0,S_{0,j}) \times \Omega$,}\\
    \wlj(s)|_{\partial \Omega} &= 0  &&\text{for $s \in (0,S_{0,j}]$,} \\
    \qquad \wlj(0) &= u_0  &&\text{in $\Omega$},
  \end{aligned} \right.
\end{align}
where
\[
  \tilde{f}_j(s):=f(\varphi_j^{m_j}(s)).
\]
We may pass to the limit $j \to \infty$ ($\lambda_j \todown 0$) in the same fashion as in Lemma~\ref{lem:energ2}, using also the fact that the $\varphi^{m_j}_j$ are uniformly Lipschitz continuous and uniformly convergent, i.e.\
\[
  \tilde{f}_j\to f\circ \varphi =:\tilde{f} \quad\text{uniformly.}
\]
We then get analogously to~\eqref{eq:l_conv} that
\begin{equation} \label{eq:converg}
\left\{
\begin{aligned}
\wlj &\to \w  \quad\text{in } \Lrm^a_\loc((0,\infty) \times \Omega;\R^m) \quad\text{for $a\in [1,\infty)$;}\\
\wlj &\to \w  \quad\text{in } \Lrm^a_\loc(0,\infty;\Wrm^{1,r}_0(\Omega;\R^m)) \quad\text{for $a\in [1,\infty)$, $r\in [1,2^*)$;}\\
\wlj &\toweak  \w  \quad\text{in }\Lrm^2_\loc(0,\infty;\Wrm^{2,r}(\Omega;\R^m))\quad\text{ for $r\in [1,2^*)$;}\\
\wlj &\toweakstar \w  \quad\text{in }\Lrm^\infty(0,\infty;(\Wrm^{1,2} \cap \Lrm^q)(\Omega;\R^m));\\
\wlj &\toweakstar \w  \quad\text{in }\Lrm^\infty_\loc(0,\infty;\Wrm^{2,2}(\Omega;\R^m));\\
\wlj &\toweakstar \w \quad\text{in }\BV(0,\infty;\Lrm^1(\Omega;\R^m))).  
\end{aligned}
\right.
\end{equation}
We further define
\begin{align*}
\widetilde{E}_j(s) &:= \Ecal_j(s,\wlj(s)):=\Wcal(\wlj(s))- \dprb{\tilde{f}_j(s),\wlj(s)}, \\
\widetilde{E}(s) &:= \Wcal(\w(s))- \dprb{\tilde{f}(s),\w(s)}.
\end{align*}
By the above a-priori estimates we obtain
\[
\widetilde{E}_j \to \widetilde{E} \quad\text{almost everywhere}.
\]

The weak lower semicontinuity of the variation and the convergence of the energy for almost every $[a,b] \subset [0,S_0]$ such that $a,b$ are not jump points, then imply
\begin{align}
\label{eq:energyineq}
\Ecal(b,\w(b))-\Ecal(a,\w(a)) \leq - \Var_{\Rcal_1}(\w;[a,b]) -\int_a^b \dprb{\partial_s \tilde{f}(s),\w(s)} \dd s. 
\end{align}
Moreover, we can refine the above energy inequality by removing an arbitrary number of jumps and replacing them via the parabolic resolutions. Around a parabolic point $s^{(\ell)} \in (a,b)$ in the new time scale, or $t^{(\ell)}$ in the old time scale, we argue as follows: Set
\[
  s_j^{(\ell)-} := \psi_j^{m_j}\left(t_{k_j}^{\lambda_j} - \lambda_j L_j\right)  \qquad\text{and}\qquad
  s_j^{(\ell)+} := \psi_j^{m_j}\left(t_{k_j}^{\lambda_j} + \lambda_j L_j\right)
\]
Then, via~\eqref{A1}, we get
\begin{align*}
\widetilde{\omega}_j \left(s_j^{(\ell)-}\right) - \widetilde{\omega}_j\left(s_j^{(\ell)+}\right)  &=  \omega^{\lambda_j}\left(t_{k_j}^{\lambda_j} + \lambda_j L_j\right) -\omega^{\lambda_j}\left(t_{k_j}^{\lambda_j} - \lambda_j L_j\right)\\
&\to d^{(r)}  \qquad\text{as $j\to\infty$.}
\end{align*}
Since
\begin{align}
  &\widetilde{\omega}_j\left(s_j^{(\ell)+}\right) - \widetilde{\omega}_j \left(s_j^{(\ell)-}\right)  \notag\\
  &\quad = \Ecal \left( t_{k_j}^{\lambda_j} - \lambda_j L_j, u_{\lambda_j}(t_{k_j}^{\lambda_j} - \lambda_j L_j) \right) - \Ecal \left( t_{k_j}^{\lambda_j} + \lambda_j L_j, u_{\lambda_j}(t_{k_j}^{\lambda_j} + \lambda_j L_j) \right)  \notag\\
  &\quad\qquad - \int_{t_{k_j}^{\lambda_j} - \lambda_j L_j}^{t_{k_j}^{\lambda_j} + \lambda_j L_j} \dprb{\dot{f}(\tau),u_{\lambda_j}(\tau)} \dd \tau  \notag\\
  &\quad = \widetilde{E}_j(s_j^{(\ell)-}) - \widetilde{E}_j(s_j^{(\ell)+}) - \int_{t_{k_j}^{\lambda_j} - \lambda_j L_j}^{t_{k_j}^{\lambda_j} + \lambda_j L_j} \dprb{\dot{f}(\tau),u_{\lambda_j}(\tau)} \dd \tau,  \label{eq:omegatilde_est}
\end{align}
we get
\begin{align*}
\widetilde{E}(b) -\widetilde{E}(a)
&= \lim_{j\to\infty} \biggl( \widetilde{E}_j(b) - \widetilde{E}_j(s_j^{(\ell)+}) + \widetilde{E}_j(s_j^{(\ell)-}) -\widetilde{E}(a) - \Bigl[ \widetilde{\omega}_j\left(s_j^{(\ell)+}\right) - \widetilde{\omega}_j \left(s_j^{(\ell)-}\right) \Bigr] \biggr)\\
&\leq - \Var_{\Rcal_1}(\w;[a,b]\setminus \{s^{(\ell)}\}) - d^{(\ell)} -\int_a^b \dprb{\partial_s \tilde{f}(s),\w(s)} \dd s. 
\end{align*}
Setting $D_K:=\{s^{(l_1)},\cdots,s^{(l_K)}\}$, we find by iteration of the above argument that
\[
\widetilde{E}(b) -\widetilde{E}(a)
\leq - \Var_{\Rcal_1}(\w;(a,b)\setminus D_K) - \sum_{i \;:\; s^{(l_i)}\in [a,b]}d^{(l_i)}
 -\int_a^b \dprb{\partial_s \tilde{f}(s),\w(s)} \dd s. 
\]

\subsection{Stability}

 We next consider the question of stability. 
\begin{proposition}
\label{prop:stab}
For almost every $s \in [0,S_0]$ the stability is satisfied in the new time variable $s$, i.e.\
\begin{align}
\label{eq:stabw}
  \dprb{- \DD W_0(\w(s)) + \tilde{f}(s), \psi} - \dprb{\nabla \w(s), \nabla \psi} 
  \leq \Rcal_1(\psi)
\end{align}
for all $\psi\in \Wrm^{1,2}_0(\Omega;\R^m)$. 
Moreover, for all $r\in [1,\infty)$,
\begin{equation} \label{eq:utilde_reg}
\norm{\w}_{\Lrm^\infty(0,S_0;\Wrm^{2,r}(\Omega;\R^m))}\leq C.
\end{equation}
\end{proposition}

\begin{proof}
From~\eqref{A3} we know
\begin{align*}
  &\Rcal_1(\partial_s\wlj (s)) 
 + \dprBB{-\frac{\lambda_j}{\partial_{s}\varphi_j^{m_j}(s)}\wlj(s) + \Delta \wlj(s) - \DD W_0(\wlj(s) ) + \tilde{f}_j(s), \xi(s) - \partial_s\wlj(s))} 
\\
&\quad  \leq \Rcal_1(\xi(s)) 
\end{align*}
for almost all $s \in [0,S_{0,j}]$ and all $\xi \in \Lrm^1(0,S_{0,j};\Wrm^{1,2}_0(\Omega;\R^m))$. Setting $\xi := \partial_s \wlj + \psi$ with $\psi \in \Wrm^{1,2}_0(\Omega;\R^m)$ and using the subadditivity of $\Rcal_1$, we find
\begin{equation} \label{eq:stabapprox}
  \dprBB{- \frac{\lambda_j}{\partial_{s}\varphi_j^{m_j}(s)} \partial_s\wlj(s) - \DD W_0(\wlj(s) ) + \tilde{f}_j(s), \psi} - 
\dprb{\nabla \wlj(s), \nabla \psi} 
  \leq \Rcal_1(\psi).
\end{equation}
Moreover, we have the following energy equality for $0\leq a<b\leq S_0$
\begin{align}
\label{eq:tildeE}
\int_a^{b}\Rcal_1(\partial_s\wlj (s))+\frac{\lambda_j}{\partial_{s}\varphi_j^{m_j}(s)}\frac{\norm{\partial_{s}\wlj(s)}_{\Lrm^2}^2}{2}\dd s = \widetilde{E}_j(a)-\widetilde{E}_j(b).
\end{align}
Hence, the uniform bounds induced by the energy~\eqref{eq:tildeE}  (note that $\widetilde{E}_j$ is continuous on $[0,S_0]$) together with~\eqref{eq:lipschitz1} imply that
\[
\int_0^{S_0}\biggl(\frac{\lambda_j}{\partial_{s}\varphi_j^{m_j}(s)}\norm{\partial_{s}\wlj(s)}_{\Lrm^2}\biggr)^2\dd s
\leq C \int_0^{S_0} \frac{\lambda_j}{\partial_{s}\varphi_j^{m_j}(s)} \dd s
\leq \frac{C}{m_j}.
\]
Therefore, this term tends to zero in $\Lrm^2(0,S_0;\Lrm^2(\Omega;\R^m))$ since $m_j\to\infty$.

We have the following convergences (see the beginning of this section):
\begin{align*}
  \wlj(s)  &\to \w(s) &&\text{in $\Lrm^q(\Omega;\R^m)$ for a.e.\ $s \in [0,S_0]$,} &&&&& \\
 \frac{\lambda_j}{\partial_{s}\varphi_j^{m_j}}\partial_s\wlj &\to 0 &&\text{in $\Lrm^2(0,S_0;\Lrm^2(\Omega;\R^m))$,} \\
  \nabla \wlj(s)  &\toweak \nabla \w(s) &&\text{in $\Lrm^2(\Omega;\R^m)$ for a.e.\ $s \in [0,S_0]$.}
\end{align*}
Hence, we can pass to the limit in~\eqref{eq:stabapprox} to get~\eqref{eq:stabw}.

Observe that the stability~\eqref{eq:stabw} implies by the dual characterization of $\Lrm^\infty$ that 
\[
\norm{-\Delta \w +\DD W_0(\w)}_{\Lrm^\infty([0,S_0]\times \Omega)}\leq C.
\]
Hence, Lemma~\ref{lem:apriori-space} implies the desired regularity estimate~\eqref{eq:utilde_reg}.
\end{proof}

\subsection{Energy equality}

We are now finally ready to completely resolve the energy dissipation.

\begin{proposition}\label{prop:energy}
Let $[a,b]\subset [0,S_0]$. Then,
\[
\widetilde{E}(b-) - \widetilde{E}(a+)
= - \int_a^b \Rcal_1(\partial_s \w(s))\dd s - \sum_{\ell \;:\; s^{(\ell)}\in [a,b]} d^{(\ell)} -\int_a^b\dprb{\partial_s\tilde{f}(s),\w(s)}\dd s,
\]
where $\partial_s \w$ is the absolutely continuous part of the time derivative of $\w$.
\end{proposition}
\begin{proof}
First assume that $a\neq 0$.
Take $m\in \N$. Following the construction of the timescale $s$ in Section~\ref{ssc:s}, we consider all sequences $\{t_{k_j}^{\lambda_j}\}_j$, which are related to $m$. There are only a finite number of such sequences, since each of them yield a dissipation (a negative jump in $\widetilde{\omega}_j$) of size at least $1/m$. For all $j$ such that $m_j\geq m$ we can then define
\[
  s_{t_j}^{\lambda_j}:=\phi_j^{m_j}(t_{k_j}^{\lambda_j}).
\]
From 
Section~\ref{ssc:s} we know that these sequences converge to some values $s^{(1)}\leq s^{(2)}\leq\cdots\leq s^{(K_m)}$, with $K_m\leq m \widetilde{\omega}(S_0)$. The construction implies that $\widetilde{\omega}$ has no jump larger than $1/m$ in $[0,S_0]\setminus \{s^{(l)}\}_{l=1}^{K_m}$. This entails that also $\widetilde{E}$ has no jump larger than $1/m$ on $[0,S_0]\setminus \{s^{(l)}\}_{l=1}^{K_m}$.

Let now $[a,b] \subset [0,S_0]$. We re-index the collection $s^{(\ell)}$ as $s^{(l_i)}$ with $i = 1,\cdots, K_m$ such that $a\leq s^{(l_1)}\leq s^{(l_2)}\leq \cdots\leq s^{(l_{K_m})}\leq b$. The respective dissipations are denoted by $d^{(l_i)}$ and can be computed via~\eqref{eq:v_energy},~\eqref{eq:omegatilde_est} as
\[
d^{(l_i)} = \widetilde{E}(s^{(l_i)}-) - \widetilde{E}(s^{(l_i)}+)
= \int_{-\infty}^{\infty} \norm{\partial_\theta v^{(l_i)}}_{\Lrm^2}^2 + \Rcal_1(\partial_\theta v^{(l_i)}) \dd \theta,
\]
where $\widetilde{E}(s^{(l_i)}\pm)$ denote the right and left limits of $\widetilde{E}$ at $s^{(l_i)}$, respectively.

Without loss of generality we assume that $a<s^{(l_1)}\leq s^{(l_{K_m})}< b$ (otherwise, at the endpoints, in the following some terms do not appear). Then,
\begin{align*}
\widetilde{E}(a)-\widetilde{E}(b)
&=\widetilde{E}(a)-\widetilde{E}(s^{(l_1)}-) +
\sum_{i=1}^{K_m}\Bigl( \widetilde{E}(s^{(l_i)}-) - \widetilde{E}(s^{(l_i)}+) \Bigr) \\
&\quad +\sum_{i=1}^{K_m-1} \Bigl( \widetilde{E}(s^{(l_i)}+) - \widetilde{E}(s^{(l_{i+1})}-) \Bigr) + \widetilde{E}(s^{(l_{K_m})}+)- \widetilde{E}(b)\\
&= \widetilde{E}(a)-\widetilde{E}(s^{(l_1)}-)
+\sum_{l=1}^{K_m}d^{(l_i)} \\
&\quad +\sum_{i=1}^{K_m-1} \Bigl( \widetilde{E}(s^{(l_i)}+) - \widetilde{E}(s^{(l_{i+1})}-) \Bigr) + \widetilde{E}(s^{(l_{K_m})}+)- \widetilde{E}(b).
\end{align*}
Since $\Ecal(\frarg,\w(\frarg))$ has no jump larger than $1/m$ on $[0,S_0]\setminus \{s^{(l)}\}_{l=1}^{K_m}$, Proposition~\ref{prop:E_contpoint} implies
\begin{align*}
&\absBB{\widetilde{E}(s^{(l_i)}+) - \widetilde{E}(s^{(l_{i+1})}-) - \Var_{\Rcal_1}\bigl(\w;(s^{(l_i)},s^{(l_{i+1})})\bigr) - \int_{{s^{(l_i)}}}^{s^{(l_{i+1})}}\dprb{\partial_s\tilde{f}(s),\w(s)}\dd s}\\
&\quad \leq \frac{C}{m} \Var_{\Rcal_1}\bigl(\w;(s^{(l_i)},s^{(l_{i+1})})\bigr).
\end{align*}
This also holds with $s^{(0)} := a$, $s^{(l_{K_m+1})} = b$. So, setting $D_m:=\{s^{(l)}\}_{l=1}^{K_m}$,
\begin{align*}
&\absBB{\widetilde{E}(a)-\widetilde{E}(b) - \Var_{\Rcal_1}\bigl(\w;[a,b] \setminus D_m \bigr) - \sum_{l=1}^{K_m}d^{(l_i)} - \int_{a}^{b}\dprb{\partial_s\tilde{f}(s),\w(s)}\dd s} \\
&\quad \leq \frac{C}{m} \Var_{\Rcal_1}\bigl(\w;[a,b] \setminus D_m\bigr).
\end{align*}
Letting $m\to \infty$, we find the assertion. Indeed,
\[
  \sum_{l=1}^{K_m}d^{(l_i)} \to \sum_{\ell \;:\; s^{(\ell)}\in [a,b]} d^{(\ell)}
\]
by monotone convergence, and
\[
  \Var_{\Rcal_1}\bigl(\w;[a,b] \setminus D_m\bigr)
  \to \int_a^b \Rcal_1(\partial_s \w(s))\dd s
\]
since we know that $\w$ is continuous on all of $[a,b]$ except for the points $\{s^{(\ell)}\}$.

For the initial value we also need to consider the case $s^{(1)}=a$. Then there are two possibilities. Either in the respective sequence $\{t^{\lambda_j}_{k_j}\}_j$, the $k_j$ are uniformly bounded. In this case the limit function defined as 
$v^1 \colon [0,\infty)\times \Omega\to \R^m$ satisfies
\begin{align*}
 \Ecal(0,u(0)+) &= \Ecal(0,u_0)-\Diss_{\mathrm{jump}}(0), 
  \\
  \Diss_{\mathrm{jump}}(0) &:= \int_{0}^\infty \norm{\partial_\theta v^1(0,\theta)}_{\Lrm^2}^2 + \Rcal_1(\partial_\theta v^1(0,\theta)) \dd \theta
  \\
  &\quad  + \sum_{j \in \Ibb_0 \setminus\{1\}} \biggl(\int_{-\infty}^\infty \norm{\partial_\theta v^j(t_k,\theta)}_{\Lrm^2}^2 + \Rcal_1(\partial_\theta v^j(t_k,\theta))\dd \theta \biggr).
  \end{align*}
In particular,
\[
d^{1}=\int_{0}^\infty \norm{\partial_\theta v^1(0,\theta)}_{\Lrm^2}^2 + \Rcal_1(\partial_\theta v^1(0,\theta)) \dd \theta.
\]
For all other $\ell$ we have $v^{(\ell)} \colon (-\infty,\infty)\times \Omega\to \R^m$.
\end{proof}

Notice that the limit functions $\widetilde{\omega}(s)$, $\varphi(s)$ as well as the dissipation values $d^{(\ell)}$ yield all the information on how the dissipation is taking place.

The rate-dependent jumps of $\widetilde{\omega}(s)$ take place at at most a countable number of (new) times $s^{(\ell)}$. As mentioned before, we can have
several repeated values $s^{(\ell)}$ (they do not appear consecutively in the sequence, because the jumps are ordered to the magnitude of the energy jump, not the times). Therefore, a given point $s^{\ast}$
could appear infinitely often in the $s^{(\ell)}$.

The original times at which these parabolic times take place are the times $t^{(\ell)}=\varphi(s^{(\ell)})$. Given that the function $\varphi$ might have plateaus, there could be several jumps associated to an original time point $t$, with rate-independent regions in between.

\subsection{Continuity over the jumps}

Using the derivation above we define for a jump point $s_k\in [0, S_0]$ of $\w$ the set
\begin{equation} \label{eq:Ik}
  \Ibb_k:=\setb{\ell\in \Nbb}{s^{(\ell)}=s_k}.
\end{equation}
\begin{proposition}
\label{pro:continuity}
For $i,\ell\in \Ibb_k$ there is a constant $C > 0$ just depending on $R_1$ such that
\begin{enumerate}[a)]
\item If $
0\leq \Ecal(s_k,v^i(\infty))- \Ecal(s_k,v^\ell(-\infty))=\delta
$, then 
$
\norm{v^i(\infty)-v^\ell(-\infty)}_{\Lrm^1}\leq C\delta.
$
\item If $
0\leq \widetilde{E}(s_k-) - \Ecal(s_k,v^\ell(-\infty))=\delta
$, then  
$
\norm{\w(s_k-)-v^\ell(-\infty)}_{\Lrm^1}\leq C\delta.
$
\item If $
0\leq  \Ecal(s_k,v^\ell(\infty))- \widetilde{E}(s_k+) =\delta
$, then 
$
\norm{v^\ell(\infty)-\w(s_k+)}_{\Lrm^1}\leq C\delta.
$
\end{enumerate}
Moreover:
\begin{enumerate}[a)]
\item[d)] For every $\epsilon>0$ there exists an $i\in \Ibb_k$ such that
$\norm{u(s_k-)-v^{i}(-\infty)}_{\Lrm^1}\leq \epsilon$.
\item[e)] For every $\epsilon>0$ and every $\ell\in \Ibb_k$ there exists $i\in \Ibb$ such that
$\norm{v^{\ell}(\infty)-v^{i}(-\infty)}_{\Lrm^1}\leq \epsilon$.
\item[f)] For every $\epsilon>0$ there exists an $i\in \Ibb_k$ such that
$\norm{v^{i}(\infty)-u(s_k+)}_{\Lrm^1}\leq \epsilon$.
\end{enumerate}
\end{proposition}
\begin{proof}
For a) we find by the construction above an approximation $L_j^i,L_j^\ell,\lambda_j,s_j^i,s_j^\ell$ such that as $j\to \infty$:
\begin{enumerate}
\item $L_j^i\to \infty$ and $L_j^\ell\to \infty$ with $\lambda_j \cdot \max\{L_j^i,L_j^\ell\}\to 0$;
\item $s_j^i\to s_k$
 and $s_j^\ell\to s_k$ with $s_j^i+L_j^i\lambda_j\leq s_j^\ell-L_j^\ell\lambda_j$;
 \item
$
u_{\lambda_j}(s_j^i+L_j^i\lambda_j)=v^i_{\lambda_j}(L_j^i)\to v^i(\infty) 
$
and 
$
u_{\lambda_j}(s_j^\ell-L_j^i\lambda_j)=v^\ell_{\lambda_j}(-L_j^\ell)\to v^\ell(-\infty) 
$ strongly in $\Lrm^1$;
\item $\Ecal(s_j^i+L_j^i\lambda_j, u_{\lambda_j}(s_j^i+L_j^i\lambda_j))\to \Ecal(s_k,v^i(\infty))$ and $\Ecal(s_j^\ell-L_j^\ell\lambda_j, u_{\lambda_j}(s_j^\ell-L_j^\ell\lambda_j))\to \Ecal(s_k,v^\ell(-\infty))$.
\end{enumerate}
This is possible since in the construction of the $v^i$, we ordered them by the size of their energy contribution such that any ''overlap'' is excluded. Then, by the energy equality~\eqref{eq:ulambda_energy}, we can estimate
\begin{align*}
&\norm{u_{\lambda_j}(s_j^i+L_j^i\lambda_j)-u_{\lambda_j}(s_j^\ell-L_j^\ell\lambda_j)}_{\Lrm^1} \\
&\qquad \leq C\int_{s_j^i+L_j^i\lambda_j}^{s_j^\ell-L_j^\ell\lambda_j}\Rcal(\dot{u}(\tau)) \dd \tau 
\\
&\qquad \leq C\Big(\Ecal(s_j^i+L_j^i\lambda_j, u_{\lambda_j}(s_j^i+L_j^i\lambda_j))-\Ecal(s_j^\ell-L_j^\ell\lambda_j, u_{\lambda_j}(s_j^\ell-L_j^\ell\lambda_j))\Big)
\\
&\qquad \qquad  +C\int_{s_j^i+L_j^i\lambda_j}^{s_j^\ell-L_j^\ell\lambda_j}\norm{\dot{f}(\tau)}_{\Lrm^\infty} \cdot\norm{u_{\lambda_j}(\tau)}_{\Lrm^1} \dd \tau.
\end{align*}
This implies the result by passing to the limit.

The proofs for b) and c) are analogous. Indeed, for b) we find approximating sequences $L_j^\ell,\lambda_j,s_j^\ell,s_j$, such that with $j\to \infty$ 
\begin{enumerate}
\item $L_j^\ell\to \infty$ with $\lambda_jL_j^\ell\to 0$;
\item $s_j^\ell\to s_k$ and $s_j\to s_k$ with $s_j\leq s_j^\ell-L_j^\ell\lambda_j$;
 \item
$
u_{\lambda_j}(s_j)\to \w(s_k-)
$
and 
$
u_{\lambda_j}(s_j^\ell-L_j^i\lambda_j)=v^\ell_{\lambda_j}(-L_j^\ell)\to v^\ell(-\infty) 
$ strongly in $\Lrm^1$;
\item $\Ecal(s_j, u_{\lambda_j}(s_j))\to \widetilde{E}(s_k-)$ and $\Ecal(s_j^\ell-L_j^\ell\lambda_j, u_{\lambda_j}(s_j^\ell-L_j^\ell\lambda_j))\to \Ecal(s_k,v^\ell(-\infty))$.
\end{enumerate}
Now, the estimate b) in the same fashion as a).

For d) observe that due to the energy equality and the monotonicity of the energy the family $\{\Ecal(s_k,v^i(-\infty))\}_{i\in \Ibb_k}\cup \{\Ecal(s_k,v^i(\infty))\}_{i\in \Ibb_k}$ can (iteratively) be completely ordered by size. Moreover, 
\[
\widetilde{E}(s_k-)=\inf_{i\in \Ibb_k} \Ecal(s_k,v^i(-\infty)).
\]
Hence, we find for every $\epsilon>0$ an $i\in \Ibb_k$ such that
\[
\widetilde{E}(s_k-)-\Ecal(s_k,v^i(-\infty)) \leq C\epsilon.
\]
Now b) implies
\[
\norm{\tilde{u}(s_k-))-v^{i}(-\infty)}\leq C\epsilon.
\]
Similarly, we find for e) that
 \[
 \Ecal(s_k,v^\ell(\infty))=\inf \, \setb{\Ecal(s_k,v^i(-\infty))}{ i\in \Ibb_k \text{ and } \Ecal(s_k,v^\ell(\infty))<\Ecal(s_k,v^i(-\infty))},
\]
which implies by a) that for every $\epsilon>0$ there exists an $i\in \Ibb_k$ with
\[
\norm{v^\ell(\infty)-v^{i}(-\infty)}\leq \epsilon.
\]
The assertion f) follows in the same way as d) (using c)).
\end{proof}

\subsection{Proofs of the main theorems} \label{ssc:proofs}

\begin{proof}[Proof of Theorem~\ref{thm:main_ex}]
All objects occurring in the claims of Theorem~\ref{thm:main_ex} have already been constructed. A simple change of variables allows to define
\begin{align*}
\widehat{\phi}(t):= \frac{T}{S_0}\phi(t),  \qquad
u(s):=\w \biggl( \frac{S_0}{T}s \biggr),  \qquad
\tilde{f}:=f\circ\widehat{\phi},
\end{align*}
where $\w$ is defined via~\eqref{eq:converg}. Hence, $\widehat{\phi} \colon [0,T]\to [0,T]$ and so the time domain for $u,\tilde{f}$ is $[0,T]$ as well.

We define $\Jbb$ as the set of jump points of $E(s) = \Ecal(s,u(s))$, which is equal to $\Jbb:=\{s^{(\ell)}\}_{\ell}$ (see the end of Subsection~\ref{ssc:s}). For every $s_k\in \Jbb$, as in~\eqref{eq:Ik} we collect in $\Ibb_k$ all $\ell$, such that $s^{(\ell)}=s_k$ and set $v^i(s_k,\theta):=v^{(\ell)}(\theta)$.

We also define
\[
  z_\lambda \in \Lrm^\infty(0,\infty;(\Wrm^{1,2} \cap \Lrm^q)(\Omega;\R^m) \cap \Crm^{0,1}(0,\infty;\Lrm^1(\Omega;\R^m))
\]
to be an $\Lrm^1$-arclength-parametrization of $\tilde{u}_{\lambda_j}$ (or of $u_{\lambda_j}$, as these two processes are just reparametrizations of each other), extended constantly to the right. By Lemma~\ref{lem:bv-comp}~(iii) with $s = 2$, $\rho = 2$, $k = 2$, $m = 1$, and $\alpha_2 \to \tilde{q} = 2^*$ ($\theta \to 1$) we then obtain the strong convergence of $z_\lambda$ to $z$ strongly in $\Lrm^a_\loc(0,\infty;\Wrm^{1,r}_0)$ for $a\in [1,\infty)$, $r\in [1,2^*)$ and weakly* in $\Crm^{0,1}(0,\infty;\Lrm^1)$. In particular, $z$, which we also restrict to $[0,L]$ for some $L$ such that to the right of $L$ the process $z$ is constant, is still Lipschitz with values in $\Lrm^1$.

In the following we will indicate how the results from this chapter imply that
\[
  u, \; \Jbb, \; \Ibb_k, \; \{v^i(s_k)\}_{i\in \Ibb_k,k\in \Jbb}
\]
indeed satisfy all requirements to be a two-speed solution. Indeed:

\begin{enumerate} [(I)]
\item follows from Propositon~\ref{prop:E_contpoint}, also see Corollary~\ref{cor:energ_cont};
\item follows from Proposition~\ref{prop:goodset} and the a-priori estimates in~\eqref{eq:converg};
\item follows from Proposition~\ref{prop:stab};
\item follows from Lemma~\ref{lem:energ2};
\item follows from Lemma~\ref{lem:energ2b} and the construction of the sets $Z^{m}_j$ in Subsection~\ref{ssc:parab} and Proposition~\ref{pro:continuity};
\item follows from Proposition~\ref{prop:energy};
\item follows from Proposition~\ref{prop:energy} for $a = 0$.
\item follows from Proposition~\ref{pro:continuity}.
\end{enumerate}
The regularity follows from the uniform estimates in Lemma~\ref{lem:apri}, see also Proposition~\ref{prop:apriori} and~\eqref{eq:converg}.
\end{proof}

\begin{proof}[Proof of Corollary~\ref{cor:main_ex}]
The proof of the corollary follows by defining $t=\psi(s)$, where $\psi$ is the (increasing) left inverse of $\widehat{\phi}$ defined above. We define $w(t):=u\circ\psi(t)$ and $\Jbb_w$ as the jump set of the respective energy.
The definition of the $v^i(t_k,\theta)=v^i(\psi(s_k),\theta)$ remains the same. Now, $(w,\Jbb_w,\{v^i(t_k)\})$ satisfy~\ref{itm:i}--\ref{itm:v}. In order to get an energy balance, we have to include the rate-independent evolutions $b^k$. We define them in the following chronological order. We take the smallest $t_k\in \Jbb$, such that $\psi$ has a jump at $t_k$ and set $s_{k*}:=\inf\set{s}{\psi(s)=t_k}$ and $s_k^*:=\sup\set{s}{\psi(s)=t_k}$. Define $b^k \colon [0,1]\times \Omega\to \R^m$ via
\[
  b^k(\tau,x):=\w\bigl(s_{k*}+(s_k^*-s_{k*})\tau \bigr),  \qquad (\tau,x) \in [0,1]\times \Omega,
\]
and the existence result is established for
\[
  w, \; \{b^k\}_{k\in \Jbb_w}, \; \{v^i(t_k)\}_{i\in \Ibb_k,k\in \Jbb_w}.
\]
The regularity follows from the uniform estimates in Lemma~\ref{lem:apri} and~\eqref{eq:converg}.
\end{proof}

\begin{proof}[Proof of Theorem~\ref{thm:main_approx}]
The proof of Theorem~\ref{thm:main_approx} can be collected from the estimates above.
\begin{enumerate}
\item[\eqref{eq:conv1}] is explained in~\eqref{eq:l_conv}.
\item[\eqref{eq:conv2}] is a consequence of the construction in Section~\ref{ssc:s}. Hereby the $\tau^{(k,i)}_j$ are defined via $t_k+\tau^{(k,i)}_j=t^{\lambda_j}_{k_j}$.
\item[\eqref{eq:conv3}] and~\eqref{eq:conv4} follows from the definition and~\ref{itm:vi} and~\ref{itm:vii} of Theorem~\ref{thm:main_ex}.
\end{enumerate}
The proof is complete.
\end{proof}



\providecommand{\bysame}{\leavevmode\hbox to3em{\hrulefill}\thinspace}
\providecommand{\MR}{\relax\ifhmode\unskip\space\fi MR }
\providecommand{\MRhref}[2]{%
  \href{http://www.ams.org/mathscinet-getitem?mr=#1}{#2}
}
\providecommand{\href}[2]{#2}

\end{document}